\documentclass[a4paper,12pt,centertags,oneside]{amsart}
\usepackage{amsmath,amstext,amsthm,a4,amssymb,amscd}
\usepackage[mathscr]{eucal}
\usepackage{mathrsfs}
\usepackage{bbold}
\usepackage{epsf}
\usepackage{typearea}
\usepackage{charter}
\usepackage{bm}

\usepackage{graphicx}
\usepackage[all]{xy}

\usepackage{hyperref}
\hypersetup{
    colorlinks = true,
    linkcolor = black,
    citecolor = black,
    urlcolor = black,
}

\newcommand{\Q}{\mathbb{Q}}
\newcommand{\R}{\mathbb{R}}
\newcommand{\C}{\mathbb{C}}
\newcommand{\Z}{\mathbb{Z}}
\newcommand{\N}{\mathbb{N}}

\newcommand{\CP}{\C\mathrm{P}}

\newcommand{\smooth}{{\mathscr{C}^\infty}}
\newcommand{\e}{\varepsilon}
\newcommand{\n}{\nabla}
\newcommand{\Sp}{\mathrm{Sp}}

\DeclareMathOperator{\Id}{Id}
\DeclareMathOperator{\tr}{Tr}

\newtheorem{prop}{Proposition}[section]
\newtheorem{thm}[prop]{Theorem}
\newtheorem{lemme}[prop]{Lemma}

\newtheorem{conj}[prop]{Conjecture}

\theoremstyle{definition}
\newtheorem{defn}[prop]{Definition}

\theoremstyle{remark}
\newtheorem{rem}[prop]{Remark}

\numberwithin{equation}{section}
\title[BCOV invariant of Calabi-Yau pairs]{BCOV invariant and blow-up}
\date{\today}
\author{Yeping ZHANG}
\address{School of Mathematics,
Korea Institute for Advanced Study,
Hoegiro 85, Dongdaemun-gu,
Seoul 02455, Korea}
\email{ypzhang@kias.re.kr}

\begin{document}

\begin{abstract}
Bershadsky, Cecotti, Ooguri and Vafa constructed a real valued invariant for Calabi-Yau manifolds,
which is now called the BCOV invariant.
In this paper,
we extend the BCOV invariant to such pairs $(X,D)$,
where $X$ is a compact K{\"a}hler manifold
and $D$ is a pluricanonical divisor on $X$ with simple normal crossing support.
We also study the behavior of the extended BCOV invariant under blow-ups.
The results in this paper lead to a joint work with Fu proving that
birational Calabi-Yau manifolds have the same BCOV invariant. 
\end{abstract}

\maketitle

\tableofcontents

\section{Introduction}
\label{ch0}

In this paper,
we consider a real valued invariant for Calabi-Yau manifolds equipped with Ricci flat metrics,
which is now called the BCOV torsion.
The BCOV torsion was introduced by Bershadsky, Cecotti, Ooguri and Vafa \cite{bcov,bcov2} as the stringy genus one partition function of $N=2$ superconformal field theory.
Their work extended the mirror symmetry conjecture of Candelas, de la Ossa, Green and Parkes \cite{cdgp}.
Fang and Lu \cite{fl} used BCOV torsion to study the moduli space of Calabi-Yau manifolds.

The BCOV torsion is an invariant on the B-side.
Its mirror on the A-side is conjecturally the genus one Gromov-Witten invariant.
Though genus $\geqslant 2$ Gromov-Witten invariants have been intensively studied recently,
there is no rigorously defined genus $\geqslant 2$ invariant on the B-side.

The BCOV invariant is a real valued invariant for Calabi-Yau manifolds,
which could be viewed as a normalization of the BCOV torsion.
Fang, Lu and Yoshikawa \cite{fly} constructed the BCOV invariant for Calabi-Yau threefolds
and established the asymptotics of the BCOV invariant (of Calabi-Yau threefolds) for one-parameter normal crossings degenerations.
They also confirmed the genus one mirror symmetry conjecture of Bershadsky, Cecotti, Ooguri and Vafa \cite{bcov,bcov2} for quintic threefolds. 

Eriksson, Freixas i Montplet and Mourougane \cite{efm} constructed the BCOV invariant for Calabi-Yau manifolds of arbitrary dimension
and established the asymptotics of the BCOV invariant for one-parameter normal crossings degenerations.
In another paper \cite{efm2},
they confirmed the (B-side) genus one mirror symmetry conjecture of Bershadsky, Cecotti, Ooguri and Vafa \cite{bcov,bcov2} for Calabi-Yau hypersurfaces of arbitrary dimension, 
which is compatible with the results of Zinger \cite{zin08,zin09} on the A-side.

For a Calabi-Yau manifold $X$,
we denote by $\tau(X)$
the logarithm of the BCOV invariant of $X$ defined in \cite{efm}.

Yoshikawa \cite[Conjecture 2.1]{yo06} conjectured that
for a pair of birational projective Calabi-Yau threefolds $(X,X')$,
we have $\tau(X') = \tau(X)$.
Eriksson, Freixas i Montplet and Mourougane \cite[Conjecture B]{efm} conjectured the following higher dimensional analogue.

\begin{conj}
\label{intro-conj}
For a pair of birational projective Calabi-Yau manifolds $(X,X')$,
we have
\begin{equation}
\tau(X') = \tau(X) \;.
\end{equation}
\end{conj}

Let $X$ and $X'$ be projective Calabi-Yau threefolds defined over a field $L$.
Let $T$ be a finite set of embeddings $L\hookrightarrow\C$.
For $\sigma\in T$,
we denote by $X_\sigma$ (resp. $X_\sigma'$) the base change of $X$ (resp. $X'$) to $\C$ via the embedding $\sigma$.
We denote by $D^b(X_\sigma)$ (resp. $D^b(X_\sigma')$) the bounded derived category of coherent sheaves on $X_\sigma$ (resp. $X_\sigma'$).
Maillot and R{\"o}ssler \cite[Theorem 1.1]{ma-ro} showed that
if one of the following conditions holds,
\begin{itemize}
\item[a)] there exists $\sigma\in T$ such that $X_\sigma$ and $X_\sigma'$ are birational,
\item[b)] there exists $\sigma\in T$ such that $D^b(X_\sigma)$ and $D^b(X_\sigma')$ are equivalent,
\end{itemize}
then there exist a positive integer $n$ and a non-zero element $\alpha\in L$ such that
\begin{equation}
\tau(X_\sigma') - \tau(X_\sigma) = \frac{1}{n} \log \big|\sigma(\alpha)\big|
\hspace{4mm} \text{for } \sigma\in T \;.
\end{equation}
While a result of Bridgeland \cite[Theorem 1.1]{brid} showed that a) implies b),
Maillot and R{\"o}ssler gave separate proofs for a) and b).

Let $X$ be a Calabi-Yau threefold.
Let $Z \hookrightarrow X$ be a $(-1,-1)$-curve.
Let $X'$ be the Atiyah flop of $X$ along $Z$,
which is also a Calabi-Yau threefold.
We assume that both $X$ and $X'$ are compact and K{\"a}hler.
The author \cite[Corollary 0.5]{z} showed that
\begin{equation}
\label{introeq-z}
\tau(X') = \tau(X) \;.
\end{equation}
In other words,
Conjecture \ref{intro-conj} holds for $3$-dimensional Atiyah flops.
The proof of \eqref{introeq-z} consists of two key ingredients:
\begin{itemize}
\item[-] we extend the BCOV invariant from Calabi-Yau manifolds to certain `Calabi-Yau pairs',
more precisely, we consider manifolds equipped with smooth reduced canonical divisors;
\item[-] we study the behavior of the extended BCOV invariant under blow-ups.
\end{itemize}

To fully confirm Conjecture \ref{intro-conj} following the strategy above,
it is necessary to further extend the BCOV invariant as well as the blow-up formula.
This is exactly the purpose of this paper.
We will consider pairs consisting of a compact K{\"a}hler manifold
and a canonical $\Q$-divisor on the manifold with simple normal crossing support and without component of multiplicity $\leqslant -1$.
We will construct the BCOV invariant of such pairs and establish a blow-up formula for our BCOV invariant.

In the joint work with Fu \cite{fz},
we will use the results in this paper together with a factorization theorem of Abramovich, Karu, Matsuki and W{\l}odarczyk \cite[Theorem 0.3.1]{akmw}
to confirm Conjecture \ref{intro-conj} in full generality.

Let us now give more detail about the matter of this paper.

\vspace{2.5mm}

\noindent\textbf{BCOV torsion.}
We will use the notations in \eqref{eq-def-det} and \eqref{eq2-def-det}.
Let $X$ be an $n$-dimensional compact K{\"a}hler manifold.
Let $H^\bullet_\mathrm{dR}(X)$ be the de Rham cohomology of $X$.
Let $H^k_\mathrm{dR}(X) = \bigoplus_{p+q=k}H^{p,q}(X)$ be the Hodge decomposition.
Set
\begin{align}
\label{introeq-lambda}
\begin{split}
& \lambda_p(X) = \det H^{p,\bullet}(X) = \bigotimes_{q=0}^n \Big( \det H^{p,q}(X) \Big)^{(-1)^q} \hspace{4mm} \text{for } p=0,\cdots,n \;,\\
& \lambda_\mathrm{dR}(X) = \bigotimes_{k=1}^{2n} \Big(\det H^k_\mathrm{dR}(X)\Big)^{(-1)^kk} = \bigotimes_{p=1}^n \Big(\lambda_p(X)\otimes\overline{\lambda_p(X)}\Big)^{(-1)^pp} \;.
\end{split}
\end{align}

Let $H^\bullet_\mathrm{Sing}(X,\C)$ be the singular cohomology of $X$ with coefficients in $\C$.
We identify $H^k_\mathrm{dR}(X)$ with $H^k_\mathrm{Sing}(X,\C)$ (see \eqref{eq-sing-dr}).
For $k=0,\cdots,2n$,
let
\begin{equation}
\sigma_{k,1},\cdots,\sigma_{k,b_k}
\in \mathrm{Im}\big(H^k_\mathrm{Sing}(X,\Z) \rightarrow H^k_\mathrm{Sing}(X,\R)\big) \subseteq H^k_\mathrm{dR}(X)
\end{equation}
be a basis of the lattice.
Set
\begin{equation}
\sigma_X = \bigotimes_{k=1}^{2n} \big(\sigma_{k,1}\wedge\cdots\wedge\sigma_{k,b_k}\big)^{(-1)^kk} \in \lambda_\mathrm{dR}(X) \;,
\end{equation}
which is well-defined up to $\pm 1$.

Let $\omega$ be a K{\"a}hler form on $X$.
Let $\big\lVert\cdot\big\rVert_{\lambda_p(X),\omega}$
be the Quillen metric (see \textsection \ref{subsect-q}) on $\lambda_p(X)$ associated with $\omega$.
Let $\big\lVert\cdot\big\rVert_{\lambda_\mathrm{dR}(X),\omega}$ be the metric on $\lambda_\mathrm{dR}(X)$
induced by $\big\lVert\cdot\big\rVert_{\lambda_p(X),\omega}$ via \eqref{introeq-lambda}.
We define
\begin{equation}
\label{introeq-def-tau-BCOV}
\tau_\mathrm{BCOV}(X,\omega) = \log \big\lVert\sigma_X\big\rVert_{\lambda_\mathrm{dR}(X),\omega} \;.
\end{equation}

\vspace{2.5mm}

\noindent\textbf{BCOV invariant.}
For a compact complex manifold $X$ and a divisor $D$ on $X$,
we denote
\begin{equation}
D = \sum_{j=1}^l m_j D_j \;,
\end{equation}
where $m_j\in\Z\backslash\{0\}$,
$D_1,\cdots,D_l \subseteq X$ are mutually distinct and irreducible.
We call $D$ a divisor with simple normal crossing support
if $D_1,\cdots,D_l$ are smooth and transversally intersect.
Let $d$ be a non-zero integer.
We assume that $D$ is of simple normal crossing support and $m_j\neq -d$ for $j=1,\cdots,l$.
For $J\subseteq\big\{1,\cdots,l\big\}$,
we denote
\begin{equation}
\label{introeq-def-wJ-DJ}
w_d^J = \prod_{j\in J} \frac{-m_j}{m_j+d} \;,\hspace{4mm}
D_J = X \cap \bigcap_{j\in J} D_J \;.
\end{equation}
In particular,
we have $w_d^\emptyset = 1$ and $D_\emptyset = X$.

Now let $X$ be a compact K{\"a}hler manifold.
Let $K_X$ be the canonical line bundle over $X$.
Let $K_X^d$ be the $d$-th tensor power of $K_X$.
Let $\gamma\in\mathscr{M}(X,K_X^d)$ be an invertible element.
We denote
\begin{equation}
\mathrm{div}(\gamma) = D = \sum_{j=1}^l m_j D_j \;,
\end{equation}
where $m_j\in\Z\backslash\{0\}$,
$D_1,\cdots,D_l \subseteq X$ are mutually distinct and irreducible.

\begin{defn}
\label{introdef-cy-pair}
We call $(X,\gamma)$ a $d$-Calabi-Yau pair if
\begin{itemize}
\item[1)] $\mathrm{div}(\gamma)$ is of simple normal crossing support;
\item[2)] $m_j\neq -d$ for $j=1,\cdots,l$.
\end{itemize}
\end{defn}

Now we assume that $(X,\gamma)$ is a $d$-Calabi-Yau pair.
Let $w_d^J$ and $D_J$ be as in \eqref{introeq-def-wJ-DJ}.
Let $\omega$ be a K{\"a}hler form on $X$.
Recall that $\tau_\mathrm{BCOV}(\cdot,\cdot)$ was constructed in \eqref{introeq-def-tau-BCOV}.
The BCOV invariant of $(X,\gamma)$ is defined as
\begin{equation}
\tau_d(X,\gamma) = \sum_{J\subseteq\{1,\cdots,l\}} w_d^J \tau_\mathrm{BCOV}\big(D_J,\omega\big|_{D_J}\big)
+ \text{correction terms} \;,
\end{equation}
where the correction terms are Bott-Chern type integrations (see Definition \ref{def-bcov-inv} and \eqref{eq-def-tau-gamma-omega}).
We will construct $\tau_d(X,\gamma)$ and show that it is independent of $\omega$.

We can further extend our construction to canonical $\Q$-divisors.
We consider a pair $(X,D)$,
where $X$ is an $n$-dimensional compact K{\"a}hler manifold,
$D$ is a canonical $\Q$-divisor on $X$ such that
\begin{itemize}
\item[-] $D$ is of simple normal crossing support;
\item[-] each component of $D$ is of multiplicity $>-1$.
\end{itemize}

\begin{defn}
Let $d$ be a positive integer such that $dD$ is a divisor with integer coefficients.
Let $\gamma$ be a meromorphic section of $K_X^d$ such that $\mathrm{div}(\gamma) = dD$.
We define
\begin{equation}
\tau(X,D) = \tau_d(X,\gamma) + \frac{\chi_d(X,dD)}{12} \log
\Big( \big(2\pi\big)^{-n} \int_{X\backslash|D|} \big|\gamma\overline{\gamma}\big|^{1/d} \Big) \;,
\end{equation}
where $\chi_d(\cdot,\cdot)$ is defined in Definition \ref{def-chi-pair},
$|D|$ is defined in \eqref{eq-def-suppD},
$\big|\gamma\overline{\gamma}\big|^{1/d}$ is the unique positive volume form on $X\backslash|D|$
whose $d$-th tensor power equals $e^{i\theta}\gamma\overline{\gamma}$ for certain $\theta\in\R$.
By Proposition \ref{intro-prop-tau-r}, \ref{intro-prop-tau-z},
the BCOV invariant $\tau(X,D)$ is well-defined, 
i.e., independent of $d$ and $\gamma$. 
\end{defn}

Our BCOV invariant differs from the one defined in \cite{efm} by a topological invariant.
More precisely,
if $X$ is a Calabi-Yau manifold,
the logarithm of the BCOV invariant of $X$ defined in \cite{efm} is equal to
\begin{equation}
\label{introeq-tau-compare}
\tau(X,\emptyset) + \frac{\log(2\pi)}{2} \sum_{k=0}^{2n} (-1)^k k(n-k) b_k(X) \;,
\end{equation}
where $b_k(X)$ is the $k$-th Betti number of $X$.
The sum of Betti numbers in \eqref{introeq-tau-compare} comes from
our choice of the $L^2$-metric (see \eqref{eq-def-L2-metric})
and the identification between singular cohomology and de Rham cohomology (see \eqref{eq-sing-dr}).

\vspace{2.5mm}

\noindent\textbf{Curvature formula.}
Let $\pi: \mathscr{X} \rightarrow S$ be a holomorphic submersion.
We assume that $\pi$ is locally K{\"a}hler,
i.e., for any $s\in S$, there exists an open subset $x\in U \subseteq S$ such that
$\pi^{-1}(U)$ is K{\"a}hler.
For $s\in S$,
we denote $X_s = \pi^{-1}(s)$.
Let
\begin{equation}
\big(\gamma_s\in\mathscr{M}(X_s,K_{X_s}^d)\big)_{s\in S}
\end{equation}
be a holomorphic family.
We assume that $(X_s,\gamma_s)$ is a $d$-Calabi-Yau pair for any $s\in S$.
We assume that there exist $l\in\N$, $m_1,\cdots,m_l\in\Z\backslash\{0,-d\}$
and $\big(D_{j,s}\subseteq X_s\big)_{j\in\{1,\cdots,l\},\; s\in S}$ such that
\begin{equation}
\mathrm{div}(\gamma_s) = \sum_{j=1}^l m_j D_{j,s}
\hspace{4mm} \text{for } s\in S \;.
\end{equation}
For $J\subseteq\{1,\cdots,l\}$ and $s\in S$,
let $D_{J,s} \subseteq X_s$ be as in \eqref{introeq-def-wJ-DJ} with $X$ replaced by $X_s$ and $D_j$ replaced by $D_{j,s}$.
We assume that $\big(D_{J,s}\big)_{s\in S}$ is a smooth holomorphic family for each $J$.

Let $\tau_d(X,\gamma)$ be the function $s\mapsto \tau_d(X_s,\gamma_s)$ on $S$.
Let $w_d^J$ be as in \eqref{introeq-def-wJ-DJ}.
Let $H^\bullet(D_J)$ be the variation of Hodge structure associated with $\big(D_{J,s}\big)_{s\in S}$.
Let $\omega_{H^\bullet(D_J)}\in\Omega^{1,1}(S)$ be its Hodge form (see \cite[\textsection 1.2]{z}).

\begin{thm}
\label{thm-curvature}
The following identity holds,
\begin{equation}
\label{introeq-thm-curvature}
\frac{\overline{\partial}\partial}{2\pi i} \tau_d(X,\gamma)
= \sum_{J\subseteq\{1,\cdots,l\}} w_d^J \omega_{H^\bullet(D_J)} \;.
\end{equation}
\end{thm}

\vspace{2.5mm}

\noindent\textbf{Blow-up formula.}
Let $(X,\gamma)$ be a $d$-Calabi-Yau pair with $d>0$.
We denote
\begin{equation}
\mathrm{div}(\gamma) = D = \sum_{j=1}^l m_j D_j \;,
\end{equation}
where $m_j\in\Z\backslash\{-d,0\}$,
$D_1,\cdots,D_l \subseteq X$ are mutually distinct and irreducible.

Let $Y \subseteq X$ be a connected complex submanifold
such that $Y,D_1,\cdots,D_l$ transversally intersect (in the sense of Definition \ref{def-ti}).
We assume that $m_j > 0$ for $j$ satisfying $Y \subseteq D_j$.
Let $r$ be the codimension of $Y\subseteq X$.
Let $s$ be the number of $D_j$ containing $Y$.
Then we have $s\leqslant r$.
Without loss of generality,
we assume that
\begin{equation}
Y \subseteq D_j \hspace{2.5mm} \text{for } j=1,\cdots,s \;;\hspace{4mm}
Y \nsubseteq D_j \hspace{2.5mm} \text{for } j=s+1,\cdots,l \;.
\end{equation}

Let $f: X' \rightarrow X$ be the blow-up along $Y$.
Let $D_j' \subseteq X'$ be the strict transformation of $D_j \subseteq X$.
Set $E = f^{-1}(Y)$.
We denote $D'=\mathrm{div}(f^*\gamma)$.
We denote
\begin{equation}
m_0 = m_1 + \cdots + m_s + rd -d \;.
\end{equation}
We have
\begin{equation}
D' = m_0 E + \sum_{j=1}^l m_j D_j' \;.
\end{equation}
Hence $(X',f^*\gamma)$ is a $d$-Calabi-Yau pair.

Set
\begin{align}
\begin{split}
D_Y = \sum_{j=s+1}^l m_j (D_j \cap Y) \;,\hspace{4mm}
D_E = \sum_{j=1}^l m_j (D_j' \cap E) \;.
\end{split}
\end{align}
Then $D_Y$ (resp. $D_E$) is a divisor on $Y$ (resp. $E$) with simple normal crossing support.

We identify $\CP^r$ with $\C^r \cup \CP^{r-1}$.
Let $(z_1,\cdots,z_r)\in\C^r$ be the coordinates.
Let $\gamma_{r,m_1,\cdots,m_s}\in\mathscr{M}(\CP^r,K_{\CP^r}^d)$ be such that
\begin{equation}
\label{introeq-gammaZ}
\gamma_{r,m_1,\cdots,m_s} \big|_{\C^r} = \big( dz_1\wedge\cdots\wedge dz_r \big)^d \prod_{j=1}^s z_j^{m_j} \;.
\end{equation}
Then $(\CP^r,\gamma_{r,m_1,\cdots,m_s})$ is a $d$-Calabi-Yau pair.

\begin{thm}
\label{introthm-tau-bl}
The following identities hold,
\begin{align}
\label{introeq-thm-tau-bl}
\begin{split}
\chi_d(X',f^*\gamma) - \chi_d(X,\gamma) & = 0 \;,\\
\tau_d(X',f^*\gamma) - \tau_d(X,\gamma) & = \chi_d(E,D_E) \tau_d\big(\CP^1,\gamma_{1,m_0}\big) \\
& \hspace{15mm} - \chi_d(Y,D_Y) \tau_d\big(\CP^r,\gamma_{r,m_1,\cdots,m_s}\big) \;,
\end{split}
\end{align}
where $\chi_d(\cdot,\cdot)$ is given by Definition \ref{def-chi-pair}.
\end{thm}

The proof of Theorem \ref{introthm-tau-bl} is based on
\begin{itemize}
\item[-] the deformation to the normal cone introduced by Baum, Fulton and MacPherson \cite[\textsection 1.5]{bfm};
\item[-] the immersion formula for Quillen metrics due to Bismut and Lebeau \cite{ble};
\item[-] the submersion formula for Quillen metrics due to Berthomieu and Bismut \cite{bb};
\item[-] the blow-up formula for Quillen metrics due to Bismut \cite{b97};
\item[-] the relation between the holomorphic torsion and the de Rham torsion established by Bismut \cite{b04}.
\end{itemize}

\vspace{2.5mm}

\noindent\textbf{Notations.}
For a complex vector space $V$,
we denote
\begin{equation}
\label{eq-def-det}
\det V = \Lambda^{\dim V} V \;,
\end{equation}
which is a complex line.
For a complex line $\lambda$,
we denote by $\lambda^{-1}$ the dual of $\lambda$.
For a graded complex vector space $V^\bullet = \bigoplus_{k=0}^m V^k$,
we denote
\begin{equation}
\label{eq2-def-det}
\det V^\bullet = \bigotimes_{k=0}^m \big(\det V^k\big)^{(-1)^k} \;.
\end{equation}

For a complex manifold $X$ and a divisor $D = m_1 D_1 + \cdots + m_l D_l$ on $X$,
where $m_1,\cdots,m_l\in\Z\backslash\{0\}$,
$D_1,\cdots,D_l$ are mutually distinct and irreducible,
we denote
\begin{equation}
\label{eq-def-suppD}
|D| = D_1 \cup \cdots \cup D_l \subseteq X \;.
\end{equation}

For a complex manifold $X$,
we denote by $\Omega^{p,q}(X)$ the vector space of $(p,q)$-forms on $X$.
We denote by $\mathscr{O}_X$ the analytic coherent sheaf of holomorphic functions on $X$.
We denote by $\Omega^p_X$ the analytic coherent sheaf of holomorphic $p$-forms on $X$.
For a complex vector bundle $E$ over $X$,
we denote by $\Omega^{p,q}(X,E)$ the vector space of $(p,q)$-forms on $X$ with values in $E$.
We denote by $\mathscr{M}(X,E)$ the vector space of meromorphic sections of $E$.
We denote by $\mathscr{O}_X(E)$ the analytic coherent sheaf of holomorphic sections of $E$.
For an analytic coherent sheaf $\mathscr{F}$ on $X$,
we denote by $H^q(X,\mathscr{F})$ the $q$-th cohomology of $\mathscr{F}$.
We denote $H^q(X,E) = H^q(X,\mathscr{O}_X(E))$.
We denote $H^{p,q}(X) = H^q(X,\Omega^p_X)$.
We denote by $H^k_\mathrm{dR}(X)$ the $k$-th de Rham cohomology of $X$ with coefficients in $\C$.
If $X$ is a compact K{\"a}hler manifold,
we identify $H^{p,q}(X)$ with a vector subspace of $H^{p+q}_\mathrm{dR}(X)$ via the Hodge theorem.

\vspace{2.5mm}

\noindent\textbf{Acknowledgments.}
The author is grateful to Professor Ken-Ichi Yoshikawa
who was the author's postdoctoral advisor and a member of the author's dissertation committee.
Prof. Yoshikawa drew the author's attention to the BCOV invariant
and gave many helpful suggestions.
The author is grateful to Professor Kenji Matsuki
who kindly explained their result \cite{akmw} to the author.
The author is grateful to
Professor Xianzhe Dai, Professor Gerard Freixas i Montplet and Professor Vincent Maillot
for their interest in this work.
The author is grateful to Professor Yang Cao for many helpful discussions.

This work was supported by JSPS KAKENHI Grant JP17F17804.
This work was supported by KIAS individual Grant MG077401 at Korea Institute for Advanced Study.

\section{Preliminaries}
\label{sect-pre}

\subsection{Divisor with simple normal crossing support}
\label{subsect-snc}

For $I\subseteq\big\{1,\cdots,n\big\}$,
we denote
\begin{equation}
\C^n_I = \big\{ (z_1,\cdots,z_n)\in \C^n \;:\; z_i = 0 \hspace{2.5mm} \text{for } i\in I \big\} \subseteq \C^n \;.
\end{equation}

Let $X$ be an $n$-dimensional complex manifold.

\begin{defn}
\label{def-ti}
For closed complex submanifolds $Y_1,\cdots,Y_l \subseteq X$,
we say that $Y_1,\cdots,Y_l$ transversally intersect if
for any $x\in X$,
there exists a holomorphic local chart $\C^n \supseteq U \xrightarrow{\varphi} X$
such that
\begin{itemize}
\item[-] $0\in U$ and $\varphi(0) = x$;
\item[-] for each $k$,
either $\varphi^{-1}(Y_k) = \emptyset$
or $\varphi^{-1}(Y_k) = U \cap \C^n_{I_k}$ for certain $I_k\subseteq\big\{1,\cdots,n\big\}$.
\end{itemize}
\end{defn}

Let $D$ be a divisor on $X$.
We denote
\begin{equation}
D = \sum_{j=1}^l m_j D_j \;,
\end{equation}
where $m_j\in\Z\backslash\{0\}$,
$D_1,\cdots,D_l \subseteq X$ are mutually distinct and irreducible.

\begin{defn}
\label{def-snc}
We call $D$ a divisor with simple normal crossing support
if $D_1,\cdots,D_l$ are smooth and transversally intersect.
\end{defn}

For $J\subseteq\big\{1,\cdots,l\big\}$,
let $w_d^J$ and $D_J$ be as in \eqref{introeq-def-wJ-DJ},
let $\chi(D_J)$ be the topological Euler characteristic of $D_J$.

\begin{defn}
\label{def-chi-pair}
If $D$ is a divisor with simple normal crossing support,
we define
\begin{equation}
\label{introeq-def-chi-pair}
\chi_d(X,D) = \sum_{J\subseteq\{1,\cdots,l\}} w_d^J \chi(D_J) \;.
\end{equation}
Moreover,
if there is a meromorphic section $\gamma$ of a holomorphic line bundle over $X$
such that $\mathrm{div}(\gamma) = D$,
we define
\begin{equation}
\label{introeq2-def-chi-pair}
\chi_d(X,\gamma) = \chi_d(X,D) \;.
\end{equation}
\end{defn}

Now we assume that $D$ is a divisor with simple normal crossing support.
Let $L$ be a holomorphic line bundle over $X$ together with $\gamma\in\mathscr{M}(X,L)$ such that $\mathrm{div}(\gamma) = D$.
Let $\gamma^{-1}\in\mathscr{M}(X,L^{-1})$ be the inverse of $\gamma$.

We denote by $\big(T^*X\oplus\overline{T^*X}\big)^{\otimes k}$
the $k$-th tensor power of $T^*X\oplus\overline{T^*X}$.
We denote
\begin{equation}
E^\pm_k = \big(T^*X\oplus\overline{T^*X}\big)^{\otimes k} \otimes L^{\pm 1} \;.
\end{equation}
In particular,
we have $E_0^{\pm} = L^{\pm 1}$.
Let $\n^{E^\pm_k}$ be a connection on $E^\pm_k$.

Let $L_j$ be the normal line bundle of $D_j \hookrightarrow X$.

\begin{defn}
\label{def-res}
We define $\mathrm{Res}_{D_j}(\gamma)\in\mathscr{M}(D_j,L \otimes L_j^{-m_j})$ as follows,
\begin{equation}
\mathrm{Res}_{D_j}(\gamma) = \left\{
\begin{array}{ll}
\frac{1}{m_j!}\Big(\n^{E^+_{m_j-1}} \cdots \n^{E^+_0} \gamma\Big) \Big|_{D_j}
& \text{if } m_j > 0 \;,\\
\frac{1}{|m_j|!} \bigg(\Big(\n^{E^-_{|m_j|-1}} \cdots \n^{E^-_0} \gamma^{-1}\Big) \Big|_{D_j} \bigg)^{-1}
& \text{if } m_j < 0 \;.
\end{array} \right.
\end{equation}
Here $\mathrm{Res}_{D_j}(\gamma)$ is independent of $\big(\n^{E^\pm_k}\big)_{k\in\N}$.
\end{defn}

For $j=1,\cdots,l$,
we have
\begin{equation}
\label{eq-res-div}
\mathrm{div}\big(\mathrm{Res}_{D_j}(\gamma)\big) = \sum_{k\in\{1,\cdots,l\}\backslash\{j\}} m_k \big( D_j \cap D_k \big) \;.
\end{equation}
For distinct $j,k=1,\cdots,l$,
we have
\begin{align}
\label{eq-res-commute}
\begin{split}
& \mathrm{Res}_{D_j \cap D_k}\big(\mathrm{Res}_{D_j}(\gamma)\big)
= \mathrm{Res}_{D_j \cap D_k}\big(\mathrm{Res}_{D_k}(\gamma)\big) \\
& \hspace{35mm} \in \mathscr{M}\big(D_j \cap D_k,L \otimes L_j^{-m_j} \otimes L_k^{-m_k}\big) \;.
\end{split}
\end{align}

\subsection{Some characteristic classes}
\label{subsect-char}

For an $(m\times m)$-matrix $A$,
we define
\begin{equation}
\label{eq-def-char}
\mathrm{ch}(A) = \tr\big[e^A\big] \;,\hspace{4mm}
\mathrm{Td}(A) = \det\Big(\frac{A}{\Id-e^{-A}}\Big) \;,\hspace{4mm}
c(A) = \det\big(\Id+A\big) \;.
\end{equation}
We have
\begin{equation}
c(tA) = 1 + \sum_{k=1}^m t^kc_k(A) \;,
\end{equation}
where $c_k(A)$ is the $k$-th elementary symmetric polynomial of the eigenvalues of $A$.

Let $V$ be an $m$-dimensional complex vector space.
Let $R\in\mathrm{End}(V)$.
Let $V^*$ be the dual of $V$.
Let $R^*\in\mathrm{End}(V^*)$ be the dual of $R$.
For $r=1,\cdots,m$,
we construct $R_r\in\mathrm{End}(\Lambda^rV^*)$ by induction,
\begin{equation}
R_1 = -R^* \;,\hspace{4mm}
R_r = R_1 \wedge \Id_{\Lambda^{r-1}V^*} + \Id_{V^*} \wedge R_{r-1} \;.
\end{equation}
We will use the convention $\Lambda^0V^* = \C$ and $R_0 = 0$.

Let $\lambda_1,\cdots,\lambda_m$ be the eigenvalues of $R$.
For $p\in\N$ and $F$ a polynomial of $\lambda_1,\cdots,\lambda_m$,
we denote by $\big\{F\big\}^{[p]}$ the component of $F$ of degree $p$.

\begin{prop}
\label{prop-total-class}
The following identities hold,
\begin{align}
\label{eq-prop-total-class}
\begin{split}
\mathrm{Td}(R) \Big( \sum_{r=0}^m (-1)^r \mathrm{ch}(R_r) \Big)
& = c_m(R) \;,\\
\bigg\{ \mathrm{Td}(R) \Big( \sum_{r=1}^m (-1)^r r \mathrm{ch}(R_r) \Big) \bigg\}^{[\leqslant m]}
& = -c_{m-1}(R) + \frac{m}{2}c_m(R) \;,\\
\bigg\{ \mathrm{Td}(R) \Big( \sum_{r=2}^m (-1)^r r(r-1) \mathrm{ch}(R_r) \Big) \bigg\}^{[m]}
& = \frac{1}{6}(c_1c_{m-1})(R) + \frac{m(3m-5)}{12}c_m(R) \;.
\end{split}
\end{align}
\end{prop}
\begin{proof}
We have
\begin{equation}
\label{eq1-pf-prop-total-class}
\mathrm{Td}(R) = \prod_{j=1}^m \frac{\lambda_j}{1-e^{-\lambda_j}} \;,\hspace{4mm}
\sum_{r=0}^m (-1)^r t^r \mathrm{ch}(R_r) = \prod_{j=1}^m \big(1-te^{-\lambda_j}\big) \;.
\end{equation}

Taking $t=1$ in \eqref{eq1-pf-prop-total-class},
we obtain the first identity in \eqref{eq-prop-total-class}.

Taking the derivative of the second identity in \eqref{eq1-pf-prop-total-class} at $t=1$,
we get
\begin{equation}
\label{eq3-pf-prop-total-class}
\sum_{r=0}^m (-1)^r r \mathrm{ch}(R_r) =
- \Big(\sum_{j=1}^m \frac{e^{-\lambda_j}}{1-e^{-\lambda_j}}\Big)
\prod_{j=1}^m \big(1-e^{-\lambda_j}\big) \;.
\end{equation}
From the first identity in \eqref{eq1-pf-prop-total-class}, \eqref{eq3-pf-prop-total-class}
and the identity
\begin{equation}
\label{eq4-pf-prop-total-class}
\frac{e^{-\lambda_j}}{1-e^{-\lambda_j}} =
\lambda_j^{-1} - \frac{1}{2} + \frac{1}{12}\lambda_j + \cdots \;,
\end{equation}
we obtain the second identity in \eqref{eq-prop-total-class}.

Taking the second derivative of the second identity in \eqref{eq1-pf-prop-total-class} at $t=1$,
we get
\begin{align}
\label{eq5-pf-prop-total-class}
\begin{split}
& \sum_{r=0}^m (-1)^r r(r-1) \mathrm{ch}(R_r) \\
& = \bigg( \Big(\sum_{j=1}^m \frac{e^{-\lambda_j}}{1-e^{-\lambda_j}}\Big)^2
- \sum_{j=1}^m \Big(\frac{e^{-\lambda_j}}{1-e^{-\lambda_j}}\Big)^2 \bigg)
\prod_{j=1}^m \big(1-e^{-\lambda_j}\big) \;.
\end{split}
\end{align}
From the first identity in \eqref{eq1-pf-prop-total-class}, \eqref{eq4-pf-prop-total-class}
and \eqref{eq5-pf-prop-total-class},
we obtain the third identity in \eqref{eq-prop-total-class}.
This completes the proof.
\end{proof}

For an $(m\times m)$-matrix $A$,
we define
\begin{equation}
\label{eq-def-tdprim}
\mathrm{Td}'(A) = \frac{\partial}{\partial t} \mathrm{Td}(A+t\Id) \Big|_{t=0} \;.
\end{equation}

\begin{prop}
\label{prop2-total-class}
We have
\begin{align}
\label{eq-prop2-total-class}
\begin{split}
\bigg\{ \mathrm{Td}'(R) \Big( \sum_{r=0}^m (-1)^r \mathrm{ch}(R_r) \Big) \bigg\}^{[m]}
& = \frac{m}{2}c_m(R) \;,\\
\bigg\{ \mathrm{Td}'(R) \Big( \sum_{r=0}^m (-1)^r r \mathrm{ch}(R_r) \Big) \bigg\}^{[m]}
& = \frac{1}{12}(c_1c_{m-1})(R) + \frac{m^2}{4}c_m(R) \;.
\end{split}
\end{align}
\end{prop}
\begin{proof}
Let $c_k'$ be as in \eqref{eq-def-tdprim} with $\mathrm{Td}$ replaced by $c_k$.
We have
\begin{equation}
\label{eq1-pf-prop2-total-class}
c_1'(R) = m \;,\hspace{4mm}
c_2'(R) = (m-1)c_1(R) \;.
\end{equation}
On the other hand,
we have
\begin{equation}
\label{eq2-pf-prop2-total-class}
\Big\{\mathrm{Td}(R)\Big\}^{[\leqslant 2]}
= 1 + \frac{1}{2}c_1(R) + \frac{1}{12}\big(c_1^2(R)+c_2(R)\big) \;.
\end{equation}
By \eqref{eq1-pf-prop2-total-class} and \eqref{eq2-pf-prop2-total-class},
we have
\begin{equation}
\label{eq3-pf-prop2-total-class}
\bigg\{\frac{\mathrm{Td}'(R)}{\mathrm{Td}(R)}\bigg\}^{[\leqslant 1]}
= \frac{m}{2} - \frac{1}{12}c_1(R) \;.
\end{equation}
From \eqref{eq-prop-total-class} and \eqref{eq3-pf-prop2-total-class},
we obtain \eqref{eq-prop2-total-class}.
This completes the proof.
\end{proof}

\subsection{Chern form and Bott-Chern form}
\label{subsect-bc}

Let $S$ be a complex manifold.
We denote
\begin{equation}
\label{eq-def-QQ}
Q^S = \bigoplus_{p=0}^{\dim S} \Omega^{p,p}(S) \;,\hspace{4mm}
Q^{S,0} = \bigoplus_{p=1}^{\dim S} \Big(\partial\,\Omega^{p-1,p}(S)+\overline{\partial}\,\Omega^{p,p-1}(S)\Big) \subseteq Q^S \;.
\end{equation}
Let $E$ be a holomorphic vector bundle over $S$.
Let $g^E$ be a Hermitian metric on $E$.
Let $R^E \in \Omega^{1,1}(S,\mathrm{End}(E))$
be the Chern curvature of $(E,g^E)$.
Recall that $c(\cdot)$ was defined in \eqref{eq-def-char}.
The total Chern form of $(E,g^E)$ is defined by
\begin{equation}
c\big(E,g^E\big) = c\Big(-\frac{R^E}{2\pi i}\Big) \in Q^S \;.
\end{equation}
The total Chern class of $E$ is defined by
\begin{equation}
c(E) = \big[c\big(E,g^E\big)\big]\in H^\mathrm{even}_\mathrm{dR}(S) \;,
\end{equation}
which is independent of $g^E$.

Let $E'\subseteq E$ be a holomorphic subbundle.
Let $E''=E/E'$.
We have a short exact sequence of holomorphic vector bundles over $S$,
\begin{equation}
0 \rightarrow E' \xrightarrow{\alpha} E \xrightarrow{\beta} E'' \rightarrow 0 \;,
\end{equation}
where $\alpha$ (resp. $\beta$) is the canonical embedding (resp. projection).
We have
\begin{equation}
c(E) = c(E')c(E'') \;.
\end{equation}
Let $g^{E'}$ be a Hermitian metric on $E'$.
Let $g^{E''}$ be a Hermitian metric on $E''$.
The Bott-Chern form \cite[Section 1f)]{bgs1}
\begin{equation}
\widetilde{c}\big(g^{E'},g^E,g^{E''}\big)\in Q^S/Q^{S,0}
\end{equation}
is such that
\begin{equation}
\label{eq-def-BC}
\frac{\overline{\partial}\partial}{2\pi i} \widetilde{c}\big(g^{E'},g^E,g^{E''}\big)
= c\big(E,g^E\big) - c\big(E',g^{E'}\big)c\big(E'',g^{E''}\big) \;.
\end{equation}
Let $\alpha^*g^E$ be the Hermitian metric on $E'$ induced by $g^E$ via the embedding $\alpha: E'\rightarrow E$.
Let $\beta_*g^E$ be the quotient Hermitian metric on $E''$ induced by $g^E$ via the surjection $\beta: E\rightarrow E''$.
We denote
\begin{equation}
\label{eq-def-BCc-ses}
\widetilde{c}\big(E',E,g^E\big) = \widetilde{c}\big(\alpha^*g^E,g^E,\beta_*g^E\big) \;.
\end{equation}

Let $\beta^*g^{E''}$ be the Hermitian pseudometric on $E$
induced by $g^{E''}$ via the surjection $\beta: E\rightarrow E''$.
For $\e>0$,
set
\begin{equation}
\label{eq-gEe}
g^E_\e = g^E + \frac{1}{\e} \beta^*g^{E''} \;.
\end{equation}

We equip $Q^S \subseteq \Omega^{\bullet,\bullet}(S)$ with the compact-open topology.
We equip $Q^S/Q^{S,0}$ with the quotient topology.

\begin{prop}
\label{prop-adiabatic-bc}
As $\e\rightarrow 0$,
\begin{equation}
\label{eq-prop-adiabatic-bc}
c\big(E,g^E_\e\big) \rightarrow c\big(E',\alpha^*g^E\big)c\big(E'',g^{E''}\big) \;,\hspace{4mm}
\widetilde{c}\big(E',E,g^E_\e\big) \rightarrow 0 \;.
\end{equation}
\end{prop}
\begin{proof}
We will follow the proof of \cite[Theorem 1.29]{bgs1}.

Let $\mathrm{pr}: S \times \C \rightarrow S$ be the canonical projection.
Let
\begin{equation}
\label{eqe2-pf-prop-adiabatic-bc}
\widetilde{\alpha}: \mathrm{pr}^* E' \rightarrow \mathrm{pr}^* E
\end{equation}
be the pull-back of $\alpha: E' \rightarrow E$.
Let $(s,z)\in S\times\C$ be coordinates.
Let $\sigma\in H^0(S\times\C,\C)$ be the holomorphic function $\sigma(s,z) = z$.
Let
\begin{equation}
\label{eqe3-pf-prop-adiabatic-bc}
\widetilde{\sigma}: \mathrm{pr}^* E' \rightarrow \mathrm{pr}^* E'
\end{equation}
be the multiplication by $\sigma$.
Set
\begin{equation}
\label{eqe4-pf-prop-adiabatic-bc}
\mathcal{E}' = \mathrm{pr}^* E' \;,\hspace{4mm}
\mathcal{E} = \mathrm{Coker} \Big( \widetilde{\alpha} \oplus \widetilde{\sigma}:
\mathrm{pr}^* E' \rightarrow \mathrm{pr}^* E \oplus \mathrm{pr}^* E' \Big) \;.
\end{equation}
We get a short exact sequence of holomorphic vector bundles over $S\times\C$,
\begin{equation}
\label{eqe5-pf-prop-adiabatic-bc}
0 \rightarrow \mathcal{E}' \rightarrow \mathcal{E} \rightarrow \mathcal{E}'' \rightarrow 0 \;,
\end{equation}
where $\mathcal{E}' \rightarrow \mathcal{E}$ is induced by the embedding
$0 \oplus \Id_{\mathrm{pr}^* E'} : \mathrm{pr}^* E' \hookrightarrow \mathrm{pr}^* E \oplus \mathrm{pr}^* E'$,
and $\mathcal{E} \rightarrow \mathcal{E}'' := \mathrm{Coker}\big(\mathcal{E}'\rightarrow\mathcal{E}\big)$ is the canonical projection.
For $z\in\C$,
let
\begin{equation}
\label{eqi1-pf-prop-adiabatic-bc}
0 \rightarrow \mathcal{E}_z' \rightarrow \mathcal{E}_z \rightarrow \mathcal{E}_z'' \rightarrow 0
\end{equation}
be the restriction of \eqref{eqe5-pf-prop-adiabatic-bc} to $S\times\{z\}$.
For $z\neq 0$,
let
\begin{equation}
\label{eqi2-pf-prop-adiabatic-bc}
\phi_z: E \rightarrow \mathcal{E}_z
= \mathrm{Coker} \Big( \alpha \oplus z \Id_{E'}: E' \rightarrow E \oplus E' \Big)
\end{equation}
be the isomorphism induced by the embedding $\Id_E \oplus 0: E \hookrightarrow E \oplus E'$.
We get a commutative diagram
\begin{equation}
\label{eqi3-pf-prop-adiabatic-bc}
\xymatrix{
0 \ar[r] & E' \ar[r]\ar[d] & E \ar[r]\ar[d] & E'' \ar[r]\ar[d] & 0 \\
0 \ar[r] &  \mathcal{E}_z' \ar[r] & \mathcal{E}_z \ar[r] & \mathcal{E}_z'' \ar[r] & 0 \;,}
\end{equation}
where the vertical maps are induced by $\phi_z$.
Let
\begin{equation}
\label{eqi6-pf-prop-adiabatic-bc}
\phi_0: E' \oplus E'' \rightarrow \mathcal{E}_0
= \mathrm{Coker} \Big( \alpha \oplus 0 : E' \rightarrow E \oplus E' \Big) = E'' \oplus E'
\end{equation}
be the obvious isomorphism.
We get a commutative diagram
\begin{equation}
\label{eqi7-pf-prop-adiabatic-bc}
\xymatrix{
0 \ar[r] & E' \ar[r]\ar[d] & E' \oplus E'' \ar[r]\ar[d] & E'' \ar[r]\ar[d] & 0 \\
0 \ar[r] &  \mathcal{E}_0' \ar[r] & \mathcal{E}_0 \ar[r] & \mathcal{E}_0'' \ar[r] & 0 \;,}
\end{equation}
where the vertical maps are induced by $\phi_0$.

We can construct a Hermitian metric $g^\mathcal{E}$ on $\mathcal{E}$ such that
\begin{equation}
\label{eqm-pf-prop-adiabatic-bc}
\phi_z^* g^\mathcal{E} = |z|^2 g^E + \beta^* g^{E''} \hspace{2.5mm} \text{for } z \neq 0 \;,\hspace{4mm}
\phi_0^* g^\mathcal{E} = \alpha^*g^E \oplus g^{E''} \;.
\end{equation}
By \eqref{eqi3-pf-prop-adiabatic-bc} and \eqref{eqm-pf-prop-adiabatic-bc},
for $\e = |z|^2 > 0$,
we have
\begin{equation}
\label{eqa-pf-prop-adiabatic-bc}
c\big(\mathcal{E}_z,g^{\mathcal{E}_z}\big) = c\big(E,g^E_\e\big) \;,\hspace{4mm}
\widetilde{c}\big(\mathcal{E}_z',\mathcal{E}_z,g^{\mathcal{E}_z}\big) = \widetilde{c}\big(E',E,g^E_\e\big) \;.
\end{equation}
By \cite[Theorem 1.29 iii)]{bgs1}, \eqref{eqi7-pf-prop-adiabatic-bc} and \eqref{eqm-pf-prop-adiabatic-bc},
we have
\begin{equation}
\label{eqb-pf-prop-adiabatic-bc}
c\big(\mathcal{E}_0,g^{\mathcal{E}_0}\big) = c\big(E',\alpha^*g^E\big)c\big(E'',g^{E''}\big) \;,\hspace{4mm}
\widetilde{c}\big(\mathcal{E}_0',\mathcal{E}_0,g^{\mathcal{E}_0}\big) = 0 \;.
\end{equation}
On the other hand,
by \cite[Theorem 1.29 ii)]{bgs1},
we have
\begin{equation}
\label{eqc-pf-prop-adiabatic-bc}
\lim_{z\rightarrow 0} c\big(\mathcal{E}_z,g^{\mathcal{E}_z}\big)
= c\big(\mathcal{E}_0,g^{\mathcal{E}_0}\big) \;,\hspace{4mm}
\lim_{z\rightarrow 0} \widetilde{c}\big(\mathcal{E}_z',\mathcal{E}_z,g^{\mathcal{E}_z}\big)
= \widetilde{c}\big(\mathcal{E}_0',\mathcal{E}_0,g^{\mathcal{E}_0}\big) \;.
\end{equation}
From \eqref{eqa-pf-prop-adiabatic-bc}-\eqref{eqc-pf-prop-adiabatic-bc},
we obtain \eqref{eq-prop-adiabatic-bc}.
This completes the proof.
\end{proof}

Let $F \subseteq E$ be a holomorphic subbundle.
Set $F' = \alpha^{-1}(F) \subseteq E'$, $F'' = \beta(F) \subseteq E''$.

\begin{prop}
\label{prop2-adiabatic-bc}
If $F'=E'$,
as $\e\rightarrow 0$,
\begin{equation}
\label{eq1-prop2-adiabatic-bc}
\widetilde{c}\big(F,E,g^E_\e\big) \rightarrow c\big(E',\alpha^*g^E\big) \widetilde{c}\big(F'',E'',g^{E''}\big) \;.
\end{equation}
If $F''=E''$,
as $\e\rightarrow 0$,
\begin{equation}
\label{eq2-prop2-adiabatic-bc}
\widetilde{c}\big(F,E,g^E_\e\big) \rightarrow c\big(E'',g^{E''}\big)\widetilde{c}\big(F',E',\alpha^*g^E\big) \;.
\end{equation}
\end{prop}
\begin{proof}
We will use the notations in the proof of Proposition \ref{prop-adiabatic-bc}.
Set
\begin{equation}
\label{eqf-pf-prop2-adiabatic-bc}
\mathcal{F} = \mathrm{Coker} \Big( \widetilde{\alpha} \oplus \widetilde{\sigma}\big|_{\mathrm{pr}^* F'} :
\mathrm{pr}^* F' \rightarrow \mathrm{pr}^* F \oplus \mathrm{pr}^* F' \Big) \subseteq \mathcal{E} \;.
\end{equation}
For $z\in\C$,
let $\mathcal{F}_z$ be the restriction of $\mathcal{F}$ to $S\times\{z\}$.

For $z\neq 0$,
we have $\phi_z(F) = \mathcal{F}_z \subseteq \mathcal{E}_z$.
By \eqref{eqm-pf-prop-adiabatic-bc},
for $\e = |z|^2 > 0$,
we have
\begin{equation}
\label{eqi2-pf-prop2-adiabatic-bc}
\widetilde{c}\big(\mathcal{F}_z,\mathcal{E}_z,g^{\mathcal{E}_z}\big) =
\widetilde{c}\big(F,E,g^E_\e\big) \;.
\end{equation}
We have $\phi_0(F) = F' \oplus F'' \subseteq  E' \oplus E'' = \mathcal{E}_0$.
By \eqref{eqm-pf-prop-adiabatic-bc},
we have
\begin{equation}
\label{eqi4-pf-prop2-adiabatic-bc}
\widetilde{c}\big(\mathcal{F}_0,\mathcal{E}_0,g^{\mathcal{E}_0}\big) =
\widetilde{c}\big(F'\oplus F'',E' \oplus E'',\alpha^*g^E \oplus g^{E''}\big) \;.
\end{equation}
By \cite[Theorem 1.29]{bgs1},
we have
\begin{align}
\label{eqi5-pf-prop2-adiabatic-bc}
\begin{split}
\widetilde{c}\big(F'\oplus F'',E' \oplus E'',\alpha^*g^E \oplus g^{E''}\big)
& = c\big(E',\alpha^*g^E\big) \widetilde{c}\big(F'',E'',g^{E''}\big)
\hspace{4mm}\text{if } F' = E' \;,\\
\widetilde{c}\big(F'\oplus F'',E' \oplus E'',\alpha^*g^E \oplus g^{E''}\big)
& = c\big(E'',g^{E''}\big)\widetilde{c}\big(F',E',\alpha^*g^E\big)
\hspace{4mm}\text{if } F'' = E'' \;.
\end{split}
\end{align}
On the other hand,
by \cite[Theorem 1.29 ii)]{bgs1},
we have
\begin{equation}
\label{eqc-pf-prop2-adiabatic-bc}
\lim_{z\rightarrow 0} \widetilde{c}\big(\mathcal{F}_z,\mathcal{E}_z,g^{\mathcal{E}_z}\big)
= \widetilde{c}\big(\mathcal{F}_0,\mathcal{E}_0,g^{\mathcal{E}_0}\big) \;.
\end{equation}
From \eqref{eqi2-pf-prop2-adiabatic-bc}-\eqref{eqc-pf-prop2-adiabatic-bc},
we obtain \eqref{eq1-prop2-adiabatic-bc} and \eqref{eq2-prop2-adiabatic-bc}.
This completes the proof.
\end{proof}

Recall that $\mathrm{Td}(\cdot)$ was defined in \eqref{eq-def-char}.
The Bott-Chern form \cite[Section 1f)]{bgs1}
\begin{equation}
\widetilde{\mathrm{Td}}\big(g^{E'},g^E,g^{E''}\big) \in Q^S/Q^{S,0}
\end{equation}
is such that
\begin{equation}
\frac{\overline{\partial}\partial}{2\pi i} \widetilde{\mathrm{Td}}\big(g^{E'},g^E,g^{E''}\big)
= \mathrm{Td}\big(E,g^E\big) - \mathrm{Td}\big(E',g^{E'}\big)\mathrm{Td}\big(E'',g^{E''}\big) \;.
\end{equation}

\begin{prop}
\label{prop3-adiabatic-bc}
Proposition \ref{prop-adiabatic-bc}, \ref{prop2-adiabatic-bc} hold with $c(\cdot)$ replaced by $\mathrm{Td}(\cdot)$.
\end{prop}

Recall that $\mathrm{ch}(\cdot)$ was defined in \eqref{eq-def-char}.
The Bott-Chern form \cite[Section 1f)]{bgs1}
\begin{equation}
\label{eq-def-BCch}
\widetilde{\mathrm{ch}}\big(g^{E'},g^E,g^{E''}\big) \in Q^S/Q^{S,0}
\end{equation}
is such that
\begin{equation}
\frac{\overline{\partial}\partial}{2\pi i} \widetilde{\mathrm{ch}}\big(g^{E'},g^E,g^{E''}\big)
=  \mathrm{ch}\big(E',g^{E'}\big) - \mathrm{ch}\big(E,g^E\big) + \mathrm{ch}\big(E'',g^{E''}\big) \;.
\end{equation}
For another Hermitian metric $\widehat{g}^E$ on $E$,
let
\begin{equation}
\widetilde{\mathrm{ch}}\big(\widehat{g}^E,g^E\big) \in Q^S/Q^{S,0}
\end{equation}
be the Bott-Chern form \cite[Section 1f)]{bgs1} such that
\begin{equation}
\frac{\overline{\partial}\partial}{2\pi i} \widetilde{\mathrm{ch}}\big(\widehat{g}^E,g^E\big)
=  \mathrm{ch}\big(E,\widehat{g}^E\big) - \mathrm{ch}\big(E,g^E\big) \;.
\end{equation}

The following proposition is a direct consequence of the construction of the Bott-Chern form \cite[Section 1f)]{bgs1}.

\begin{prop}
\label{prop-mult-bc}
For another Hermitian metric $\widehat{g}^E$ (resp. $\widehat{g}^{E'}$, $\widehat{g}^{E''}$)  on $E$ (resp. $E'$, $E''$),
we have
\begin{align}
\begin{split}
& \widetilde{\mathrm{ch}}\big(\widehat{g}^{E'},\widehat{g}^E,\widehat{g}^{E''}\big) \\
& = \widetilde{\mathrm{ch}}\big(g^{E'},g^E,g^{E''}\big)
+ \widetilde{\mathrm{ch}}\big(\widehat{g}^{E'},g^{E'}\big)
- \widetilde{\mathrm{ch}}\big(\widehat{g}^E,g^E\big)
+ \widetilde{\mathrm{ch}}\big(\widehat{g}^{E''},g^{E''}\big) \;.
\end{split}
\end{align}
For $a,b>0$,
we have
\begin{equation}
\widetilde{\mathrm{ch}}\big(a g^E,bg^E\big) = \mathrm{ch}\big(E,g^E\big) (\log b - \log a)  \;.
\end{equation}
For $\big(g^E_t\big)_{t>0}$ a family of Hermitian metrics on $E$ converging to $g^E$ as $t\rightarrow 0$,
we have
\begin{equation}
\widetilde{\mathrm{ch}}\big(g^E_t,g^E\big) \rightarrow 0 \hspace{4mm} \text{as } t \rightarrow 0 \;.
\end{equation}
\end{prop}

Let $E^*$ be the dual of $E$.
For $p=0,\cdots,\dim E$ and $s=0,\cdots,p-1$,
set
\begin{align}
\label{eq-Ips}
\begin{split}
I^p_s = \Big\{ u \in \Lambda^pE^* \; & : \; u(v_1,\cdots,v_p) = 0 \\
& \hspace{5mm} \text{for any } v_1,\cdots,v_{s+1} \in E',\; v_{s+2},\cdots,v_p \in E \Big\} \;.
\end{split}
\end{align}
For convenience,
we denote $I^p_p = \Lambda^pE^*$ and $I^p_{-1} = 0$.
We get a filtration
\begin{equation}
\Lambda^pE^* = I^p_p \hookleftarrow I^p_{p-1} \hookleftarrow \cdots \hookleftarrow I^p_{-1} = 0 \;.
\end{equation}
For $r=0,\cdots,\dim E''$ and $s=0,\cdots,\dim E'$,
we denote $E_{r,s} = \Lambda^s{E'}^* \otimes \Lambda^r{E''}^*$.
We have a short exact sequence of holomorphic vector bundles over $S$,
\begin{equation}
0 \rightarrow I^{r+s}_{s-1} \rightarrow I^{r+s}_s \rightarrow E_{r,s} \rightarrow 0 \;.
\end{equation}

Recall that $g^E_\e$ was defined in \eqref{eq-gEe}.
Let $g^{\Lambda^pE^*}_\e$ be the Hermitian metric on $\Lambda^pE^*$ induced by $g^E_\e$.
Let $g^{I^{r+s}_s}_\e$ be the restriction of $g^{\Lambda^pE^*}_\e$ to $I^{r+s}_s$.
Let $g^{E_{r,s}}_\e$ be the quotient metric on $E_{r,s}$ induced by $g^{I^{r+s}_s}_\e$ via the surjection $I^{r+s}_s \rightarrow E_{r,s}$.

Similarly to Proposition \ref{prop-adiabatic-bc},
we have the following proposition.

\begin{prop}
\label{prop4-adiabatic-bc}
As $\e\rightarrow 0$,
\begin{equation}
\label{eq-prop4-adiabatic-bc}
\widetilde{\mathrm{ch}}\big(g^{I^{r+s}_{s-1}}_\e,g^{I^{r+s}_s}_\e,g^{E_{r,s}}_\e\big) \rightarrow 0 \;.
\end{equation}
\end{prop}
\begin{proof}
Let $0 \rightarrow \mathcal{E}' \rightarrow \mathcal{E} \rightarrow \mathcal{E}'' \rightarrow 0$ be as in \eqref{eqe5-pf-prop-adiabatic-bc}.
Let $\mathcal{I}^p_s \subseteq \Lambda^p\mathcal{E}^*$ be as in \eqref{eq-Ips} with $E$ replaced by $\mathcal{E}$ and $E'$ replaced by $\mathcal{E}'$.
We denote $\mathcal{E}_{r,s} = \Lambda^s{\mathcal{E}'}^* \otimes \Lambda^r{\mathcal{E}''}^*$.
We have a short exact sequence of holomorphic vector bundles over $S\times\C$,
\begin{equation}
\label{eqe-pf-prop4-adiabatic-bc}
0 \rightarrow \mathcal{I}^{r+s}_{s-1} \rightarrow \mathcal{I}^{r+s}_s \rightarrow \mathcal{E}_{r,s} \rightarrow 0 \;.
\end{equation}
Proceeding in the same way as in the proof of Proposition \ref{prop-adiabatic-bc} with \eqref{eqe5-pf-prop-adiabatic-bc} replaced by \eqref{eqe-pf-prop4-adiabatic-bc},
we obtain \eqref{eq-prop4-adiabatic-bc}.
This completes the proof.
\end{proof}

\subsection{Quillen metric}
\label{subsect-q}

Let $X$ be an $n$-dimensional compact K{\"a}hler manifold.
Let $E$ be a holomorphic vector bundle over $X$.
Let $\overline{\partial}^E$ be the Dolbeault operator on
\begin{equation}
\Omega^{0,\bullet}(X,E) =
\smooth\big(X,\Lambda^\bullet(\overline{T^*X})\otimes E\big) \;.
\end{equation}
For $q=0,\cdots,n$,
we have $H^q(X,E) = H^q\big(\Omega^{0,\bullet}(X,E),\overline{\partial}^E\big)$.
Set
\begin{equation}
\label{eq-lambda-E}
\lambda(E) = \det H^\bullet(X,E) := \bigotimes_{q=0}^n \Big(\det H^q(X,E)\Big)^{(-1)^q} \;.
\end{equation}

Let $g^{TX}$ be a K{\"a}hler metric on $TX$.
Let $g^E$ be a Hermitian metric on $E$.
Let $\big\langle\cdot,\cdot\big\rangle_{\Lambda^\bullet(\overline{T^*X})\otimes E}$
be the Hermitian product on $\Lambda^\bullet(\overline{T^*X})\otimes E$ induced by $g^{TX}$ and $g^E$.
Let $dv_X$ be the volume form on $X$ induced by $g^{TX}$.
For $s_1,s_2\in \Omega^{0,\bullet}(X,E)$,
set
\begin{equation}
\label{eq-def-L2-metric}
\big\langle s_1,s_2 \big\rangle = (2\pi)^{-n} \int_X
\big\langle s_1,s_2 \big\rangle_{\Lambda^\bullet(\overline{T^*X})\otimes E} dv_X \;,
\end{equation}
which we call the $L^2$-product.

Let $\overline{\partial}^{E,*}$ be the formal adjoint of $\overline{\partial}^E$
with respect to the Hermitian product \eqref{eq-def-L2-metric}.
The Dolbeault Laplacian on $\Omega^{0,\bullet}(X,E)$ is defined by
\begin{equation}
\label{def-Delta-dol}
\Delta^E = \overline{\partial}^E\overline{\partial}^{E,*} + \overline{\partial}^{E,*}\overline{\partial}^E \;.
\end{equation}
Let $\Delta^E_q$ be the restriction of $\Delta^E$ to $\Omega^{0,q}(X,E)$.

By the Hodge theorem,
we have
\begin{equation}
\mathrm{Ker}\big(\Delta^E_q\big) =
\Big\{s\in\Omega^{0,q}(X,E)\;:\;
\overline{\partial}^Es=0\;,\;\overline{\partial}^{E,*}s=0\Big\} \;.
\end{equation}
Moreover,
the following map is bijective,
\begin{align}
\label{eq-iso-hodge}
\begin{split}
\mathrm{Ker}\big(\Delta^E_q\big) & \rightarrow H^q(X,E) \\
s & \mapsto [s] \;.
\end{split}
\end{align}
Let $\big|\cdot\big|_{\lambda(E)}$ be the metric on $\lambda(E)$
induced by the metric \eqref{eq-def-L2-metric} via \eqref{eq-lambda-E} and \eqref{eq-iso-hodge}.

Let $\Sp(\Delta^E_q)$ be the spectrum of $\Delta^E_q$,
which is a multiset.
For $z\in\C$ with $\mathrm{Re}(z)>n$,
set
\begin{equation}
\label{eq-def-theta}
\theta(z) = \sum_{q=1}^n (-1)^{q+1}q \sum_{\lambda\in\Sp(\Delta^E_q),\lambda\neq 0}  \lambda^{-z} \;.
\end{equation}
By \cite{se},
the function $\theta(z)$ extends to a meromorphic function of $z\in\C$,
which is holomorphic at $z=0$.

The following definition is due to Quillen \cite{q}, Bismut, Gillet and Soul{\'e} \cite[Section 1d)]{bgs3}.

\begin{defn}
\label{def-q}
The Quillen metric on $\lambda(E)$ is defined by
\begin{equation}
\big\lVert\cdot\big\rVert_{\lambda(E)} =
\exp\Big(\frac{1}{2}\theta'(0)\Big)
\big|\cdot\big|_{\lambda(E)} \;.
\end{equation}
\end{defn}

\begin{rem}
\label{rem-q}
Denote $\chi(X,E) = \sum_{q=0}^n (-1)^q \dim H^q(X,E)$.
For $a>0$,
if we replace $g^E$ by $ag^E$,
then $\big\lVert\cdot\big\rVert_{\lambda(E)}$ is replaced by $a^{\chi(X,E)/2}\big\lVert\cdot\big\rVert_{\lambda(E)}$.
\end{rem}

\subsection{Analytic torsion form}
\label{subsect-tf}

Let $\pi: X \rightarrow Y$ be a holomorphic submersion between K{\"a}hler manifolds.
For $y\in Y$,
we denote $Z_y = \pi^{-1}(y)$.
We may omit the index $y$ as long as there is no confusion.
We assume that $Z$ is compact.

Let $E$ be a holomorphic vector bundle over $X$.
Let $R^\bullet\pi_*E$ be the derived direct image of $E$,
which is a graded analytic coherent sheaf on $Y$.
We assume that $R^\bullet\pi_*E$ is a graded holomorphic vector bundle.
Let $H^\bullet(Z,E)$ be the fiberwise cohomology.
More precisely,
its fiber at $y\in Y$ is given by $H^\bullet\big(Z_y,E\big|_{Z_y}\big)$.
We have a canonical identification $R^\bullet\pi_*E = H^\bullet(Z,E)$.
We have the Grothendieck-Riemann-Roch formula,
\begin{equation}
\label{eq-grr}
\mathrm{ch}(H^\bullet(Z,E)) = \int_Z \mathrm{Td}(TZ) \mathrm{ch}(E) \in H^\mathrm{even}_\mathrm{dR}(Y) \;.
\end{equation}

Let $\omega\in\Omega^{1,1}(X)$ be a K{\"a}hler form.
Let $g^{TZ}$ be the Hermitian metric on $TZ$ associated with $\omega$.
Let $g^E$ be a Hermitian metric on $E$.
Let $g^{H^\bullet(Z,E)}$ be the $L^2$-metric on $H^\bullet(Z,E)$ associated with $g^{TZ}$ and $g^E$.

We will use the notations in \eqref{eq-def-QQ}.
Let $\mathrm{ch}\big(H^\bullet(Z,E),g^{H^\bullet(Z,E)}\big) \in Q^Y$ be the Chern characteristic form of $\big(H^\bullet(Z,E),g^{H^\bullet(Z,E)}\big)$.
We introduce $\mathrm{Td}\big(TZ,g^{TZ}\big) \in Q^X$ and $\mathrm{ch}\big(E,g^E\big) \in Q^X$ in the same way.

Bismut and K{\"o}hler \cite[Definition 3.8]{bk} defined the analytic torsion forms.
The analytic torsion form associated with $\big(\pi: X \rightarrow Y, \omega, E, g^E\big)$ is a differential form on $Y$,
which we denote by $T\big(\omega,g^E\big)$.
Moreover,
we have
\begin{equation}
\label{eq-def-tf}
T\big(\omega,g^E\big) \in Q^Y \;.
\end{equation}
We sometimes view $T\big(\omega,g^E\big)$ as an element in $Q^Y/Q^{Y,0}$.
By \cite[Theorem 3.9]{bk},
we have
\begin{equation}
\label{eq-tf}
\frac{\overline{\partial}\partial}{2\pi i} T\big(\omega,g^E\big) =
\mathrm{ch}\big(H^\bullet(Z,E),g^{H^\bullet(Z,E)}\big) - \int_Z \mathrm{Td}\big(TZ,g^{TZ}\big) \mathrm{ch}\big(E,g^E\big) \;.
\end{equation}
The identity \eqref{eq-tf} is a refinement of the Grothendieck-Riemann-Roch formula \eqref{eq-grr}.

For $y\in Y$,
let $\theta_y(z)$ be as in \eqref{eq-def-theta} with $\big(X,g^{TX},E,g^E\big)$ replaced by $\big(Z_y,g^{TZ_y},E\big|_{Z_y},g^E\big|_{Z_y}\big)$.
Let $\theta'(0)$ be the function $y \mapsto \theta_y'(0)$ on $Y$.
By the construction of the analytic torsion forms,
we have
\begin{equation}
\label{eq-tf-deg0}
\big\{T\big(\omega,g^E\big)\big\}^{(0,0)} = \theta'(0) \in \smooth(Y)\;,
\end{equation}
where $\big\{\cdot\big\}^{(0,0)}$ means the component of degree $(0,0)$.

Let $F$ be a holomorphic vector bundle over $Y$.
Let $\pi^*F$ be its pull-back via $\pi$,
which is a holomorphic vector bundle over $X$.
Let $g^F$ be a Hermitian metric on $F$.
Let $g^{E\otimes\pi^*F}$ be the Hermitian metric on $E\otimes\pi^*F$ induced by $g^E$ and $g^F$.
Let
\begin{equation}
T\big(\omega,g^{E\otimes\pi^*F}\big) \in Q^Y
\end{equation}
be the analytic torsion form associated with $\big(\pi: X \rightarrow Y, \omega, E\otimes\pi^*F, g^{E\otimes\pi^*F}\big)$.

The following proposition is a direct consequence of the construction of the analytic torsion forms.

\begin{prop}
\label{prop-tf-mult}
The following identity holds in $Q^Y/Q^{Y,0}$,
\begin{equation}
T\big(\omega,g^{E\otimes\pi^*F}\big) =  \mathrm{ch}\big(F,g^F\big) T\big(\omega,g^E\big) \;.
\end{equation}
\end{prop}

For $p=0,\cdots,\dim Z$,
let $g^{\Lambda^p(T^*Z)}$ be the metric on $\Lambda^p(T^*Z)$ induced by $g^{TZ}$.
Let
\begin{equation}
T\big(\omega,g^{\Lambda^p(T^*Z)}\big) \in Q^Y
\end{equation}
be the analytic torsion form associated with $\big(\pi: X \rightarrow Y, \omega, \Lambda^p(T^*Z), g^{\Lambda^p(T^*Z)}\big)$.

The following theorem is due to Bismut \cite[Theorem 4.15]{b04}.

\begin{thm}
\label{thm-tf-vanishing}
The following identity holds in $Q^Y/Q^{Y,0}$,
\begin{equation}
\sum_{p=0}^{\dim Z} (-1)^p T\big(\omega,g^{\Lambda^p(T^*Z)}\big) = 0 \;.
\end{equation}
\end{thm}

\subsection{Properties of the Quillen metric}

In this subsection,
we state several results describing the behavior of the Quillen metric under  submersion, resolution, immersion and blow-up.

\noindent\textbf{Submersion.}
Let $\pi: X \rightarrow Y$, $Z$, $E$ and $H^\bullet(Z,E)$ be as in \textsection \ref{subsect-tf}.
We assume that $X$ and $Y$ are compact.
We further assume that
\begin{equation}
\label{eq-H-sub}
H^q(X,E) = \bigoplus_{j+k = q} H^j\big(Y,H^k(Z,E)\big) \hspace{4mm} \text{for } q = 0,\cdots,\dim X \;.
\end{equation}
We denote
\begin{align}
\label{eq-det-sub}
\begin{split}
\det H^\bullet\big(Y,H^\bullet(Z,E)\big)
& = \bigotimes_{k=0}^{\dim Z} \Big(\det H^\bullet\big(Y,H^k(Z,E)\big)\Big)^{(-1)^k} \\
& = \bigotimes_{j=0}^{\dim Y} \bigotimes_{k=0}^{\dim Z} \Big(\det H^j\big(Y,H^k(Z,E)\big)\Big)^{(-1)^{j+k}} \;.
\end{split}
\end{align}
Let
\begin{equation}
\sigma \in \det H^\bullet(X,E) \otimes \Big(\det H^\bullet\big(Y,H^\bullet(Z,E)\big)\Big)^{-1}
\end{equation}
be the canonical section induced by \eqref{eq-H-sub}.

Let $\omega_X\in\Omega^{1,1}(X)$ and $\omega_Y\in\Omega^{1,1}(Y)$ be K{\"a}hler forms.
For $\e>0$,
set
\begin{equation}
\omega_\e = \omega_X + \frac{1}{\e}\pi^*\omega_Y \;.
\end{equation}
Let $g^E$ be a Hermitian metric on $E$.

Let $g^{TX}_\e$ be the metric on $TX$ associated with $\omega_\e$.
Let
\begin{equation}
\label{eq1-Q-sub}
\big\lVert\cdot\big\rVert_{\det H^\bullet(X,E),\e}
\end{equation}
be the Quillen metric on $\det H^\bullet(X,E)$ associated with $g^{TX}_\e$ and $g^E$.
Let $g^{TY}$ be the metric on $TY$ associated with $\omega_Y$.
Let $g^{TZ}$ be the metric on $TZ$ associated with $\omega_X\big|_Z$.
Let $g^{H^\bullet(Z,E)}$ be the $L^2$-metric on $H^\bullet(Z,E)$ associated with $g^{TZ}$ and $g^E$.
For $k=0,\cdots,\dim Z$,
let
\begin{equation}
\label{eq2-Q-sub}
\big\lVert\cdot\big\rVert_{\det H^\bullet(Y,H^k(Z,E))}
\end{equation}
be the Quillen metric on $\det H^\bullet\big(Y,H^k(Z,E)\big)$ associated with $g^{TY}$ and $g^{H^k(Z,E)}$.
Let
\begin{equation}
\label{eq3-Q-sub}
\big\lVert\cdot\big\rVert_{\det H^\bullet(Y,H^\bullet(Z,E))}
\end{equation}
be the metric on $\det H^\bullet\big(Y,H^\bullet(Z,E)\big)$ induced by the Quillen metrics \eqref{eq2-Q-sub} via \eqref{eq-det-sub}.
Let $\big\lVert\sigma\big\rVert_\e$ be the norm of $\sigma$ with respect to the metrics \eqref{eq1-Q-sub} and \eqref{eq3-Q-sub}.

We will use the notations in \eqref{eq-def-QQ}.
Let $\mathrm{Td}\big(TY,g^{TY}\big) \in Q^Y$ be the Todd form of $\big(TY,g^{TY}\big)$.
Let
\begin{equation}
T\big(\omega,g^E\big) \in Q^Y
\end{equation}
be the analytic torsion form (see \textsection \ref{subsect-tf}) associated with $\big(\pi: X \rightarrow Y, \omega_X, E, g^E\big)$.

Recall that $\mathrm{Td}'(\cdot)$ was defined by \eqref{eq-def-tdprim}.

The following theorem is due to Berthomieu and Bismut \cite[Theorem 3.2]{bb}.

\begin{thm}
\label{thm-sub}
As $\e\rightarrow 0$,
\begin{equation}
\log \big\lVert\sigma\big\rVert^2_\e + \int_Y \mathrm{Td}'(TY) \int_Z \mathrm{Td}(TZ)\mathrm{ch}(E) \log \e
\rightarrow \int_Y \mathrm{Td}\big(TY,g^{TY}\big)T\big(\omega,g^E\big) \;.
\end{equation}
\end{thm}

\noindent\textbf{Resolution.}
Let $X$ be a compact K{\"a}hler manifold.
Let
\begin{equation}
\label{eq-E-resol}
0 \rightarrow E^0 \rightarrow E^1 \rightarrow E^2 \rightarrow 0
\end{equation}
be a short exact sequence of holomorphic vector bundles over $X$.
Let
\begin{equation}
\sigma \in \bigotimes_{k=0}^2 \Big(\det H^\bullet(X,E^k)\Big)^{(-1)^{k+1}}
\end{equation}
be the canonical section induced by the long exact sequence induced by \eqref{eq-E-resol}.

Let $g^{TX}$ be a K{\"a}hler metric on $TX$.
For $k=0,1,2$,
let $g^{E^k}$ be a Hermitian metric on $E^k$.
Let
\begin{equation}
\label{eq-Q-resol}
\big\lVert\cdot\big\rVert_{\det H^\bullet(X,E^k)}
\end{equation}
be the Quillen metric on $\det H^\bullet(X,E^k)$ associated with $g^{TX}$ and $E^k$.
Let $\big\lVert\sigma\big\rVert$ be the norm of $\sigma$ with respect to the metrics \eqref{eq-Q-resol}.

We will use the notations in \eqref{eq-def-QQ}.
Let $\mathrm{Td}\big(TX,g^{TX}\big) \in Q^X$ be the Todd form of $\big(TX,g^{TX}\big)$.
Let $\mathrm{ch}\big(E^k,g^{E^k}\big) \in Q^X$ be the Chern characteristic form of $\big(E^k,g^{E^k}\big)$.
Let
\begin{equation}
\widetilde{\mathrm{ch}}\big(g^{E^\bullet}\big) \in Q^X/Q^{X,0}
\end{equation}
be the Bott-Chern form \cite[Section 1f)]{bgs1} such that
\begin{equation}
\frac{\overline{\partial}\partial}{2\pi i} \widetilde{\mathrm{ch}}\big(g^{E^\bullet}\big) =
\sum_{k=0}^2 (-1)^k \mathrm{ch}\big(E^k,g^{E^k}\big) \;.
\end{equation}

The following theorem is a special case of the immersion formula due to Bismut and Lebeau \cite[Theorem 0.1]{ble}.

\begin{thm}
\label{thm-resol}
The following identity holds,
\begin{equation}
\log \big\lVert\sigma\big\rVert^2 = \int_X \mathrm{Td}\big(TX,g^{TX}\big) \widetilde{\mathrm{ch}}\big(g^{E^\bullet}\big) \;.
\end{equation}
\end{thm}

\noindent\textbf{Immersion.}
Let $X$ be a compact K{\"a}hler manifold.
Let $Y \subseteq X$ be a complex submanifold of codimension $1$.
Let $i: Y \hookrightarrow X$ be the canonical embedding.
Let $F$ be a holomorphic vector bundle over $Y$.
Let $v: E_1 \rightarrow E_0$ be a map between holomorphic vector bundles over $X$ which,
together with a restriction map $r: E_0\big|_Y \rightarrow F$,
provides a resolution of $i_*\mathscr{O}_Y(F)$.
More precisely,
we have an exact sequence of analytic coherent sheaves on $X$,
\begin{equation}
\label{eq-F-resol}
0 \rightarrow \mathscr{O}_X(E_1) \xrightarrow{v} \mathscr{O}_X(E_0) \xrightarrow{r} i_*\mathscr{O}_Y(F) \rightarrow 0 \;.
\end{equation}
Let
\begin{equation}
\sigma \in \Big(\det H^\bullet(X,E_1)\Big)^{-1} \otimes \det H^\bullet(X,E_0) \otimes \Big(\det H^\bullet(Y,F)\Big)^{-1}
\end{equation}
be the canonical section induced by the long exact sequence induced by \eqref{eq-F-resol}.

Let $\omega\in\Omega^{1,1}(X)$ be a K{\"a}hler form.
For $k=0,1$,
let $g^{E_k}$ be a Hermitian metric on $E_k$.
Let $g^F$ be a Hermitian metric on $F$.
Assume that there is an open neighborhood $Y \subseteq U \subseteq X$ such that $v\big|_{X\backslash U}$ is isometric, i.e.,
\begin{equation}
g^{E_1}\big|_{X\backslash U} = v^*g^{E_0}\big|_{X\backslash U} \;.
\end{equation}

Let $g^{TX}$ be the metric on $TX$ associated with $\omega$.
For $k=0,1$,
let
\begin{equation}
\label{eq1-Q-im}
\big\lVert\cdot\big\rVert_{\det H^\bullet(X,E_k)}
\end{equation}
be the Quillen metric on $\det H^\bullet(X,E_k)$ associated with $g^{TX}$ and $g^{E_k}$.
Let $g^{TY}$ be the metric on $TY$ associated with $\omega\big|_Y$.
Let
\begin{equation}
\label{eq2-Q-im}
\big\lVert\cdot\big\rVert_{\det H^\bullet(Y,F)}
\end{equation}
be the Quillen metric on $\det H^\bullet(Y,F)$ associated with $g^{TY}$ and $g^F$.
Let $\big\lVert\sigma\big\rVert$ be the norm of $\sigma$ with respect to the metrics \eqref{eq1-Q-im} and \eqref{eq2-Q-im}.

The following theorem is a direct consequence of
the immersion formula due to Bismut and Lebeau \cite[Theorem 0.1]{ble} and the anomaly formula due to Bismut, Gillet and Soul{\'e} \cite[Theorem 1.23]{bgs3}.

\begin{thm}
\label{thm-im}
We have
\begin{equation}
\log \big\lVert\sigma\big\rVert^2 = \alpha\big(U,\omega\big|_U,v\big|_U,g^{E_\bullet}\big|_U,r,g^F\big)
\end{equation}
where $\alpha\big(U,\omega\big|_U,v\big|_U,r\big|_U,g^{E_\bullet},g^F\big)$ is a real number determined by
\begin{equation}
U \;,\hspace{2.5mm}
\omega\big|_U \;,\hspace{5mm}
v\big|_U: E_1\big|_U \rightarrow E_0\big|_U \;,\hspace{2.5mm}
g^{E_\bullet}\big|_U \;,\hspace{2.5mm}
r: E_0\big|_Y \rightarrow F \;,\hspace{2.5mm}
g^F \;.
\end{equation}
More precisely,
given
\begin{equation}
\widetilde{Y} \subseteq \widetilde{U} \subseteq \widetilde{X} \;,\hspace{2.5mm}
\widetilde{\omega} \;,\hspace{2.5mm}
\widetilde{v}: \widetilde{E}_1 \rightarrow \widetilde{E}_0 \;,\hspace{2.5mm}
\widetilde{r}: \widetilde{E}_0 \big|_{\widetilde{Y}} \rightarrow \widetilde{F} \;,\hspace{2.5mm}
g^{\widetilde{E}_\bullet} \;,\hspace{2.5mm}
g^{\widetilde{F}}
\end{equation}
satisfying the same properties that
\begin{equation}
Y \subseteq U \subseteq X \;,\hspace{2.5mm}
\omega \;,\hspace{2.5mm}
v: E_1 \rightarrow E_0 \;,\hspace{2.5mm}
r: E_0 \big|_Y \rightarrow F \;,\hspace{2.5mm}
g^{E_\bullet} \;,\hspace{2.5mm}
g^F
\end{equation}
satisfy,
if there is a biholomorphic map $U \rightarrow \widetilde{U}$ inducing an isomorphism between the restrictions of the data above to $U$ and $\widetilde{U}$,
then
\begin{equation}
\log \big\lVert\sigma\big\rVert^2 = \log \big\lVert\widetilde{\sigma}\big\rVert^2 \;,
\end{equation}
where
\begin{equation}
\widetilde{\sigma} \in
\Big(\det H^\bullet(\widetilde{X},\widetilde{E}_1)\Big)^{-1}
\otimes \det H^\bullet(\widetilde{X},\widetilde{E}_0)
\otimes \Big(\det H^\bullet(\widetilde{Y},\widetilde{F})\Big)^{-1}
\end{equation}
is the canonical section,
and $\big\lVert\widetilde{\sigma}\big\rVert$ is its norm with respect to the Quillen metrics.
\end{thm}

\begin{rem}
\label{rk-im}
The real number $\alpha\big(U,\omega\big|_U,v\big|_U,r\big|_U,g^{E_\bullet},g^F\big)$ depends continuously on the input data.
\end{rem}

\noindent\textbf{Blow-up.}
Let $X$ be a compact K{\"a}hler manifold.
Let $Y \subseteq X$ be a complex submanifold of codimension $r \geqslant 2$.
Let $f: X' \rightarrow X$ be the blow-up along $Y$.
Let $E$ be a holomorphic vector bundle over $X$.
Let $f^*E$ be the pull-back of $E$ via $f$,
which is a holomorphic vector bundle over $X'$.
Applying spectral sequence,
we get a canonical identification
\begin{equation}
\label{eq-bl}
H^\bullet(X',f^*E) = H^\bullet(X,E) \;.
\end{equation}
Let
\begin{equation}
\sigma \in \Big(\det H^\bullet(X,E)\Big)^{-1} \otimes \det H^\bullet(X',f^*E)
\end{equation}
be the canonical section induced by \eqref{eq-bl}.

Let $\omega\in\Omega^{1,1}(X)$ and $\omega'\in\Omega^{1,1}(X')$ be K{\"a}hler forms.
Assume that there are open neighborhoods $Y \subseteq U \subseteq X$ and $f^{-1}(Y) \subseteq U' \subseteq X'$ such that
\begin{equation}
f^{-1}(U) = U' \;,\hspace{5mm}
f^*\big( \omega\big|_{X\backslash U} \big) = \omega'\big|_{X'\backslash U'} \;.
\end{equation}
Let $g^E$ be a Hermitian metric on $E$.

Let $g^{TX}$ be the metric on $TX$ associated with $\omega$.
Let
\begin{equation}
\label{eq1-Q-bl}
\big\lVert\cdot\big\rVert_{\det H^\bullet(X,E)}
\end{equation}
be the Quillen metric on $\det H^\bullet(X,E)$ associated with $g^{TX}$ and $g^E$.
Let $g^{TX'}$ be the metric on $TX'$ associated with $\omega'$.
Let
\begin{equation}
\label{eq2-Q-bl}
\big\lVert\cdot\big\rVert_{\det H^\bullet(X',f^*E)}
\end{equation}
be the Quillen metric on $\det H^\bullet(X',f^*E)$ associated with $g^{TX'}$ and $f^*g^E$.
Let $\big\lVert\sigma\big\rVert$ be the norm of $\sigma$ with respect to the metrics \eqref{eq1-Q-bl} and \eqref{eq2-Q-bl}.

The following theorem is a direct consequence of the blow-up formula due to Bismut \cite[Theorem 8.10]{b97}.

\begin{thm}
\label{thm-bl}
We have
\begin{equation}
\log \big\lVert\sigma\big\rVert^2 = \alpha\big(U,\omega\big|_U,U',\omega'\big|_{U'},E\big|_U,g^E\big|_U\big) \;,
\end{equation}
where $\alpha\big(U,\omega\big|_U,U',\omega'\big|_{U'},E\big|_U,g^E\big|_U\big)$ is a real number determined by
\begin{equation}
U \;,\hspace{2.5mm}
\omega\big|_U \;,\hspace{2.5mm}
U' \;,\hspace{2.5mm}
\omega'\big|_{U'} \;,\hspace{2.5mm}
E\big|_U \;,\hspace{2.5mm}
g^E\big|_U \;.
\end{equation}
\end{thm}

\begin{rem}
\label{rk-bl}
The real number $\alpha\big(U,\omega\big|_U,U',\omega'\big|_{U'},E\big|_U,g^E\big|_U\big)$ depends continuously on the input data.
\end{rem}

\subsection{Topological torsion and BCOV torsion}

Let $X$ be an $n$-dimensional compact K{\"a}hler manifold.
For $p=0,\cdots,n$,
set
\begin{equation}
\label{eq-def-lambda-p}
\lambda_p(X) = \det H^{p,\bullet}(X) := \bigotimes_{q=0}^n \Big(\det H^{p,q}(X)\Big)^{(-1)^q} \;.
\end{equation}
Set
\begin{equation}
\label{eq-def-eta}
\eta(X) = \det H^\bullet_\mathrm{dR}(X) := \bigotimes_{k=0}^{2n} \Big(\det H^k_\mathrm{dR}(X)\Big)^{(-1)^k} = \bigotimes_{p=0}^n \Big(\lambda_p(X)\Big)^{(-1)^p} \;.
\end{equation}
Set
\begin{align}
\label{eq-def-lambda}
\begin{split}
\lambda(X) &
= \bigotimes_{0\leqslant p,q\leqslant n} \Big( \det H^{p,q}(X) \Big)^{(-1)^{p+q}p}
= \bigotimes_{p=1}^n \Big(\lambda_p(X)\Big)^{(-1)^pp} \;,\\
\lambda_\mathrm{dR}(X) &
= \bigotimes_{k=1}^{2n} \Big(\det H^k_\mathrm{dR}(X)\Big)^{(-1)^kk}
=  \lambda(X) \otimes \overline{\lambda(X)} \;.
\end{split}
\end{align}
The identities in \eqref{eq-def-lambda} appeared in \cite{k14}.
They were applied to the theory of BCOV invariant by Eriksson, Freixas i Montplet and Mourougane \cite{efm}.

For $\mathbb{A} = \Z,\R,\C$,
we denote by $H^\bullet_\mathrm{Sing}(X,\mathbb{A})$
the singular cohomology of $X$ with coefficients in $\mathbb{A}$.
For $k=0,\cdots,2n$,
let
\begin{equation}
\sigma_{k,1},\cdots,\sigma_{k,b_k}
\in \mathrm{Im}\big(H^k_\mathrm{Sing}(X,\Z) \rightarrow H^k_\mathrm{Sing}(X,\R)\big)
\end{equation}
be a basis of the lattice.
We fix a square root of $i$.
In what follows,
the choice of square root will be irrelevant.
We identify $H^k_\mathrm{dR}(X)$ with $H^k_\mathrm{Sing}(X,\C)$ as follows,
\begin{align}
\label{eq-sing-dr}
\begin{split}
H^k_\mathrm{dR}(X) & \rightarrow H^k_\mathrm{Sing}(X,\C) \\
[\alpha] & \mapsto \Big[\mathfrak{a} \mapsto \big(2\pi i\big)^{-k/2} \int_\mathfrak{a}\alpha\Big]\;,
\end{split}
\end{align}
where $\alpha$ is a closed $k$-form on $X$
and $\mathfrak{a}$ is a $k$-chain in $X$.
Then $\sigma_{k,1},\cdots,\sigma_{k,b_k}$ form a basis of $H^k_\mathrm{dR}(X)$.
Set
\begin{align}
\label{eq-def-epsilon}
\begin{split}
& \sigma_k = \sigma_{k,1}\wedge\cdots\wedge\sigma_{k,b_k} \in \det H^k_\mathrm{dR}(X) \;,\\
& \epsilon_X = \bigotimes_{k=0}^{2n} \sigma_k^{(-1)^k} \in \eta(X) \;,\hspace{4mm}
\sigma_X = \bigotimes_{k=1}^{2n} \sigma_k^{(-1)^kk} \in \lambda_\mathrm{dR}(X) \;,
\end{split}
\end{align}
which are well-defined up to $\pm 1$.

Let $\omega$ be a K{\"a}hler form on $X$.
Let $\big\lVert\cdot\big\rVert_{\lambda_p(X),\omega}$
be the Quillen metric on $\lambda_p(X)$ associated with $\omega$.
Let $\big\lVert\cdot\big\rVert_{\eta(X)}$ be the metric on $\eta(X)$
induced by $\big\lVert\cdot\big\rVert_{\lambda_p(X),\omega}$ via \eqref{eq-def-eta}.
The same calculation as in \cite[Theorem 2.1]{z}
together with the first identity in Proposition \ref{prop-total-class}
shows that $\big\lVert\cdot\big\rVert_{\eta(X)}$ is independent of $\omega$.

\begin{defn}
\label{def-tau-top}
We define
\begin{equation}
\tau_\mathrm{top}(X) = \log \big\lVert\epsilon_X\big\rVert_{\eta(X)} \;.
\end{equation}
\end{defn}

Let $\big\lVert\cdot\big\rVert_{\lambda(X),\omega}$
be the metric on $\lambda(X)$ induced by $\big\lVert\cdot\big\rVert_{\lambda_p(X),\omega}$ via the first identity in \eqref{eq-def-lambda}.
Let $\big\lVert\cdot\big\rVert_{\lambda_\mathrm{dR}(X),\omega}$
be the metric on $\lambda_\mathrm{dR}(X)$ induced by $\big\lVert\cdot\big\rVert_{\lambda(X),\omega}$ via the second identity in \eqref{eq-def-lambda}.

\begin{defn}
\label{def-bcov-torsion}
We define
\begin{equation}
\tau_\mathrm{BCOV}(X,\omega) = \log \big\lVert\sigma_X\big\rVert_{\lambda_\mathrm{dR}(X),\omega} \;.
\end{equation}
\end{defn}

For $p=0,\cdots,n$,
let $g^{\Lambda^p(T^*X)}_\omega$ be the metric on $\Lambda^p(T^*X)$ induced by $\omega$.
Let $g^{\Omega^{p,q}(X)}_\omega$ be the $L^2$-metric on $\Omega^{p,q}(X)$.
More precisely,
$g^{\Omega^{p,q}(X)}_\omega$ is defined by \eqref{eq-def-L2-metric} with $(E,g^E)$ replaced by $(\Lambda^p(T^*X),g^{\Lambda^p(T^*X)}_\omega)$.
Let $g^{H^{p,q}(X)}_\omega$ be the $L^2$-metric on $H^{p,q}(X)$.
More precisely,
$g^{H^{p,q}(X)}_\omega$ is induced by $g^{\Omega^{p,q}(X)}_\omega$ via the Hodge theorem.
Let $\big|\cdot\big|_{\eta(X),\omega}$ be the metric on $\eta(X)$ induced by $\big(g^{H^{p,q}(X)}_\omega\big)_{0\leqslant p,q\leqslant n}$ via \eqref{eq-def-lambda-p} and \eqref{eq-def-eta}.

\begin{prop}
\label{prop-tau-top}
The following identity holds,
\begin{equation}
\label{eq-prop-tau-top}
\tau_\mathrm{top}(X) =
\log \big|\epsilon_X\big|_{\eta(X),\omega} = 0 \;.
\end{equation}
\end{prop}
\begin{proof}
Let $\Delta_p$ be as in \eqref{def-Delta-dol}
with $\big( \Omega^{0,\bullet}(X,E), \overline{\partial}^E, g^E \big)$
replaced by $\big( \Omega^{p,\bullet}(X), \overline{\partial}, g^{\Lambda^p(T^*X)}_\omega \big)$.
Let $\Delta_{p,q}$ be the restriction of $\Delta_p$ to $\Omega^{p,q}(X)$.
Let $\theta_p(z)$ be as in \eqref{eq-def-theta} with $\Delta^E_q$ replaced by $\Delta^E_{p,q}$.
By Definition \ref{def-q}, \ref{def-tau-top},
the first equality in \eqref{eq-prop-tau-top} is equivalent to
\begin{equation}
\label{eq1-pf-prop-tau-top}
\sum_{p=0}^n (-1)^p \theta_p'(0) = 0 \;,
\end{equation}
which was indicated in \cite[page 1304]{b04}.

Denote by $\mathrm{covol}\big(H^k_\mathrm{Sing}(X,\Z),\omega\big)$ the covolume of $\mathrm{Im}\big(H^k_\mathrm{Sing}(X,\Z) \rightarrow H^k_\mathrm{Sing}(X,\R)\big)$
with respect to the metric induced by $\bigoplus_{p+q=k} g^{H^{p,q}(X)}_\omega$ via \eqref{eq-sing-dr}.
We have
\begin{equation}
\label{eq21-pf-prop-tau-top}
\big|\epsilon_X\big|_{\eta(X),\omega} =
\prod_{k=0}^{2n} \Big( \mathrm{covol}\big(H^k_\mathrm{Sing}(X,\Z),\omega\big) \Big)^{(-1)^k} \;.
\end{equation}
On the other hand,
by \cite[Remark 5.5(ii)]{efm},
we have
\begin{equation}
\label{eq22-pf-prop-tau-top}
\mathrm{covol}\big(H^k_\mathrm{Sing}(X,\Z),\omega\big)
\mathrm{covol}\big(H^{2n-k}_\mathrm{Sing}(X,\Z),\omega\big) = 1 \;.
\end{equation}
Here we remark that,
due to the normalization in \eqref{eq-def-L2-metric} and \eqref{eq-sing-dr},
the covolume in the sense of \cite[Remark 5.5(ii)]{efm} equals
$(2\pi)^{(n-k)b_k/2}\mathrm{covol}\big(H^k_\mathrm{Sing}(X,\Z),\omega\big)$,
where $b_k$ is the $k$-th Betti number of $X$.
From \eqref{eq21-pf-prop-tau-top} and \eqref{eq22-pf-prop-tau-top},
we obtain $\big|\epsilon_X\big|_{\eta(X),\omega} = 1$,
which is equivalent to the second equality in \eqref{eq-prop-tau-top}.
This completes the proof.
\end{proof}

\section{Several properties of the BCOV torsion}

\subsection{K{\"a}hler metric on projective bundle}
\label{subsect-km}

Let $Y$ be an $m$-dimensional compact K{\"a}hler manifold.
Let $N$ be a holomorphic vector bundle over $Y$ of rank $n$.
Let $\mathbb{1}$ be the trivial line bundle over $Y$.
Set
\begin{equation}
X = \mathbb{P}(N\oplus\mathbb{1}) \;.
\end{equation}
Let $\pi: X \rightarrow Y$ be the canonical projection.
For $y\in Y$,
we denote $Z_y = \pi^{-1}(y)$,
which is isomorphic to $\CP^n$.
Let $\omega_{\CP^n}$ be the K{\"a}hler form on $\CP^n$ associated with the Fubini-Study metric.

\begin{lemme}
\label{lem-kahler-bundle}
There exists a K{\"a}hler form $\omega$ on $X$ such that for any $y\in Y$,
there exists an isomorphism $\phi_y: \CP^n \rightarrow Z_y$ such that $\phi_y^*\big(\omega\big|_{Z_y}\big) = \omega_{\CP^n}$.
\end{lemme}
\begin{proof}
We refer the reader to the proof of \cite[Proposition 3.18]{v}.
\end{proof}

Let $s\in\{1,\cdots,n\}$.
We assume that there are holomorphic line bundles $L_1,\cdots,L_s$ over $Y$
together with a surjection between holomorphic vector bundles,
\begin{equation}
\label{eq-N-L}
N \rightarrow L_1 \oplus \cdots \oplus L_s \;.
\end{equation}
For $k=1,\cdots,s$,
let $N \rightarrow L_k$
be the composition of \eqref{eq-N-L} and the canonical projection $L_1 \oplus \cdots \oplus L_s \rightarrow L_k$.
Set
\begin{equation}
N_k = \mathrm{Ker}\big(N\rightarrow L_k\big) \subseteq N \;,\hspace{4mm}
X_k = \mathbb{P}(N_k\oplus\mathbb{1}) \subseteq X \;,\hspace{4mm}
X_0 = \mathbb{P}(N) \subseteq X \;.
\end{equation}

Let $[\xi_0:\cdots:\xi_n]$ be homogenous coordinates on $\CP^n$.
For $k=0,\cdots,n$,
we denote
$H_k = \big\{\xi_k = 0 \big\} \subseteq \CP^n$.

\begin{lemme}
\label{lem2-kahler-bundle}
There exists a K{\"a}hler form $\omega$ on $X$ such that for any $y\in Y$,
there exists an isomorphism $\phi_y: \CP^n \rightarrow Z_y$ such that
$\phi_y^*\big(\omega\big|_{Z_y}\big) = \omega_{\CP^n}$ and $\phi_y^{-1}\big(X_k \cap Z_y\big) = H_k$ for $k=0,\cdots,s$.
\end{lemme}
\begin{proof}
Let $N^*$ be the dual of $N$.
We have $L_1^{-1}\oplus\cdots\oplus L_s^{-1} \hookrightarrow N^*$.
Let $g^{N^*}$ be a Hermitian metric on $N^*$ such that $L_1^{-1},\cdots,L_s^{-1}\subseteq N^*$ are mutually orthogonal.
Let $g^N$ be the dual metric on $N$.
Now,
proceeding in the same way as in the proof of \cite[Proposition 3.18]{v},
we obtain $\omega$ satisfying the desired properties.
This completes the proof.
\end{proof}

\subsection{Behavior under adiabatic limit}

We will use the notations in \textsection \ref{subsect-km}.
By Lemma \ref{lem-kahler-bundle},
there exists a K{\"a}hler form $\omega_X$ on $X$
such that for any $y\in Y$,
there exists an isomorphism $\phi_y: \CP^n\rightarrow Z_y$ such that
\begin{equation}
\label{eq-hyp-isometry}
\phi_y^*\big(\omega_X\big|_{Z_y}\big) = \omega_{\CP^n} \;.
\end{equation}
Let $\omega_{Z_y} = \omega_X\big|_{Z_y}$.
Note that $\big(Z_y,\omega_{Z_y}\big)_{y\in Y}$ are mutually isometric,
we will omit the index $y$ as long as there is no confusion.
Let $\omega_Y$ be a K{\"a}hler form on $Y$.
For $\e>0$,
set
\begin{equation}
\label{eq-omega-e}
\omega_\e = \omega_X + \frac{1}{\e} \pi^*\omega_Y \;.
\end{equation}

We denote
\begin{equation}
(c_1c_{m-1})(Y) = \int_Y c_1(TY)c_{m-1}(TY) \;.
\end{equation}
Let $\chi(\cdot)$ be the topological Euler characteristic.
Recall that $\tau_\mathrm{BCOV}(\cdot,\cdot)$ was defined in Definition \ref{def-bcov-torsion}.

\begin{thm}
\label{thm-bcov-adiabatic}
As $\e\rightarrow 0$,
\begin{align}
\label{eq-thm-bcov-adiabatic}
\begin{split}
& \tau_\mathrm{BCOV}(X,\omega_\e) - \frac{1}{12} \chi(Z) \Big( m \chi(Y) + (c_1c_{m-1})(Y) \Big) \log \e  \\
& \hspace{35mm} \rightarrow \chi(Z) \tau_\mathrm{BCOV}(Y,\omega_Y) + \chi(Y) \tau_\mathrm{BCOV}(Z,\omega_Z) \;.
\end{split}
\end{align}
\end{thm}
\begin{proof}
The proof consists of several steps.

Recall that $\eta(\cdot)$ was constructed in \eqref{eq-def-eta} and $\lambda_\mathrm{dR}(\cdot)$ was constructed in \eqref{eq-def-lambda}.

\noindent\textbf{Step 1.}
We construct two canonical sections of
\begin{equation}
\lambda_\mathrm{dR}(X) \otimes \Big(\lambda_\mathrm{dR}(Y)\Big)^{-\chi(Z)} \otimes \Big(\eta(Y)\Big)^{-n\chi(Z)} \;.
\end{equation}

For $p=0,\cdots,m+n$ and $s=0,\cdots,p-1$,
set
\begin{align}
\label{eq1a0-pf-thm-bcov-adiabatic}
\begin{split}
I^p_s = \Big\{ u \in \Lambda^p(T^*X) \; & : \; u(v_1,\cdots,v_p) = 0 \\
& \hspace{5mm} \text{for any } v_1,\cdots,v_{s+1} \in TZ,\; v_{s+2},\cdots,v_p \in TX \Big\} \;.
\end{split}
\end{align}
For convenience,
we denote $I^p_p = \Lambda^p(T^*X)$ and $I^p_{-1} = 0$.
We get a filtration
\begin{equation}
\label{eq1a1-pf-thm-bcov-adiabatic}
\Lambda^p(T^*X) = I^p_p \hookleftarrow I^p_{p-1} \hookleftarrow \cdots \hookleftarrow I^p_{-1} = 0 \;.
\end{equation}
For $r=0,\cdots,m$ and $s=0,\cdots,n$,
we denote
\begin{equation}
\label{eq1a2-pf-thm-bcov-adiabatic}
E_{r,s} = \Lambda^s(T^*Z) \otimes \pi^*\Lambda^r(T^*Y) \;.
\end{equation}
We have a short exact sequence of holomorphic vector bundles over $X$,
\begin{equation}
\label{eq1a3-pf-thm-bcov-adiabatic}
0 \rightarrow I^{r+s}_{s-1} \rightarrow I^{r+s}_s \rightarrow E_{r,s} \rightarrow 0 \;.
\end{equation}
Let
\begin{equation}
\label{eq1a-pf-thm-bcov-adiabatic}
\alpha_{r,s} \in
\Big(\det H^\bullet\big(X,I^{r+s}_{s-1}\big)\Big)^{-1}
\otimes \det H^\bullet\big(X,I^{r+s}_s\big)
\otimes \Big(\det H^\bullet\big(X,E_{r,s}\big)\Big)^{-1} \;.
\end{equation}
be the canonical section induced by the long exact sequence induced by \eqref{eq1a3-pf-thm-bcov-adiabatic}.

Let $H^{\bullet,\bullet}(Z)$ be the fiberwise cohomology.
Since $Z \simeq \CP^n$,
we have
\begin{equation}
\label{eq-vanish-HZ}
H^{p,p}(Z) = \C \hspace{2.5mm} \text{for } p = 0,\cdots, n \;,\hspace{4mm}
H^{p,q}(Z) = 0 \hspace{2.5mm} \text{for } p \neq q \;.
\end{equation}
Applying spectral sequence while using \eqref{eq1a2-pf-thm-bcov-adiabatic} and \eqref{eq-vanish-HZ},
we get
\begin{equation}
\label{eq1b1-pf-thm-bcov-adiabatic}
H^q\big(X,E_{r,s}\big) \simeq H^{r,q-s}\big(Y,H^{s,s}(Z)\big) := H^{q-s}\big(Y,\Lambda^r(T^*Y) \otimes H^{s,s}(Z)\big) \;.
\end{equation}
Let
\begin{equation}
\label{eq1b-pf-thm-bcov-adiabatic}
\beta_{r,s} \in \det H^\bullet\big(X,E_{r,s}\big) \otimes
\Big( \det H^{r,\bullet}\big(Y,H^{s,s}(Z)\big) \Big)^{-(-1)^s}
\end{equation}
be the canonical section induced by \eqref{eq1b1-pf-thm-bcov-adiabatic}.

We have a generator of lattice,
\begin{equation}
\label{eq1c1-pf-thm-bcov-adiabatic}
\delta_s \in H^{2s}_\mathrm{Sing}(\CP^n,\Z) \subseteq H^{2s}_\mathrm{Sing}(\CP^n,\R) \subseteq H^{2s}_\mathrm{Sing}(\CP^n,\C) \;.
\end{equation}
We identify $H^{2s}_\mathrm{Sing}(\CP^n,\C)$ with $H^{2s}_\mathrm{dR}(\CP^n) = H^{s,s}(\CP^n)$ (see \eqref{eq-sing-dr}).
We have an isomorphism
\begin{align}
\label{eq1c3-pf-thm-bcov-adiabatic}
\begin{split}
H^{r,\bullet}(Y) & \rightarrow H^{r,\bullet}\big(Y,H^{s,s}(Z)\big) = H^{r,\bullet}(Y) \otimes H^{s,s}(\CP^n)\\
u & \mapsto u \otimes \delta_s \;.
\end{split}
\end{align}
Let
\begin{equation}
\label{eq1c-pf-thm-bcov-adiabatic}
\gamma_{r,s} \in \Big( \det H^{r,\bullet}\big(Y,H^{s,s}(Z)\big) \Big)^{(-1)^s} \otimes
\Big( \det H^{r,\bullet}(Y) \Big)^{-(-1)^s}
\end{equation}
be the canonical section induced by \eqref{eq1c3-pf-thm-bcov-adiabatic}.
By \eqref{eq1a-pf-thm-bcov-adiabatic},
\eqref{eq1b-pf-thm-bcov-adiabatic},
\eqref{eq1c-pf-thm-bcov-adiabatic},
we have
\begin{align}
\label{eq1d-pf-thm-bcov-adiabatic}
\begin{split}
& \alpha_{r,s} \otimes \beta_{r,s} \otimes \gamma_{r,s} \\
& \in  \Big(\det H^\bullet\big(X,I^{r+s}_{s-1}\big)\Big)^{-1}
\otimes \det H^\bullet\big(X,I^{r+s}_s\big)
\otimes \Big( \det H^{r,\bullet}(Y) \Big)^{-(-1)^s} \;.
\end{split}
\end{align}

Recall that $\lambda(\cdot)$ was defined in \eqref{eq-def-lambda}.
By \eqref{eq-def-lambda} and \eqref{eq1a1-pf-thm-bcov-adiabatic},
we have
\begin{align}
\label{eq1s1-pf-thm-bcov-adiabatic}
\begin{split}
\lambda(X) & = \bigotimes_{p=1}^{m+n} \Big(\det H^\bullet\big(X,\Lambda^p(T^*X)\big)\Big)^{(-1)^pp} \\
& = \bigotimes_{r=0}^m \bigotimes_{s=0}^n
\bigg( \Big(\det H^\bullet\big(X,I^{r+s}_{s-1}\big)\Big)^{-1}
\otimes \det H^\bullet\big(X,I^{r+s}_s\big) \bigg)^{(-1)^{r+s}(r+s)} \;.
\end{split}
\end{align}
On the other hand,
by \eqref{eq-def-eta}, \eqref{eq-def-lambda}
and the identities
\begin{equation}
\label{eq-sum-bs}
n+1 = \chi(Z) \;,\hspace{4mm}
\sum_{s=0}^n s = \frac{n(n+1)}{2} = \frac{n}{2}\chi(Z) \;,
\end{equation}
we have
\begin{equation}
\label{eq1s2-pf-thm-bcov-adiabatic}
\bigotimes_{r=0}^m \bigotimes_{s=0}^n \Big(\det H^{r,\bullet}(Y) \Big)^{(-1)^r(r+s)} =
\Big(\lambda(Y)\Big)^{\chi(Z)} \otimes \Big(\eta(Y)\Big)^{n\chi(Z)/2} \;.
\end{equation}
By \eqref{eq1d-pf-thm-bcov-adiabatic},
\eqref{eq1s1-pf-thm-bcov-adiabatic} and
\eqref{eq1s2-pf-thm-bcov-adiabatic},
we have
\begin{equation}
\label{eq1s3-pf-thm-bcov-adiabatic}
\prod_{r=0}^m \prod_{s=0}^n \Big( \alpha_{r,s}\otimes\beta_{r,s}\otimes\gamma_{r,s} \Big)^{(-1)^{r+s}(r+s)}
\in \lambda(X) \otimes \Big(\lambda(Y)\Big)^{-\chi(Z)} \otimes \Big(\eta(Y)\Big)^{-n\chi(Z)/2} \;.
\end{equation}
By \eqref{eq-def-lambda} and \eqref{eq1s3-pf-thm-bcov-adiabatic},
we have
\begin{align}
\label{eq1s4-pf-thm-bcov-adiabatic}
\begin{split}
\prod_{r=0}^m \prod_{s=0}^n \Big( \alpha_{r,s}\otimes\beta_{r,s}\otimes\gamma_{r,s} \Big)^{(-1)^{r+s}(r+s)}
\otimes \overline{ \prod_{r=0}^m \prod_{s=0}^n \Big( \alpha_{r,s}\otimes\beta_{r,s}\otimes\gamma_{r,s} \Big)^{(-1)^{r+s}(r+s)} } & \\
\in \lambda_\mathrm{dR}(X) \otimes \Big(\lambda_\mathrm{dR}(Y)\Big)^{-\chi(Z)} \otimes \Big(\eta(Y)\Big)^{-n\chi(Z)} & \;,
\end{split}
\end{align}
where $\overline{\,\cdot\,}$ is the conjugation.

Let $\sigma_X \in \lambda_\mathrm{dR}(X)$, $\sigma_Y \in \lambda_\mathrm{dR}(Y)$ and $\epsilon_Y \in \eta(Y)$ be as in \eqref{eq-def-epsilon}.
Obviously,
we have
\begin{equation}
\sigma_X \otimes \sigma_Y^{-\chi(Z)} \otimes \epsilon_Y^{-n\chi(Z)}
\in \lambda_\mathrm{dR}(X) \otimes \Big(\lambda_\mathrm{dR}(Y)\Big)^{-\chi(Z)} \otimes \Big(\eta(Y)\Big)^{-n\chi(Z)} \;.
\end{equation}

\noindent\textbf{Step 2.}
We show that
\begin{align}
\label{eq1s-pf-thm-bcov-adiabatic}
\begin{split}
\prod_{r=0}^m \prod_{s=0}^n \Big( \alpha_{r,s}\otimes\beta_{r,s}\otimes\gamma_{r,s} \Big)^{(-1)^{r+s}(r+s)}
\otimes \overline{ \prod_{r=0}^m \prod_{s=0}^n \Big( \alpha_{r,s}\otimes\beta_{r,s}\otimes\gamma_{r,s} \Big)^{(-1)^{r+s}(r+s)} } & \\
= \pm \sigma_X \otimes \sigma_Y^{-\chi(Z)} \otimes \epsilon_Y^{-n\chi(Z)} & \;.
\end{split}
\end{align}

We have canonical identifications
\begin{align}
\label{eq1s5-pf-thm-bcov-adiabatic}
\begin{split}
& H^{2j}_\mathrm{Sing}(\CP^n,\Z) = \Z \hspace{4mm} \text{for } j = 0,\cdots,n \;,\\
& H^k_\mathrm{Sing}(X,\Z) = \bigoplus_{j=0}^n H^{k-2j}_\mathrm{Sing}(Y,\Z) \otimes H^{2j}_\mathrm{Sing}(\CP^n,\Z)
= \bigoplus_{j=0}^n H^{k-2j}_\mathrm{Sing}(Y,\Z) \;,
\end{split}
\end{align}
which induce isomorphisms of Hodge structures.
Complexifying \eqref{eq1s5-pf-thm-bcov-adiabatic} and applying Hodge decomposition,
we get
\begin{align}
\label{eq1s6-pf-thm-bcov-adiabatic}
\begin{split}
& H^{j,j}(\CP^n) = \C \hspace{4mm} \text{for } j = 0,\cdots,n \;,\\
& H^{p,q}(X) = \bigoplus_{j=0}^n H^{p-j,q-j}(Y) \otimes H^{j,j}(\CP^n)
= \bigoplus_{j=0}^n H^{p-j,q-j}(Y) \;.
\end{split}
\end{align}
We will use the identifications in \eqref{eq1s5-pf-thm-bcov-adiabatic} and \eqref{eq1s6-pf-thm-bcov-adiabatic} till the end of Step 2.

Recall that $I^{r+s}_s$ was defined in \eqref{eq1a0-pf-thm-bcov-adiabatic} and $E_{r,s}$ was defined in \eqref{eq1a2-pf-thm-bcov-adiabatic}.
We have
\begin{equation}
\label{eq1s7-pf-thm-bcov-adiabatic}
H^q\big(X,I^{r+s}_s\big) = \bigoplus_{j=0}^s H^{r+s-j,q-j}(Y) \;,\hspace{5mm}
H^q\big(X,E_{r,s}\big) = H^{r,q-s}(Y) \;.
\end{equation}
By \eqref{eq1s7-pf-thm-bcov-adiabatic},
we have
\begin{align}
\label{eq1s8-pf-thm-bcov-adiabatic}
\begin{split}
\Big(\det H^\bullet\big(X,I^{r+s}_{s-1}\big)\Big)^{-1}
\otimes \det H^\bullet\big(X,I^{r+s}_s\big)
\otimes \Big(\det H^\bullet\big(X,E_{r,s}\big)\Big)^{-1} = \C & \;,\\
\alpha_{r,s} = 1 & \;.
\end{split}
\end{align}
Similar argument shows that
\begin{align}
\begin{split}
\det H^\bullet\big(X,E_{r,s}\big) \otimes \Big( \det H^{r,\bullet}\big(Y,H^{s,s}(Z)\big) \Big)^{-(-1)^s} = \C & \;,\\
\beta_{r,s} = 1 & \;, \\
\Big( \det H^{r,\bullet}\big(Y,H^{s,s}(Z)\big) \Big)^{(-1)^s} \otimes \Big( \det H^{r,\bullet}(Y) \Big)^{-(-1)^s} = \C & \;,\\
\gamma_{r,s} = 1 \;.
\end{split}
\end{align}
Using \eqref{eq-def-lambda}, \eqref{eq-sing-dr} and \eqref{eq1s5-pf-thm-bcov-adiabatic},
we can show that
\begin{align}
\label{eq1s9-pf-thm-bcov-adiabatic}
\begin{split}
\lambda_\mathrm{dR}(X) \otimes \Big(\lambda_\mathrm{dR}(Y)\Big)^{-\chi(Z)} \otimes \Big(\eta(Y)\Big)^{-n\chi(Z)} = \C & \;,\\
\sigma_X \otimes \sigma_Y^{-\chi(Z)} \otimes \epsilon_Y^{-n\chi(Z)} = \pm 1 & \;.
\end{split}
\end{align}
From \eqref{eq1s8-pf-thm-bcov-adiabatic}-\eqref{eq1s9-pf-thm-bcov-adiabatic},
we obtain \eqref{eq1s-pf-thm-bcov-adiabatic}.

\noindent\textbf{Step 3.}
We introduce several Quillen metrics.

\begin{itemize}
\item[-] Let $g^{TX}_\e$ be the metric on $TX$ induced by $\omega_\e$.
\item[-] Let $g^{\Lambda^p(T^*X)}_\e$ be the metric on $\Lambda^p(T^*X)$ induced by $g^{TX}_\e$.
\item[-] Let $g^{I^p_s}_\e$ be the metric on $I^p_s$ induced by $g^{\Lambda^p(T^*X)}_\e$ via \eqref{eq1a1-pf-thm-bcov-adiabatic}.
\item[-] Let $g^{TY}$ be the metric on $TY$ induced by $\omega_Y$.
\item[-] Let $g^{\Lambda^r(T^*Y)}$ be the metric on $\Lambda^r(T^*Y)$ induced by $g^{TY}$.
\item[-] Let $g^{TZ}$ be the metric on $TZ$ induced by $\omega_Z = \omega_\e\big|_Z$.
\item[-] Let $g^{\Lambda^s(T^*Z)}$ be the metric on $\Lambda^s(T^*Z)$ induced by $g^{TZ}$.
\item[-] Let $g^{E_{r,s}}$ be the metric on $E_{r,s}$ induced by $g^{\Lambda^r(T^*Y)}$ and $g^{\Lambda^s(T^*Z)}$ via \eqref{eq1a2-pf-thm-bcov-adiabatic}.
\end{itemize}

Let
\begin{equation}
\label{eq3a1-pf-thm-bcov-adiabatic}
\big\lVert\cdot\big\rVert_{\det H^\bullet(X,I^p_s),\e}
\end{equation}
be the Quillen metric on $\det H^\bullet\big(X,I^p_s\big)$
associated with $g^{TX}_\e$ and $g^{I^p_s}$.
Let
\begin{equation}
\label{eq3a2-pf-thm-bcov-adiabatic}
\big\lVert\cdot\big\rVert_{\det H^\bullet(X,E_{r,s}),\e}
\end{equation}
be the Quillen metric on $\det H^\bullet\big(X,E_{r,s}\big)$
associated with $g^{TX}_\e$ and $g^{E_{r,s}}$.
Recall that $\alpha_{r,s}$ was defined by \eqref{eq1a-pf-thm-bcov-adiabatic}.
Let $\big\lVert\alpha_{r,s}\big\rVert_\e$ be the norm of $\alpha_{r,s}$ with respect to the metrics
\eqref{eq3a1-pf-thm-bcov-adiabatic} and \eqref{eq3a2-pf-thm-bcov-adiabatic}.

\begin{itemize}
\item[-] Let $g^{\Omega^{s,s}(Z)}$ be the $L^2$-metric on $\Omega^{s,s}(Z)$ induced by $g^{TZ}$ (see \eqref{eq-def-L2-metric}).
\item[-] Let $g^{H^{s,s}(Z)}$ be the metric on $H^{s,s}(Z)$ induced by $g^{\Omega^{s,s}(Z)}$ via the Hodge theorem.
\end{itemize}

Let
\begin{equation}
\label{eq3b1-pf-thm-bcov-adiabatic}
\big\lVert\cdot\big\rVert_{\det H^{r,\bullet}(Y,H^{s,s}(Z))}
\end{equation}
be the Quillen metric on $\det H^{r,\bullet}\big(Y,H^{s,s}(Z)\big) = \det H^\bullet\big(Y,\Lambda^r(T^*Y)\otimes H^{s,s}(Z)\big)$
associated with $g^{TY}$ and $g^{\Lambda^r(T^*Y)} \otimes g^{H^{s,s}(Z)}$.
Recall that $\beta_{r,s}$ was defined by \eqref{eq1b-pf-thm-bcov-adiabatic}.
Let $\big\lVert\beta_{r,s}\big\rVert_\e$ be the norm of $\beta_{r,s}$ with respect to the metrics
\eqref{eq3a2-pf-thm-bcov-adiabatic} and \eqref{eq3b1-pf-thm-bcov-adiabatic}.
Let
\begin{equation}
\label{eq3c1-pf-thm-bcov-adiabatic}
\big\lVert\cdot\big\rVert_{\det H^{r,\bullet}(Y)}
\end{equation}
be the Quillen metric on $\det H^{r,\bullet}(Y) = \det H^\bullet\big(Y,\Lambda^r(T^*Y)\big)$
associated with $g^{TY}$ and $g^{\Lambda^r(T^*Y)}$.
Recall that $\gamma_{r,s}$ was defined by \eqref{eq1c-pf-thm-bcov-adiabatic}.
Let $\big\lVert\gamma_{r,s}\big\rVert$ be the norm of $\gamma_{r,s}$ with respect to the metrics
\eqref{eq3b1-pf-thm-bcov-adiabatic} and \eqref{eq3c1-pf-thm-bcov-adiabatic}.

By \eqref{eq-def-lambda} and \eqref{eq1a1-pf-thm-bcov-adiabatic},
we have
\begin{equation}
\sigma_X \in \lambda_\mathrm{dR}(X)
= \bigotimes_{p=1}^{m+n} \Big(\det H^\bullet\big(X,I^p_p\big)\Big)^{(-1)^pp} \otimes \overline{\bigotimes_{p=1}^{m+n} \Big(\det H^\bullet\big(X,I^p_p\big)\Big)^{(-1)^pp}} \;.
\end{equation}
Let $\big\lVert\sigma_X\big\rVert_\e$ be the norm of $\sigma_X$ with respect to the metrics \eqref{eq3a1-pf-thm-bcov-adiabatic} with $s=p$.
By \eqref{eq-def-eta} and \eqref{eq-def-lambda},
we have
\begin{align}
\begin{split}
& \epsilon_Y \in \eta(Y) = \bigotimes_{r=0}^m \Big(\det H^{r,\bullet}(Y)\Big)^{(-1)^r} \;,\\
& \sigma_Y \in \lambda_\mathrm{dR}(Y)
= \bigotimes_{r=1}^m \Big(\det H^{r,\bullet}(Y)\Big)^{(-1)^rr} \otimes \overline{\bigotimes_{r=1}^m \Big(\det H^{r,\bullet}(Y)\Big)^{(-1)^rr}} \;.
\end{split}
\end{align}
Let $\big\lVert\epsilon_Y\big\rVert$ be the norm of $\epsilon_Y$ with respect to the metrics \eqref{eq3c1-pf-thm-bcov-adiabatic}.
Let $\big\lVert\sigma_Y\big\rVert$ be the norm of $\sigma_Y$ with respect to the metrics \eqref{eq3c1-pf-thm-bcov-adiabatic}.
By \eqref{eq1s-pf-thm-bcov-adiabatic},
we have
\begin{align}
\label{eq3s1-pf-thm-bcov-adiabatic}
\begin{split}
& \sum_{r=0}^m \sum_{s=0}^n (-1)^{r+s}(r+s)
\Big( \log\big\lVert\alpha_{r,s}\big\rVert^2_\e +
\log\big\lVert\beta_{r,s}\big\rVert^2_\e +
\log\big\lVert\gamma_{r,s}\big\rVert^2 \Big) \\
& = \log\big\lVert\sigma_X\big\rVert_\e
- \chi(Z) \log\big\lVert\sigma_Y\big\rVert
- n \chi(Z) \log\big\lVert\epsilon_Y\big\rVert \;.
\end{split}
\end{align}
On the other hand,
by Definition \ref{def-tau-top} and Proposition \ref{prop-tau-top},
we have
\begin{equation}
\label{eq3s2-pf-thm-bcov-adiabatic}
\log\big\lVert\epsilon_Y\big\rVert = 0 \;.
\end{equation}
By Definition  \ref{def-bcov-torsion}, \eqref{eq3s1-pf-thm-bcov-adiabatic} and \eqref{eq3s2-pf-thm-bcov-adiabatic},
we have
\begin{align}
\label{eq3s-pf-thm-bcov-adiabatic}
\begin{split}
\tau_\mathrm{BCOV}(X,\omega_\e)
& = \chi(Z) \tau_\mathrm{BCOV}(Y,\omega_Y) \\
+ & \sum_{r=0}^m \sum_{s=0}^n (-1)^{r+s}(r+s)
\Big( \log\big\lVert\alpha_{r,s}\big\rVert^2_\e +
\log\big\lVert\beta_{r,s}\big\rVert^2_\e +
\log\big\lVert\gamma_{r,s}\big\rVert^2 \Big) \;.
\end{split}
\end{align}

\noindent\textbf{Step 4.}
We estimate $\log\big\lVert\alpha_{r,s}\big\rVert^2_\e$.

Recall that $I^{r+s}_s$ was defined in \eqref{eq1a0-pf-thm-bcov-adiabatic},
$E_{r,s}$ was defined in \eqref{eq1a2-pf-thm-bcov-adiabatic},
$g^{I^{r+s}_s}_\e$ and $g^{E_{r,s}}$ were defined at the beginning of Step 3.
Let $g^{E_{r,s}}_\e$ be quotient metric on $E_{r,s}$ induced by $g^{I^{r+s}_s}_\e$ via the surjection $I^{r+s}_s \rightarrow E_{r,s}$ in \eqref{eq1a3-pf-thm-bcov-adiabatic}.
Note that $g^{I^{r+s}_s}_\e$ is induced by $\omega_\e$.
By \eqref{eq-omega-e},
as $\e \rightarrow 0$,
\begin{equation}
\label{eqa0-pf-thm-bcov-adiabatic}
\e^{-r} g^{E_{r,s}}_\e \rightarrow  g^{E_{r,s}} \;.
\end{equation}

We will use the notations in \eqref{eq-def-QQ}.
Let
\begin{equation}
\widetilde{T}_{r,s,\e} = \widetilde{\mathrm{ch}}\big( g^{I^{r+s}_{s-1}}_\e, g^{I^{r+s}_s}_\e, g^{E_{r,s}}_\e \big) \in Q^X/Q^{X,0}
\end{equation}
be the Bott-Chern form \eqref{eq-def-BCch}
with $0\rightarrow E' \rightarrow E \rightarrow E'' \rightarrow 0$ replaced by \eqref{eq1a3-pf-thm-bcov-adiabatic}
and $\big(g^{E'},g^E,g^{E''}\big)$ replaced by $\big( g^{I^{r+s}_{s-1}}_\e, g^{I^{r+s}_s}_\e, g^{E_{r,s}}_\e \big)$.
Let
\begin{equation}
T_{r,s,\e} = \widetilde{\mathrm{ch}}\big( g^{I^{r+s}_{s-1}}_\e, g^{I^{r+s}_s}_\e, g^{E_{r,s}} \big) \in Q^X/Q^{X,0}
\end{equation}
be the Bott-Chern form \eqref{eq-def-BCch}
with $0\rightarrow E' \rightarrow E \rightarrow E'' \rightarrow 0$ replaced by \eqref{eq1a3-pf-thm-bcov-adiabatic}
and $\big(g^{E'},g^E,g^{E''}\big)$ replaced by $\big( g^{I^{r+s}_{s-1}}_\e, g^{I^{r+s}_s}_\e, g^{E_{r,s}} \big)$.
By Proposition \ref{prop-mult-bc} and \eqref{eqa0-pf-thm-bcov-adiabatic},
as $\e \rightarrow 0$,
\begin{align}
\label{eqa3-pf-thm-bcov-adiabatic}
\begin{split}
& T_{r,s,\e} - \widetilde{T}_{r,s,\e} - \mathrm{ch}\big(E_{r,s},g^{E_{r,s}}\big) r \log \e \\
& = \widetilde{\mathrm{ch}}\big(g^{E_{r,s}},g^{E_{r,s}}_\e\big) - \mathrm{ch}\big(E_{r,s},g^{E_{r,s}}\big) r \log \e \rightarrow 0 \;.
\end{split}
\end{align}
On the other hand,
by Proposition \ref{prop4-adiabatic-bc},
as $\e \rightarrow 0$,
\begin{equation}
\label{eqa4-pf-thm-bcov-adiabatic}
\widetilde{T}_{r,s,\e} \rightarrow 0 \;.
\end{equation}
By \eqref{eqa3-pf-thm-bcov-adiabatic} and \eqref{eqa4-pf-thm-bcov-adiabatic},
as $\e \rightarrow 0$,
\begin{equation}
\label{eqa5-pf-thm-bcov-adiabatic}
T_{r,s,\e} - \mathrm{ch}\big(E_{r,s},g^{E_{r,s}}\big) r \log \e \rightarrow 0 \;.
\end{equation}

Applying Theorem \ref{thm-resol} to the short exact sequence \eqref{eq1a3-pf-thm-bcov-adiabatic},
we get
\begin{equation}
\label{eqa6-pf-thm-bcov-adiabatic}
\log \big\lVert\alpha_{r,s}\big\rVert_\e^2
= \int_X \mathrm{Td}\big(TX,g^{TX}_\e\big) T_{r,s,\e} \;.
\end{equation}
By Proposition \ref{prop3-adiabatic-bc},
as $\e \rightarrow 0$,
\begin{equation}
\label{eqa7-pf-thm-bcov-adiabatic}
\mathrm{Td}\big(TX,g^{TX}_\e\big) \rightarrow
\pi^*\mathrm{Td}\big(TY,g^{TY}\big) \mathrm{Td}\big(TZ,g^{TZ}\big) \;.
\end{equation}
On the other hand,
by the Grothendieck-Riemann-Roch formula \eqref{eq-grr},
\eqref{eq1a2-pf-thm-bcov-adiabatic} and \eqref{eq-vanish-HZ},
we have
\begin{align}
\label{eqa8-pf-thm-bcov-adiabatic}
\begin{split}
& \int_X \pi^*\mathrm{Td}\big(TY,g^{TY}\big) \mathrm{Td}\big(TZ,g^{TZ}\big) \mathrm{ch}\big(E_{r,s},g^{E_{r,s}}\big) \\
& = \int_Y \mathrm{Td}(TY) \mathrm{ch}\big(H^\bullet(Z,E_{r,s})\big) \\
& = \int_Y \mathrm{Td}(TY) \mathrm{ch}\big(\Lambda^r(T^*Y)\big) \mathrm{ch}\big(H^{s,\bullet}(Z)\big)
= (-1)^s \int_Y \mathrm{Td}(TY) \mathrm{ch}\big(\Lambda^r(T^*Y)\big) \;.
\end{split}
\end{align}
By \eqref{eqa5-pf-thm-bcov-adiabatic}-\eqref{eqa8-pf-thm-bcov-adiabatic},
as $\e \rightarrow 0$,
\begin{equation}
\label{eqa9-pf-thm-bcov-adiabatic}
\log \big\lVert\alpha_{r,s}\big\rVert_\e^2
- (-1)^s r \int_Y \mathrm{Td}(TY) \mathrm{ch}\big(\Lambda^r(T^*Y)\big) \log \e \rightarrow 0 \;.
\end{equation}

By Proposition \ref{prop-total-class},
\eqref{eq-sum-bs} and \eqref{eqa9-pf-thm-bcov-adiabatic},
as $\e \rightarrow 0$,
\begin{align}
\label{eqa-pf-thm-bcov-adiabatic}
\begin{split}
& \sum_{r=0}^m \sum_{s=0}^n (-1)^{r+s}(r+s) \log\big\lVert\alpha_{r,s}\big\rVert^2_\e \\
& - \Big( \frac{m(3m+3n+1)}{12}\chi(Y) + \frac{1}{6}(c_1c_{m-1})(Y) \Big) \chi(Z) \log \e \rightarrow 0 \;.
\end{split}
\end{align}

\noindent\textbf{Step 5.}
We estimate $\log\big\lVert\beta_{r,s}\big\rVert^2_\e$.

Let
\begin{equation}
T_{r,s}\in Q^Y
\end{equation}
be the Bismut-K{\"o}hler analytic torsion form (see \textsection \ref{subsect-tf}) associated with
$\big( \pi: X \rightarrow Y, \omega_X, E_{r,s}, g^{E_{r,s}} \big)$.
Applying Theorem \ref{thm-sub} with $E=E_{r,s}$,
as $\e \rightarrow 0$,
\begin{equation}
\label{eqb2-pf-thm-bcov-adiabatic}
\log \big\lVert\beta_{r,s}\big\rVert_\e^2
+ \int_Y \mathrm{Td}'(TY) \int_Z \mathrm{Td}(TZ) \mathrm{ch}(E_{r,s}) \log\e
\rightarrow \int_Y \mathrm{Td}\big(TY,g^{TY}\big) T_{r,s}  \;.
\end{equation}
Similarly to \eqref{eqa8-pf-thm-bcov-adiabatic},
we have
\begin{equation}
\label{eqb3-pf-thm-bcov-adiabatic}
\int_Y \mathrm{Td}'(TY) \int_Z \mathrm{Td}(TZ) \mathrm{ch}(E_{r,s})
= (-1)^s \int_Y \mathrm{Td}'(TY)\mathrm{ch}\big(\Lambda^r(T^*Y)\big) \;.
\end{equation}
Applying Proposition \ref{prop-tf-mult} with $E = E_{0,s}$ and $F = \Lambda^r(T^*Y)$,
we get
\begin{equation}
\label{eqb4-pf-thm-bcov-adiabatic}
T_{r,s} = \mathrm{ch}\big(\Lambda^r(T^*Y),g^{\Lambda^r(T^*Y)}\big)T_{0,s} \hspace{2.5mm} \text{modulo } Q^{Y,0} \;.
\end{equation}
By \eqref{eqb2-pf-thm-bcov-adiabatic}-\eqref{eqb4-pf-thm-bcov-adiabatic},
as $\e \rightarrow 0$,
\begin{align}
\label{eqb5-pf-thm-bcov-adiabatic}
\begin{split}
& \log \big\lVert\beta_{r,s}\big\rVert_\e^2
+ (-1)^s \int_Y \mathrm{Td}'(TY)\mathrm{ch}\big(\Lambda^r(T^*Y)\big) \log\e \\
& \hspace{30mm} \rightarrow \int_Y \mathrm{Td}\big(TY,g^{TY}\big) \mathrm{ch}\big(\Lambda^r(T^*Y),g^{\Lambda^r(T^*Y)}\big)T_{0,s}  \;.
\end{split}
\end{align}
On the other hand,
by Theorem \ref{thm-tf-vanishing},
we have
\begin{equation}
\label{eqb6-pf-thm-bcov-adiabatic}
\sum_{s=0}^n (-1)^s T_{0,s} = 0 \hspace{2.5mm} \text{modulo } Q^{Y,0} \;.
\end{equation}
By Proposition \ref{prop-total-class}, \ref{prop2-total-class},
\eqref{eq-sum-bs}, \eqref{eqb5-pf-thm-bcov-adiabatic} and \eqref{eqb6-pf-thm-bcov-adiabatic},
as $\e\rightarrow 0$,
\begin{align}
\label{eqb-pf-thm-bcov-adiabatic}
\begin{split}
& \sum_{r=0}^m \sum_{s=0}^n (-1)^{r+s}(r+s) \log\big\lVert\beta_{r,s}\big\rVert^2_\e \\
& + \Big( \frac{m(m+n)}{4}\chi(Y) + \frac{1}{12}(c_1c_{m-1})(Y) \Big) \chi(Z) \log \e \\
& \rightarrow \int_Y c_m\big(TY,g^{TY}\big) \sum_{s=0}^n (-1)^s s T_{0,s}
= \int_Y c_m\big(TY,g^{TY}\big) \sum_{s=0}^n (-1)^s s \big\{ T_{0,s} \big\}^{(0,0)} \;,
\end{split}
\end{align}
where $\big\{\cdot\big\}^{(0,0)}$ means the component of degree $(0,0)$.

\noindent\textbf{Step 6.}
We calculate $\log\big\lVert\gamma_{r,s}\big\rVert^2$.

Recall that $H^{s,s}(Z)$ is a trivial line bundle over $Y$.
Recall that $g^{H^{s,s}(Z)}$ was constructed in the paragraph above \eqref{eq3b1-pf-thm-bcov-adiabatic}.
By our assumption \eqref{eq-hyp-isometry},
$g^{H^{s,s}(Z)}$ is a constant metric.
Recall that $\delta_s \in H^{s,s}(Z)$ was constructed in \eqref{eq1c1-pf-thm-bcov-adiabatic}.
Let $\big|\delta_s\big|$ be the norm of $\delta_s$ with respect to $g^{H^{s,s}(Z)}$,
which is a constant function on $Y$.
In the sequel,
we will not distinguish between a constant function and its value.
We denote $\chi_r(Y) = \sum_{q=0}^m (-1)^q \dim H^{r,q}(Y)$.
By Remark \ref{rem-q},
we have
\begin{equation}
\label{eqc2-pf-thm-bcov-adiabatic}
\log\big\lVert\gamma_{r,s}\big\rVert^2 = (-1)^s \chi_r(Y) \log \big|\delta_s\big|^2 \;,
\end{equation}

Let $\epsilon_Z \in \eta(Z)$ be as in \eqref{eq-def-epsilon}.
We have
\begin{equation}
\label{eqc4-pf-thm-bcov-adiabatic}
\epsilon_Z = \pm \bigotimes_{s=0}^n \delta_s  \;.
\end{equation}
Let $\big|\epsilon_Z\big|$ be the norm of $\epsilon_Z$ with respect to the metrics $g^{H^{s,s}(Z)}$.
By Proposition \ref{prop-tau-top} and \eqref{eqc4-pf-thm-bcov-adiabatic},
we have
\begin{equation}
\label{eqc5-pf-thm-bcov-adiabatic}
\sum_{s=0}^n \log \big|\delta_s\big|^2 = \log \big|\epsilon_Z\big|^2 = 0 \;.
\end{equation}

Let $\sigma_Z \in \lambda_\mathrm{dR}(Z)$ be as in \eqref{eq-def-epsilon}.
We have
\begin{equation}
\label{eqc6-pf-thm-bcov-adiabatic}
\sigma_Z = \pm \bigotimes_{s=1}^n \delta_s^{2s}  \;.
\end{equation}
Let $\big|\sigma_Z\big|$ be the norm of $\sigma_Z$ with respect to the metrics $g^{H^{s,s}(Z)}$.
By \eqref{eqc6-pf-thm-bcov-adiabatic},
we have
\begin{equation}
\label{eqc7-pf-thm-bcov-adiabatic}
\sum_{s=0}^n s \log \big|\delta_s\big|^2 = \log \big| \sigma_Z \big| \;.
\end{equation}

By \eqref{eqc2-pf-thm-bcov-adiabatic}, \eqref{eqc5-pf-thm-bcov-adiabatic}, \eqref{eqc7-pf-thm-bcov-adiabatic}
and the identity $\sum_{r=0}^m (-1)^r \chi_r(Y) = \chi(Y)$,
we have
\begin{equation}
\label{eqc-pf-thm-bcov-adiabatic}
\sum_{r=0}^m \sum_{s=0}^n (-1)^{r+s}(r+s) \log\big\lVert\gamma_{r,s}\big\rVert^2
= \chi(Y) \log \big| \sigma_Z \big| \;.
\end{equation}

\noindent\textbf{Step 7.}
We conclude.

By \eqref{eq3s-pf-thm-bcov-adiabatic},
\eqref{eqa-pf-thm-bcov-adiabatic}
\eqref{eqb-pf-thm-bcov-adiabatic} and
\eqref{eqc-pf-thm-bcov-adiabatic},
as $\e \rightarrow 0$,
\begin{align}
\label{eq70-pf-thm-bcov-adiabatic}
\begin{split}
& \tau_\mathrm{BCOV}(X,\omega_\e) - \frac{1}{12} \chi(Z) \Big( m \chi(Y) + (c_1c_{m-1})(Y) \Big) \log \e \\
& \rightarrow \chi(Z) \tau_\mathrm{BCOV}(Y,\omega_Y) + \chi(Y) \log \big| \sigma_Z \big| + \int_Y c_m\big(TY,g^{TY}\big) \sum_{s=0}^n (-1)^s s \big\{T_{0,s}\big\}^{(0,0)} \;.
\end{split}
\end{align}

Let $\theta_s(z)$ be as in \eqref{eq-def-theta}
with $(X,\omega)$ replaced by $(Z,\omega_Z)$
and $(E,g^E)$ replaced by $(\Lambda^s(T^*Z),g^{\Lambda^s(T^*Z)})$.
By Definition \ref{def-q}, \ref{def-bcov-torsion},
we have
\begin{equation}
\label{eq71-pf-thm-bcov-adiabatic}
\tau_\mathrm{BCOV}(Z,\omega_Z) =
\log \big| \sigma_Z \big| + \sum_{s=0}^n (-1)^s s \theta_s'(0) \;.
\end{equation}
By \eqref{eq-hyp-isometry},
all the terms in \eqref{eq71-pf-thm-bcov-adiabatic} are constant functions on $Y$.
By \eqref{eq-tf-deg0},
we have
\begin{equation}
\label{eq72-pf-thm-bcov-adiabatic}
\big\{ T_{0,s} \big\}^{(0,0)} = \theta_s'(0) \;.
\end{equation}
From \eqref{eq70-pf-thm-bcov-adiabatic}-\eqref{eq72-pf-thm-bcov-adiabatic},
we obtain \eqref{eq-thm-bcov-adiabatic}.
This completes the proof.
\end{proof}

\subsection{Behavior under blow-ups}
\label{subsect-bl}

The following lemma is direct consequence of Bott formula \cite{bott} (see also \cite[page 5]{oss}).

\begin{lemme}
\label{lem-vanishing}
Let $L$ be the holomorphic line bundle of degree $1$ over $\CP^n$.
For $k=1,\cdots,n$ and $s =1,\cdots,k$,
we have
\begin{equation}
\label{eq-lem-vanishing}
H^\bullet\big(\CP^n,\Lambda^k(T^*\CP^n) \otimes L^s \big) = 0 \;.
\end{equation}
\end{lemme}

Let $X$ be an $n$-dimensional compact K{\"a}hler manifold.
Let $Y\subseteq X$ be a closed complex submanifold.
Let $f: X' \rightarrow X$ be the blow-up along $Y$.
Let $Y \subseteq U \subseteq X$ be an open neighborhood of $Y$.
Set $U' = f^{-1}(U)$.
Let $\omega$ be a K{\"a}hler form on $X$.
Let $\omega'$ be a K{\"a}hler form on $X'$ such that
\begin{equation}
\label{eq-f-omega}
\omega'\big|_{X'\backslash U'} = f^*\big(\omega\big|_{X\backslash U}\big) \;.
\end{equation}

\begin{thm}
\label{thm-bl-bcovt}
We have
\begin{equation}
\tau_\mathrm{BCOV}(X',\omega') - \tau_\mathrm{BCOV}(X,\omega) =
\alpha\big(U,U',\omega\big|_U,\omega'\big|_{U'}\big) \;,
\end{equation}
where $\alpha\big(U,U',\omega\big|_U,\omega'\big|_{U'}\big)$
is a real number determined by $U$, $U'$, $\omega\big|_U$ and $\omega'\big|_{U'}$.
\end{thm}
\begin{proof}
The proof consists of several steps.

\noindent\textbf{Step 0.}
We introduce several notations.

We denote $D = f^{-1}(Y)$.
Let $i: D \hookrightarrow X'$ be the canonical embedding.

For $p=0,\cdots,n$,
there exist holomorphic vector bundles over $X$ linked by holomorphic maps
\begin{equation}
\label{eq0F0-thm-bl-bcovt}
f^*\Lambda^p(T^*X) = F^p_p \rightarrow F^p_{p-1} \rightarrow \cdots \rightarrow F^p_0 = \Lambda^p(T^*X')
\end{equation}
such that for $s=0,\cdots,p-1$,
the map $F^p_{s+1} \rightarrow F^p_s$ induces the following exact sequence of analytic coherent sheaves on $X'$,
\begin{equation}
\label{eq0F-thm-bl-bcovt}
0 \rightarrow \mathscr{O}_{X'}(F^p_{s+1}) \rightarrow \mathscr{O}_{X'}(F^p_s) \rightarrow i_*\mathscr{O}_D\big(F^p_s\big|_D\big) \;.
\end{equation}

Let $N_Y$ be the normal bundle of $Y\hookrightarrow X$.
Let $\pi: D = \mathbb{P}(N_Y) \rightarrow Y$ be the canonical projection.
Let $TD \rightarrow \pi^* TY$ be the derivative of $\pi$.
Set
\begin{equation}
T^VD = \mathrm{Ker}\big(TD \rightarrow \pi^* TY\big) \subseteq TD \subseteq TX'\big|_D \;.
\end{equation}
Set
\begin{align}
\label{eq0I-thm-bl-bcovt}
\begin{split}
I^p_s = \Big\{\alpha\in\Lambda^p(T^*X')\big|_D \; & : \; \alpha (v_1,\cdots,v_p) = 0 \\
& \text{ for any } v_1,\cdots,v_{s+1}\in T^VD,\; v_{s+2},\cdots,v_p\in TX'\big|_D \Big\} \;.
\end{split}
\end{align}
We get a filtration of holomorphic vector bundles over $D$,
\begin{equation}
\label{eq0Ip-thm-bl-bcovt}
\Lambda^p(T^*X')\big|_D = I^p_p \supseteq I^p_{p-1} \supseteq \cdots \supseteq I^p_0 \;.
\end{equation}
Let $N_D$ be the normal line bundle of $D\hookrightarrow X'$.
For $s=0,\cdots,p-1$,
we can show that the image of the map $\mathscr{O}_{X'}(F^p_s) \rightarrow i_*\mathscr{O}_D\big(F^p_s\big|_D\big)$ in \eqref{eq0F-thm-bl-bcovt} 
is given by $i_*\mathscr{O}_D\big(N_D^{-s} \otimes \big(I^p_p/I^{p}_s\big)\big)$.
For convenience,
we denote
\begin{equation}
\label{eq0G-thm-bl-bcovt}
G^p_s = N_D^{-s} \otimes \big(I^p_p/I^{p}_s\big) \;.
\end{equation}
Then we get a short exact sequence
\begin{equation}
\label{eq0Ffilt-thm-bl-bcovt}
0 \rightarrow \mathscr{O}_{X'}(F^p_{s+1}) \rightarrow \mathscr{O}_{X'}(F^p_s) \rightarrow i_*\mathscr{O}_D\big(G^p_s\big) \rightarrow 0 \;.
\end{equation}

Let $r$ be the codimension of $Y\hookrightarrow X$.

\noindent\textbf{Step 1.}
We show that
\begin{equation}
\label{eq1-thm-bl-bcovt}
H^q(D,G^p_0) = \bigoplus_{k=1}^{r-1} H^{k,k}(\CP^{r-1}) \otimes H^{p-k,q-k}(Y) \;,\hspace{5mm}
H^q(D,G^p_s) = 0 \hspace{4mm}\text{for } s \neq 0 \;,
\end{equation}

Set
\begin{align}
\label{eq0J-thm-bl-bcovt}
\begin{split}
J^p_s = \Big\{\alpha\in\Lambda^p(T^*D) \; & : \; \alpha (v_1,\cdots,v_p) = 0 \\
& \text{ for any } v_1,\cdots,v_{s+1}\in T^VD,\; v_{s+2},\cdots,v_p\in TD \Big\} \;.
\end{split}
\end{align}
Let $\phi: \Lambda^p(T^*X')\big|_D \rightarrow \Lambda^p(T^*D)$ be the canonical projection.
By \eqref{eq0I-thm-bl-bcovt} and \eqref{eq0J-thm-bl-bcovt}, 
we have
\begin{equation}
\label{eq1J-thm-bl-bcovt}
J^p_s = \phi(I^p_s) \subseteq \Lambda^p(T^*D) \;.
\end{equation}
By \eqref{eq0Ip-thm-bl-bcovt} and \eqref{eq1J-thm-bl-bcovt},
we have a filtration of holomorphic vector bundles over $D$,
\begin{equation}
\label{eq1Jp-thm-bl-bcovt}
\Lambda^p(T^*D) = J^p_p \supseteq J^p_{p-1} \supseteq \cdots \supseteq J^p_0 \;.
\end{equation}
We also have
\begin{equation}
\label{eq1Jfilt-thm-bl-bcovt}
J^p_k/J^p_{k-1} = \pi^*\big(\Lambda^{p-k}(T^*Y)\big) \otimes \Lambda^k(T^{V,*}D) \;,
\end{equation}
and a short exact sequence of holomorphic vector bundles over $D$,
\begin{equation}
\label{eq1IJ-thm-bl-bcovt}
0 \rightarrow N_D^{-1} \otimes J^{p-1}_k \rightarrow I^p_k \rightarrow J^p_k \rightarrow 0 \;.
\end{equation}
Combining \eqref{eq1Jfilt-thm-bl-bcovt} and \eqref{eq1IJ-thm-bl-bcovt},
we get a short exact sequence,
\begin{align}
\label{eqIfilt-thm-bl-bcovt}
\begin{split}
0 & \rightarrow N_D^{-1}\otimes\pi^*\big(\Lambda^{p-k-1}(T^*Y)\big) \otimes \Lambda^k(T^{V,*}D) \rightarrow I^p_k/I^p_{k-1} \\
& \hspace{45mm} \rightarrow \pi^*\big(\Lambda^{p-k}(T^*Y)\big) \otimes \Lambda^k(T^{V,*}D) \rightarrow 0 \;.
\end{split}
\end{align}
By \eqref{eq0G-thm-bl-bcovt} and \eqref{eqIfilt-thm-bl-bcovt},
$G^p_s$ admits a filtration with factors
\begin{equation}
\label{eqIfactor-thm-bl-bcovt}
\Big( N_D^{-s-\epsilon}\otimes\pi^*\big(\Lambda^{p-k-\epsilon}(T^*Y)\big) \otimes \Lambda^k(T^{V,*}D)
\Big)_{\epsilon = 0,1,\; k = s+1,\cdots,p} \;.
\end{equation}

We remark that $\pi: D \rightarrow Y$ is a $\CP^{r-1}$-bundle
and the restriction of $N_D^{-1}$ to the fiber of $\pi: D \rightarrow Y$ is a holomorphic line bundle of degree $1$.
Applying spectral sequence while using Lemma \ref{lem-vanishing},
we see that
the cohomology of the holomorphic vector bundles in \eqref{eqIfactor-thm-bl-bcovt} vanishes
unless $\epsilon = s = 0$.
Hence we obtain the second identity in \eqref{eq1-thm-bl-bcovt}.
The argument above also shows that
\begin{equation}
\label{eq11-thm-bl-bcovt}
H^q(D,G^p_0) = H^q(D,I^p_p/I^p_0) = H^q(D,J^p_p/J^p_0) \;.
\end{equation}
Using spectral sequence and \eqref{eq1Jfilt-thm-bl-bcovt},
we get
\begin{equation}
\label{eq12-thm-bl-bcovt}
H^q\big(D,J^p_k/J^p_{k-1}\big) = H^{k,k}(\CP^{r-1}) \otimes H^{p-k,q-k}(Y) \;.
\end{equation}
On the other hand,
it is classical that
\begin{equation}
\label{eq13-thm-bl-bcovt}
H^q(D,J^p_p) = H^q\big(D,\Lambda^p(T^*D)\big)
= \bigoplus_{k=0}^{r-1} H^{k,k}(\CP^{r-1}) \otimes H^{p-k,q-k}(Y) \;.
\end{equation}
From \eqref{eq11-thm-bl-bcovt}-\eqref{eq13-thm-bl-bcovt},
we obtain the first identity in \eqref{eq1-thm-bl-bcovt}.

Set
\begin{equation}
\label{eq2lambdaG-thm-bl-bcovt}
\lambda(G^\bullet_0) = \bigotimes_{p=1}^n\Big(\det H^\bullet\big(D,G^p_0\big)\Big)^{(-1)^pp} \;,\hspace{4mm}
\lambda_\mathrm{dR}(G^\bullet_0) = \lambda(G^\bullet_0) \otimes \overline{\lambda(G^\bullet_0)} \;.
\end{equation}

Recall that $\lambda_\mathrm{dR}(X)$ was defined in \eqref{eq-def-lambda}.

\noindent\textbf{Step 2.}
We construct two canonical sections of
\begin{equation}
\big(\lambda_\mathrm{dR}(X)\big)^{-1} \otimes \lambda_\mathrm{dR}(X') \otimes \big(\lambda_\mathrm{dR}(G^\bullet_0)\big)^{-1}
\end{equation}
and show that they coincide up to $\pm 1$.

Let
\begin{equation}
\label{eqmus-thm-bl-bcovt}
\mu_{p,s} \in \Big( \det H^\bullet\big(X',F^p_{s+1}\big) \Big)^{-1}
\otimes \det H^\bullet\big(X',F^p_s\big) \otimes \Big(\det H^\bullet\big(D,G^p_s\big)\Big)^{-1}
\end{equation}
be the canonical section induced by the long exact sequence induced by \eqref{eq0Ffilt-thm-bl-bcovt}.
Indeed,
by \eqref{eq1-thm-bl-bcovt},
we have
\begin{equation}
\mu_{p,s} \in \Big( \det H^\bullet\big(X',F^p_{s+1}\big) \Big)^{-1} \otimes \det H^\bullet\big(X',F^p_s\big)
\hspace{4mm} \text{for } s \neq 0 \;.
\end{equation}
Set
\begin{align}
\label{eqmuprod-thm-bl-bcovt}
\begin{split}
\mu_p & = \bigotimes_{s=0}^{p-1} \mu_{p,s} \\
& \in \Big( \det H^\bullet\big(X',F^p_p\big) \Big)^{-1}
\otimes \det H^\bullet\big(X',F^p_0\big) \otimes \Big(\det H^\bullet\big(D,G^p_0\big)\Big)^{-1} \\
& \hspace{5mm} = \Big( \det H^\bullet\big(X',f^*\Lambda^p(T^*X)\big) \Big)^{-1}
\otimes \det H^{p,\bullet}(X') \otimes \Big(\det H^\bullet\big(D,G^p_0\big)\Big)^{-1} \;.
\end{split}
\end{align}

Using spectral sequence,
we get a canonical identification
\begin{equation}
\label{eqf-thm-bl-bcovt}
H^{p,\bullet}(X) = H^\bullet\big(X',f^*\Lambda^p(T^*X)\big) \;.
\end{equation}
Let
\begin{equation}
\label{eqnu-thm-bl-bcovt}
\nu_p \in \Big( \det H^{p,\bullet}(X) \Big)^{-1} \otimes \det H^\bullet\big(X',f^*\Lambda^p(T^*X)\big)
\end{equation}
be the canonical section induced by \eqref{eqf-thm-bl-bcovt}.

By \eqref{eqmuprod-thm-bl-bcovt} and \eqref{eqnu-thm-bl-bcovt},
we have
\begin{equation}
\label{eqeta-thm-bl-bcovt}
\mu_p \otimes \nu_p \in \Big( \det H^{p,\bullet}(X) \Big)^{-1}
\otimes \det H^{p,\bullet}(X') \otimes \Big(\det H^\bullet\big(D,G^p_0\big)\Big)^{-1} \;.
\end{equation}
By \eqref{eq-def-lambda}, \eqref{eq2lambdaG-thm-bl-bcovt} and \eqref{eqeta-thm-bl-bcovt},
we have
\begin{equation}
\bigotimes_{p=1}^n \big(\mu_p \otimes \nu_p\big)^{(-1)^pp} \in \big(\lambda(X)\big)^{-1} \otimes \lambda(X') \otimes \big(\lambda(G^\bullet_0)\big)^{-1} \;,
\end{equation}
and
\begin{align}
\begin{split}
& \bigotimes_{p=1}^n \big(\mu_p \otimes \nu_p\big)^{(-1)^pp} \otimes \overline{\bigotimes_{p=1}^n \big(\mu_p \otimes \nu_p\big)^{(-1)^pp}} \\
& \hspace{40mm} \in \big(\lambda_\mathrm{dR}(X)\big)^{-1} \otimes \lambda_\mathrm{dR}(X') \otimes \big(\lambda_\mathrm{dR}(G^\bullet_0)\big)^{-1} \;.
\end{split}
\end{align}

We have the Hodge decomposition
\begin{equation}
\label{eq2HY-thm-bl-bcovt}
H^j_\mathrm{dR}(Y) = \bigoplus_{p+q=j} H^{p,q}(Y) \;.
\end{equation}
Let $b_k$ be the $k$-th Betti number of $Y$.
By \eqref{eq1-thm-bl-bcovt}, \eqref{eq2lambdaG-thm-bl-bcovt} and \eqref{eq2HY-thm-bl-bcovt},
we have
\begin{equation}
\lambda_\mathrm{dR}(G^\bullet_0) = \bigotimes_{k=1}^{r-1} \bigotimes_{j=2k}^{2k+2n-2r}
\bigg( \Big(\det H^{2k}_\mathrm{dR}(\CP^{r-1})\Big)^{b_{j-2k}} \otimes \det H^{j-2k}_\mathrm{dR}(Y) \bigg)^{(-1)^jj} \;.
\end{equation}
Let
\begin{equation}
\delta_j\in H^j_\mathrm{Sing}(\CP^{r-1},\Z) \subseteq
H^j_\mathrm{Sing}(\CP^{r-1},\C) = H^j_\mathrm{dR}(\CP^{r-1})
\end{equation}
be a generator of $H^j_\mathrm{Sing}(\CP^{r-1},\Z)$.
Let
\begin{equation}
\tau_{j,1},\cdots,\tau_{j,b_j}\in \mathrm{Im}\Big( H^j_\mathrm{Sing}(Y,\Z) \rightarrow H^j_\mathrm{Sing}(Y,\R) \Big)
\subseteq H^j_\mathrm{dR}(Y)
\end{equation}
be a basis of the lattice.
We denote $\tau_j = \tau_{j,1}\wedge\cdots\wedge\tau_{j,b_j} \in \det H^j_\mathrm{dR}(Y)$.
Set
\begin{equation}
\label{eq2G-thm-bl-bcovt}
\sigma_{G^\bullet_0} = \bigotimes_{k=1}^{r-1} \bigotimes_{j=2k}^{2k+2n-2r}
\Big( \delta_{2k}^{b_{j-2k}} \otimes \tau_{j-2k} \Big)^{(-1)^jj}
\in \lambda_\mathrm{dR}(G^\bullet_0) \;.
\end{equation}
Let $\sigma_X \in \lambda_\mathrm{dR}(X)$ and $\sigma_{X'} \in \lambda_\mathrm{dR}(X')$ be as in \eqref{eq-def-epsilon}.
Obviously,
we have
\begin{equation}
\sigma_X^{-1} \otimes \sigma_{X'} \otimes \sigma_{G^\bullet_0}^{-1}
\in \big(\lambda_\mathrm{dR}(X)\big)^{-1} \otimes \lambda_\mathrm{dR}(X') \otimes \big(\lambda_\mathrm{dR}(G^\bullet_0)\big)^{-1} \;.
\end{equation}

We have a canonical identification
\begin{equation}
\label{eq2hdg-thm-bl-bcovt}
H^j_\mathrm{Sing}(X',\Z) = H^j_\mathrm{Sing}(X,\Z) \oplus
\bigoplus_{k=1}^{r-1}
H^{2k}_\mathrm{Sing}(\CP^{r-1},\Z) \otimes H^{j-2k}_\mathrm{Sing}(Y,\Z) \;,
\end{equation}
which induces an isomorphism of Hodge structures.
Similarly to Step 2 in the proof of Theorem \ref{thm-bcov-adiabatic},
using \eqref{eq2hdg-thm-bl-bcovt},
we can show that
\begin{equation}
\label{eq2-thm-bl-bcovt}
\bigotimes_{p=1}^n \big(\mu_p \otimes \nu_p\big)^{(-1)^pp} \otimes \overline{\bigotimes_{p=1}^n \big(\mu_p \otimes \nu_p\big)^{(-1)^pp}}
= \pm \sigma_X^{-1} \otimes \sigma_{X'} \otimes \sigma_{G^\bullet_0}^{-1} \;.
\end{equation}

\noindent\textbf{Step 3.}
We introduce Quillen metrics.

Let $g^{TX}$ be the metric on $TX$ induced by $\omega$ .
Let $g^{\Lambda^p(T^*X)}$ be the metric on $\Lambda^p(T^*X)$ induced by $g^{TX}$.
Let
\begin{equation}
\label{eq3Q1-thm-bl-bcovt}
\big\lVert\cdot\big\rVert_{\det H^{p,\bullet}(X)}
\end{equation}
be the Quillen metric on $\det H^{p,\bullet}(X) = \det H^\bullet\big(X,\Lambda^p(T^*X)\big)$ associated with $g^{TX}$ and $g^{\Lambda^p(T^*X)}$.

Let $g^{TX'}$ be the metric on $TX'$ induced by $\omega'$.
Let $g^{\Lambda^p(T^*X')}$ be the metric on $\Lambda^p(T^*X')$ induced by $g^{TX'}$.
Let
\begin{equation}
\label{eq3Q2-thm-bl-bcovt}
\big\lVert\cdot\big\rVert_{\det H^{p,\bullet}(X')}
\end{equation}
be the Quillen metric on $\det H^{p,\bullet}(X') = \det H^\bullet\big(X',\Lambda^p(T^*X')\big)$ associated with $g^{TX'}$ and $g^{\Lambda^p(T^*X')}$.

Let
\begin{equation}
\label{eq3Q3-thm-bl-bcovt}
\big\lVert\cdot\big\rVert_{\det H^\bullet(X',f^*\Lambda^p(T^*X))}
\end{equation}
be the Quillen metric on $\det H^\bullet\big(X',f^*\Lambda^p(T^*X)\big)$ associated with $g^{TX'}$ and $f^*g^{\Lambda^p(T^*X)}$.

Let $g^{TD}$ and $g^{N_D}$ be the metrics on $TD$ and $N_D$ induced by $g^{TX'}$.
Let $g^{I^p_s}$ be the metric on $I^p_s$ induced by $g^{\Lambda^p(T^*X')}$ via \eqref{eq0Ip-thm-bl-bcovt}.
Let $g^{G^p_s}$ be the metric on $G^p_s$ induced by $g^{N_D}$ and $g^{I^p_s}$ via \eqref{eq0G-thm-bl-bcovt}.
Let
\begin{equation}
\label{eq3Q4-thm-bl-bcovt}
\big\lVert\cdot\big\rVert_{\det H^\bullet(D,G^p_s)}
\end{equation}
be the Quillen metric on $\det H^\bullet(D,G^p_s)$ associated with $g^{TD}$ and $g^{G^p_s}$.
By the second identity in \eqref{eq1-thm-bl-bcovt},
we have a canonical identification $\det H^\bullet(D,G^p_s) = \C$ for $s\neq 0$.
But the metric \eqref{eq3Q4-thm-bl-bcovt} with $s\neq 0$ is not necessarily the standard metric on $\C$.

We remark that
\begin{equation}
\Lambda^p(T^*X')\big|_{X'\backslash U'}
= F^p_s \big|_{X'\backslash U'}
= f^*\Lambda^p(T^*X)\big|_{X'\backslash U'} \hspace{4mm} \text{for } s=0,\cdots,p \;.
\end{equation}
We equip $F^p_s$ with Hermitian metric $g^{F^p_s}$ such that
\begin{align}
\label{eq3Q0-thm-bl-bcovt}
\begin{split}
& g^{F^p_0} = g^{\Lambda^p(T^*X')} \;,\hspace{4mm}
g^{F^p_p} = f^*g^{\Lambda^p(T^*X)} \;,\\
& g^{F^p_{s+1}}\big|_{X'\backslash U'} = g^{F^p_s}\big|_{X'\backslash U'}  \hspace{4mm} \text{for } s = 0,\cdots,p-1 \;.
\end{split}
\end{align}
Our assumption \eqref{eq-f-omega} implies $g^{\Lambda^p(T^*X')}\big|_{X'\backslash U'} = f^*\big(g^{\Lambda^p(T^*X)}\big|_{X\backslash U}\big)$,
which guarantees the existence of $g^{F^p_s}$ satisfying \eqref{eq3Q0-thm-bl-bcovt}.
Let
\begin{equation}
\label{eq3Q6-thm-bl-bcovt}
\big\lVert\cdot\big\rVert_{\det H^\bullet(X',F^p_s)}
\end{equation}
be the Quillen metric on $\det H^\bullet(X',F^p_s)$
associated with $g^{TX'}$ and $g^{F^p_s}$.
We remark that $H^\bullet(X',F^p_0) = H^{p,\bullet}(X')$ and
\begin{equation}
\big\lVert\cdot\big\rVert_{\det H^\bullet(X',F^p_0)} = \big\lVert\cdot\big\rVert_{\det H^{p,\bullet}(X')} \;.
\end{equation}

Recall that $\mu_{p,s}$ was defined in \eqref{eqmus-thm-bl-bcovt}.
Let $\big\lVert\mu_{p,s}\big\rVert$ be the norm of $\mu_{p,s}$
with respect to the metrics \eqref{eq3Q4-thm-bl-bcovt} and \eqref{eq3Q6-thm-bl-bcovt}.

Recall that $\nu_p$ was defined in \eqref{eqnu-thm-bl-bcovt}.
Let $\big\lVert\nu_p\big\rVert$ be the norm of $\nu_p$
with respect to the Quillen metrics \eqref{eq3Q1-thm-bl-bcovt} and \eqref{eq3Q3-thm-bl-bcovt}.

Recall that $\sigma_{G^\bullet_0}$ was defined in \eqref{eq2G-thm-bl-bcovt}.
By \eqref{eq2lambdaG-thm-bl-bcovt} and the second identity in \eqref{eq1-thm-bl-bcovt},
we may and we will view $\sigma_{G^\bullet_0}$ as section of
\begin{equation}
\lambda_\mathrm{dR}(G^\bullet_\bullet) := \bigotimes_{p=1}^n \bigotimes_{s=0}^{p-1} \Big(\det H^\bullet\big(D,G^p_s\big)\Big)^{(-1)^pp} \otimes
\overline{\bigotimes_{p=1}^n \bigotimes_{s=0}^{p-1} \Big(\det H^\bullet\big(D,G^p_s\big)\Big)^{(-1)^pp}} \;.
\end{equation}
Let $\big\lVert\sigma_{G^\bullet_0}\big\rVert_{\lambda_\mathrm{dR}(G^\bullet_\bullet)}$
be the norm of $\sigma_{G^\bullet_0}\in\lambda_\mathrm{dR}(G^\bullet_\bullet)$ with respect to the metrics \eqref{eq3Q4-thm-bl-bcovt}.

Let $\big\lVert\sigma_X\big\rVert_{\lambda_\mathrm{dR}(X)}$ be the norm of $\sigma_X$ with respect to the metrics \eqref{eq3Q1-thm-bl-bcovt}.
Let $\big\lVert\sigma_{X'}\big\rVert_{\lambda_\mathrm{dR}(X')}$ be the norm of $\sigma_{X'}$ with respect to the metrics \eqref{eq3Q2-thm-bl-bcovt}.
By \eqref{eqmuprod-thm-bl-bcovt} and \eqref{eq2-thm-bl-bcovt},
we have
\begin{align}
\label{eq30-thm-bl-bcovt}
\begin{split}
& \log \big\lVert\sigma_{X'}\big\rVert_{\lambda_\mathrm{dR}(X')}
- \log \big\lVert\sigma_X\big\rVert_{\lambda_\mathrm{dR}(X)}
- \log \big\lVert\sigma_{G^\bullet_0}\big\rVert_{\lambda_\mathrm{dR}(G^\bullet_\bullet)} \\
& = \sum_{p=1}^n (-1)^pp \Big( \log \big\lVert\nu_p\big\rVert^2 + \sum_{s=0}^{p-1} \log \big\lVert\mu_{p,s}\big\rVert^2 \Big)  \;.
\end{split}
\end{align}
By Definition \ref{def-bcov-torsion} and \eqref{eq30-thm-bl-bcovt},
we have
\begin{align}
\label{eq3-thm-bl-bcovt}
\begin{split}
& \tau_\mathrm{BCOV}(X',\omega') - \tau_\mathrm{BCOV}(X,\omega) \\
& = \log \big\lVert\sigma_{G^\bullet_0}\big\rVert_{\lambda_\mathrm{dR}(G^\bullet_\bullet)} +
\sum_{p=1}^n (-1)^pp \Big( \log \big\lVert\nu_p\big\rVert^2 + \sum_{s=0}^{p-1} \log \big\lVert\mu_{p,s}\big\rVert^2 \Big) \;.
\end{split}
\end{align}

\noindent\textbf{Step 4.}
We conclude.

For ease of notations,
we denote
\begin{equation}
\label{eq41-thm-bl-bcovt}
\alpha_{p,s} = \log \big\lVert\mu_{p,s}\big\rVert^2 \;.
\end{equation}
Applying Theorem \ref{thm-im} to the short exact sequence \eqref{eq0Ffilt-thm-bl-bcovt} while using the second line in \eqref{eq3Q0-thm-bl-bcovt},
we see that $\alpha_{p,s}$ is determined by $\big( U', \omega'\big|_{U'}, g^{F^p_s}\big|_{U'}, g^{F^p_{s+1}}\big|_{U'} \big)$.
We denote
\begin{equation}
\label{eq42-thm-bl-bcovt}
\alpha_p = \sum_{s=0}^{p-1} \alpha_{p,s} \;.
\end{equation}
We remark that for $s = 1,\cdots,p-1$,
the contributions of the metric $\big\lVert\cdot\big\rVert_{\det H^\bullet(X',F^p_s)}$ (see \eqref{eq3Q6-thm-bl-bcovt}) to $\alpha_{p,s-1}$ and $\alpha_{p,s}$ cancel with each other.
Thus $\alpha_p$ is independent of $\big(g^{F^p_s}\big)_{s=1,\cdots,p-1}$.
Hence $\alpha_p$ is determined by $\big( U', \omega'\big|_{U'}, g^{F^p_0}\big|_{U'}, g^{F^p_p}\big|_{U'} \big)$.
Now,
applying the first line in \eqref{eq3Q0-thm-bl-bcovt},
we see that $\alpha_p$ is determined by $\big(U, U', \omega\big|_U, \omega'\big|_{U'}\big)$.

For ease of notations,
we denote
\begin{equation}
\label{eq43-thm-bl-bcovt}
\beta_p = \log \big\lVert\nu_p\big\rVert^2 \;,
\end{equation}
Applying Theorem \ref{thm-bl} with $E = \Lambda^p(T^*X)$ while using \eqref{eq-f-omega},
we see that $\beta_p$ is determined by $\big(U, U', \omega\big|_U, \omega'\big|_{U'}\big)$.

By \eqref{eq3-thm-bl-bcovt}-\eqref{eq43-thm-bl-bcovt},
we have
\begin{equation}
\label{eq44-thm-bl-bcovt}
\tau_\mathrm{BCOV}(X',\omega') - \tau_\mathrm{BCOV}(X,\omega)
= \log \big\lVert\sigma_{G^\bullet_0}\big\rVert_{\lambda_\mathrm{dR}(G^\bullet_\bullet)}
+ \sum_{p=1}^n (-1)^pp \big( \alpha_p + \beta_p \big) \;.
\end{equation}
Here
\begin{itemize}
\item[-] the section $\sigma_{G^\bullet_0}\in\lambda_\mathrm{dR}(G^\bullet_\bullet)$ is determined by $D\subseteq U'$ and its normal bundle;
\item[-] the Quillen metric $\big\lVert\cdot\big\rVert_{\lambda_\mathrm{dR}(G^\bullet_\bullet)}$ is determined by $\omega'\big|_{U'}$;
\item[-] the real number $\alpha_p$ is determined by $\big(U, U', \omega\big|_U, \omega'\big|_{U'}\big)$;
\item[-] the real number $\beta_p$ is determined by $\big(U, U', \omega\big|_U, \omega'\big|_{U'}\big)$.
\end{itemize}
In conclusion,
the right hand side of \eqref{eq44-thm-bl-bcovt} is determined by $\big(U, U', \omega\big|_U, \omega'\big|_{U'}\big)$.
This completes the proof.
\end{proof}

Let $\pi: \mathscr{U} \rightarrow \C$ be a holomorphic submersion between complex manifolds.
Let $\mathscr{Y} \subseteq \mathscr{U}$ be a closed complex submanifold.
We assume that $\pi\big|_\mathscr{Y}: \mathscr{Y} \rightarrow \C$ is a holomorphic submersion with compact fiber.
For $z\in \C$,
we denote $U_z = \pi^{-1}(z)$ and $Y_z = U_z \cap \mathscr{Y}$.
Assume that for any $z\in\C$,
$U_z$ can be extended to a compact K{\"a}hler manifold.
More precisely,
there exist a compact K{\"a}hler manifold $X_z$ and a holomorphic embedding $i_z: U_z \hookrightarrow X_z$ whose image is open.
Here $\big\{ X_z \;:\; z\in\C \big\}$ is just a set of complex manifolds parameterized by $\C$.
The topology of $X_z$ may vary as $z$ varies.
We identify $U_z$ with $i_z(U_z)\subseteq X_z$.
Let $f_z: X_z' \rightarrow X_z$ be the blow-up along $Y_z$.
Set $U_z' = f_z^{-1}(U_z) \subseteq X_z'$.
Let
\begin{equation}
\big(\omega_z\in\Omega^{1,1}(X_z)\big)_{z\in \C} \;,\hspace{4mm}
\big(\omega_z'\in\Omega^{1,1}(X_z')\big)_{z\in \C}
\end{equation}
be K{\"a}hler forms.
We assume that $\big(\omega_z\big|_{U_z}\big)_{z\in \C}$ and $\big(\omega_z'\big|_{U_z'}\big)_{z\in \C}$ are smooth families.
We further assume that
\begin{equation}
\omega_z'\big|_{X_z'\backslash U_z'} = f_z^*\big(\omega_z\big|_{X_z\backslash U_z}\big)
\hspace{4mm}\text{for } z\in \C \;.
\end{equation}

\begin{thm}
\label{thm2-bl-bcovt}
The function $z \mapsto \tau_\mathrm{BCOV}(X_z',\omega_z') - \tau_\mathrm{BCOV}(X_z,\omega_z)$ is continuous.
\end{thm}
\begin{proof}
We proceed in the same way as in the proof of Theorem \ref{thm-bl-bcovt}.
Each object constructed becomes a function of $z\in \C$.
In particular,
the identity \eqref{eq44-thm-bl-bcovt} becomes
\begin{equation}
\label{eq-thm2-bl-bcovt}
\tau_\mathrm{BCOV}(X_z',\omega_z') - \tau_\mathrm{BCOV}(X_z,\omega_z)
= \log \big\lVert\sigma_{G^\bullet_0}\big\rVert_{\lambda_\mathrm{dR}(G^\bullet_\bullet),z}
+ \sum_{p=1}^n (-1)^pp \big( \alpha_{p,z} + \beta_{p,z} \big) \;.
\end{equation}
From Remark \ref{rk-im}, \ref{rk-bl} and the last paragraph in the proof of Theorem \ref{thm-bl-bcovt},
we see that each term on the right hand side of \eqref{eq-thm2-bl-bcovt} is a continuous function of $z$.
This completes the proof.
\end{proof}

\section{BCOV invariant}

\subsection{Several meromorphic sections}
\label{subsect-ms}

Let $X$ be a compact complex manifold.
Let $K_X$ be the canonical line bundle of $X$.
Let $d$ be a non-zero integer.
Let $K_X^d$ be the $d$-th tensor power of $K_X$.
Let $\gamma\in\mathscr{M}(X,K_X^d)$ be an invertible element.
We denote
\begin{equation}
\mathrm{div}(\gamma) = D = \sum_{j=1}^l m_j D_j \;,
\end{equation}
where $m_j\in\Z\backslash\{0\}$,
$D_1,\cdots,D_l \subseteq X$ are mutually distinct and irreducible.
We assume that $D$ is of simple normal crossing support (see Definition \ref{def-snc}).

For $J\subseteq\big\{1,\cdots,l\big\}$,
let $D_J \subseteq X$ be as in \eqref{introeq-def-wJ-DJ}.
For $j\in J\subseteq\big\{1,\cdots,l\big\}$,
let $L_{J,j}$ be the normal line bundle of $D_J \hookrightarrow D_{J\backslash\{j\}}$.
Set
\begin{equation}
\label{eq-def-KJ}
K_J = K_X^d\big|_{D_J} \otimes \bigotimes_{j\in J} L_{J,j}^{-m_j}
= K_{D_J}^d \otimes \bigotimes_{j\in J} L_{J,j}^{-m_j-d} \;,
\end{equation}
which is a holomorphic line bundle over $D_J$.
In particular,
we have $K_\emptyset = K_X^d$.

Recall that $\mathrm{Res}_\cdot(\cdot)$ was defined in Definition \ref{def-res}.
By \eqref{eq-res-commute},
there exist
\begin{equation}
\Big(\gamma_J \in \mathscr{M}(D_J,K_J)\Big)_{J \subseteq\{1,\cdots,l\}}
\end{equation}
such that
\begin{equation}
\label{eq-def-gammaJ}
\gamma_\emptyset = \gamma \;,\hspace{4mm}
\gamma_J = \mathrm{Res}_{D_J}(\gamma_{J\backslash\{j\}})
\hspace{2.5mm} \text{ for } j\in J\subseteq\big\{1,\cdots,l\big\} \;.
\end{equation}
By \eqref{eq-res-div},
we have
\begin{equation}
\label{eq-div-gammaJ}
\mathrm{div}(\gamma_J) = \sum_{j\notin J} m_j D_{J\cup\{j\}} \;.
\end{equation}

\subsection{Construction of BCOV invariant}
\label{subsect-bcov}

We will use the notations in \textsection \ref{subsect-ms}.
We further assume that $X$ is K{\"a}hler and $m_j\neq -d$ for $j=1,\cdots,l$.
Then $(X,\gamma)$ is a $d$-Calabi-Yau pair (see Definition \ref{introdef-cy-pair}).

Let $\omega$ be a K{\"a}hler form on $X$.
Let $\big|\cdot\big|_{K_{D_J},\omega}$ be the metric on $K_{D_J}$ induced by $\omega$.
Let $\big|\cdot\big|_{L_{J,j},\omega}$ be the metric on $L_{J,j}$ induced by $\omega$.
Let $\big|\cdot\big|_{K_J,\omega}$ be the metric on $K_J$
induced by $\big|\cdot\big|_{K_{D_J},\omega}$ and $\big|\cdot\big|_{L_{J,j},\omega}$ via \eqref{eq-def-KJ}.

We will use the notations in \eqref{eq-def-QQ}.
For $J\subseteq\big\{1,\cdots,l\big\}$,
let $|J|$ be the number of elements in $J$,
let $g^{TD_J}_\omega$ be the metric on $TD_J$ induced by $\omega$,
let $c_k\big(TD_J,g^{TD_J}_\omega\big) \in Q^{D_J}$ be $k$-th Chern form of $\big(TD_J,g^{TD_J}_\omega\big)$.
Let $n = \dim X$.
Set
\begin{equation}
\label{eq-def-aJ}
a_J(\gamma,\omega) =
\frac{1}{12} \int_{D_J} c_{n-|J|}\big(TD_J,g^{TD_J}_\omega\big)
\log \big|\gamma_J\big|^{2/d}_{K_J,\omega} \;.
\end{equation}

We consider the short exact sequence of holomorphic vector bundles over $D_J$,
\begin{equation}
\label{eq-ses-TJ}
0 \rightarrow TD_J \rightarrow TD_{J\backslash\{j\}}\big|_{D_J} \rightarrow L_{J,j} \rightarrow 0 \;.
\end{equation}
Let
\begin{equation}
\widetilde{c}\Big(TD_J,TD_{J\backslash\{j\}}\big|_{D_J},g^{TD_{J\backslash\{j\}}}_\omega\big|_{D_J}\Big)
\in Q^{D_J}/Q^{D_J,0}
\end{equation}
be the Bott-Chern form \eqref{eq-def-BCc-ses}
with $0\rightarrow E' \rightarrow E \rightarrow E''$ replaced by \eqref{eq-ses-TJ}
and $g^E$ replaced by $g^{TD_{J\backslash\{j\}}}_\omega\big|_{D_J}$.
Set
\begin{equation}
\label{eq-def-bIk}
b_{J,j}(\omega) =
\frac{1}{12} \int_{D_J}
\widetilde{c}\Big(TD_J,TD_{J\backslash\{j\}}\big|_{D_J},g^{TD_{J\backslash\{j\}}}_\omega\big|_{D_J}\Big) \;.
\end{equation}

Let $w_d^J$ be as in \eqref{introeq-def-wJ-DJ}.
Recall that $\tau_\mathrm{BCOV}(\cdot,\cdot)$ was defined in Definition \ref{def-bcov-torsion}.
For ease of notation,
we denote $\tau_\mathrm{BCOV}(D_J,\omega) = \tau_\mathrm{BCOV}\big(D_J,\omega\big|_{D_J}\big)$.
We define
\begin{equation}
\label{eq-def-tau-gamma-omega}
\tau_d(X,\gamma,\omega) = \sum_{J \subseteq\{1,\cdots,l\}}
w_d^J \bigg( \tau_\mathrm{BCOV}(D_J,\omega) - a_J(\gamma,\omega) - \sum_{j\in J} \frac{m_j+d}{d} b_{J,j}(\omega) \bigg) \;.
\end{equation}

\begin{thm}
\label{thm-ind-omega}
The real number $\tau_d(X,\gamma,\omega)$ is independent of $\omega$.
\end{thm}
\begin{proof}
Let $\big(\omega_s\big)_{s\in\CP^1}$ be a smooth family of K{\"a}hler forms on $X$ parameterized by $\CP^1$.
It is sufficient to show that
$\tau_d(X,\gamma,\omega_s)$ is independent of $s$.

We will view the terms
involved in \eqref{eq-def-tau-gamma-omega}
as smooth functions on $\CP^1$,
i.e.,
\begin{align}
\label{eq-notation-pf-thm-ind-omega}
\begin{split}
& \tau_d(X,\gamma,\omega):
s \mapsto \tau_d(X,\gamma,\omega_s) \;,\hspace{4mm}
\tau_\mathrm{BCOV}\big(D_J,\omega\big):
s \mapsto \tau_\mathrm{BCOV}\big(D_J,\omega_s\big) \;,\hspace{4mm}\text{etc} \;.
\end{split}
\end{align}
We will view $TD_J$ and $L_{J,j}$ as holomorphic vector bundles over $D_J \times \CP^1$.
Let $g^{TD_J}_\omega$ and $g^{L_{J,j}}_\omega$
be metrics on $TD_J$ and $L_{J,j}$ induced by $\big(\omega_s\big)_{s\in\CP^1}$.
More precisely,
the restrictions
$g^{TD_J}_\omega\big|_{D_J\times\{s\}}$ and $g^{L_{J,j}}_\omega\big|_{D_J\times\{s\}}$
are induced by $\omega_s$.
By \cite[Theorem 1.6]{z},
we have
\begin{equation}
\label{eq-taup-pf-thm-ind-omega}
\frac{\overline{\partial}\partial}{2\pi i} \tau_\mathrm{BCOV}\big(D_J,\omega\big)
= \frac{1}{12} \int_{D_J}
c_{n-|J|}\big(TD_J,g^{TD_J}_\omega\big) c_1\big(TD_J,g^{TD_J}_\omega\big) \;.
\end{equation}
Similarly to \cite[(2.9)]{z},
by the Poicar{\'e}-Lelong formula,
\eqref{eq-def-KJ}, \eqref{eq-div-gammaJ} and \eqref{eq-def-aJ},
we have
\begin{align}
\label{eq-ap-pf-thm-ind-omega}
\begin{split}
\frac{\overline{\partial}\partial}{2\pi i} a_J(\gamma,\omega)
& = \frac{1}{12d} \int_{D_J} c_{n-|J|}\big(TD_J,g^{TD_J}_\omega\big)
\Big( - c_1\big(K_J,\big|\cdot\big|_{K_J,\omega}\big) + \delta_{\mathrm{div}(\gamma_J)} \Big) \\
& = \frac{1}{12} \int_{D_J} c_{n-|J|}\big(TD_J,g^{TD_J}_\omega\big) c_1\big(TD_J,g^{TD_J}_\omega\big) \\
& \hspace{4mm} +\sum_{j\in J} \frac{m_j+d}{12d} \int_{D_J} c_{n-|J|}\big(TD_J,g^{TD_J}_\omega\big) c_1\big(L_{J,j},\big|\cdot\big|_{L_{J,j},\omega}\big) \\
& \hspace{4mm} + \sum_{j\notin J} \frac{m_j}{12d} \int_{D_{J\cup\{j\}}} c_{n-|J|}\big(TD_J,g^{TD_J}_\omega\big) \;.
\end{split}
\end{align}
Similarly to \cite[(2.10)]{z},
by \eqref{eq-def-BC}, \eqref{eq-def-BCc-ses} and \eqref{eq-def-bIk},
we have
\begin{align}
\label{eq-bp-pf-thm-ind-omega}
\begin{split}
\frac{\overline{\partial}\partial}{2\pi i} b_{J,j}(\omega) & =
\frac{1}{12} \int_{D_J} c_{n-|J|+1}\big(TD_{J\backslash\{j\}},g^{TD_{J\backslash\{j\}}}_\omega\big) \\
& \hspace{4mm} - \frac{1}{12} \int_{D_J} c_{n-|J|}\big(TD_J,g^{TD_J}_\omega\big)c_1\big(L_{J,j},g^{L_{J,j}}_\omega\big) \;.
\end{split}
\end{align}
By \eqref{eq-taup-pf-thm-ind-omega}-\eqref{eq-bp-pf-thm-ind-omega},
we have
\begin{align}
\label{eq-p-pf-thm-ind-omega}
\begin{split}
& \frac{\overline{\partial}\partial}{2\pi i}
\bigg( \tau_\mathrm{BCOV}(D_J,\omega) - a_J(\gamma,\omega) - \sum_{k\in J} \frac{m_j+d}{d} b_{J,j}(\omega) \bigg) \\
& = - \sum_{j\in J}  \frac{m_j+d}{12d} \int_{D_J} c_{n-|J|+1}\big(TD_{J\backslash\{j\}},g^{TD_{J\backslash\{j\}}}_\omega\big) \\
& \hspace{4mm} -  \sum_{j\notin J} \frac{m_j}{12d} \int_{D_{J\cup\{j\}}} c_{n-|J|}\big(TD_J,g^{TD_J}_\omega\big) \;.
\end{split}
\end{align}
From \eqref{introeq-def-wJ-DJ}, \eqref{eq-def-tau-gamma-omega} and \eqref{eq-p-pf-thm-ind-omega},
we obtain $\overline{\partial}\partial \tau_d(X,\gamma,\omega) = 0$.
Hence $s\mapsto\tau_d(X,\gamma,\omega_s)$ is constant on $\CP^1$.
This completes the proof.
\end{proof}

\begin{defn}
\label{def-bcov-inv}
The BCOV invariant of $(X,\gamma)$ is defined by
\begin{equation}
\tau_d(X,\gamma) = \tau_d(X,\gamma,\omega) \;.
\end{equation}
By Theorem \ref{thm-ind-omega},
$\tau_d(X,\gamma)$ is well-defined.
\end{defn}

\begin{prop}
\label{intro-prop-tau-r}
For a non-zero integer $r$,
let $\gamma^r\in\mathscr{M}(X,K_X^{rd})$ be the $r$-th tensor power of $\gamma$.
Then $(X,\gamma^r)$ is a $rd$-Calabi-Yau pair and
\begin{equation}
\label{intro-eq-tau-r}
\tau_{rd}(X,\gamma^r) = \tau_d(X,\gamma) \;.
\end{equation}
\end{prop}
\begin{proof}
Once we replace $\gamma$ by $\gamma^r$,
each $\gamma_J$ is replaced by $\gamma_J^r$.
We can directly verify that
\begin{equation}
\label{eq1-pf-prop-tau-r}
\tau_{rd}(X,\gamma^r,\omega) = \tau_d(X,\gamma,\omega) \;.
\end{equation}
From Definition \ref{def-bcov-inv} and \eqref{eq1-pf-prop-tau-r},
we obtain \eqref{intro-eq-tau-r}.
This completes the proof.
\end{proof}

Recall that $\chi_d(\cdot,\cdot)$ was defined in Definition \ref{def-chi-pair}.

\begin{prop}
\label{intro-prop-tau-z}
For $z\in\C^*$,
we have
\begin{equation}
\label{intro-eq-tau-z}
\tau_d(X,z\gamma) = \tau_d(X,\gamma) - \frac{\chi_d(X,D)}{12}\log|z|^{2/d} \;.
\end{equation}
\end{prop}
\begin{proof}
Once we replace $\gamma$ by $z\gamma$,
each $\gamma_J$ is replaced by $z\gamma_J$.
By \eqref{eq-def-aJ},
we have
\begin{equation}
\label{eq1-pf-prop-tau-z}
a_J(z\gamma,\omega) - a_J(\gamma,\omega) = \frac{\chi(D_J)}{12} \log \big|z\big|^{2/d} \;.
\end{equation}
By Definition \ref{def-chi-pair}, \eqref{eq-def-tau-gamma-omega} and \eqref{eq1-pf-prop-tau-z},
we have
\begin{equation}
\label{eq2-pf-prop-tau-z}
\tau_d(X,z\gamma,\omega) - \tau_d(X,\gamma,\omega) = - \frac{\chi_d(X,D)}{12} \log \big|z\big|^{2/d} \;.
\end{equation}
From Definition \ref{def-bcov-inv} and \eqref{eq2-pf-prop-tau-z},
we obtain \eqref{intro-eq-tau-z}.
This completes the proof.
\end{proof}

\begin{proof}[Proof of Theorem \ref{thm-curvature}]
Since $\pi: \mathscr{X} \rightarrow S$ is locally K{\"a}hler,
for any $s_0\in S$,
there exist an open subset $s_0\in U \subseteq S$ and a K{\"a}hler form $\omega$ on $\pi^{-1}(U)$.
For $s\in U$,
we denote $\omega_s = \omega\big|_{X_s}$.
Similarly to the proof of Theorem \ref{thm-ind-omega},
we view the terms involved in \eqref{eq-def-tau-gamma-omega} as smooth functions on $U$.

Though the fibration $\pi^{-1}(U) \rightarrow U$ is not necessarily trivial,
the identities \eqref{eq-ap-pf-thm-ind-omega} and \eqref{eq-bp-pf-thm-ind-omega} still hold.
On the other hand,
by \cite[Theorem 1.6]{z},
we have
\begin{equation}
\label{eq1-pf-thm-curvature}
\frac{\overline{\partial}\partial}{2\pi i} \tau_\mathrm{BCOV}\big(D_J,\omega\big)
= \omega_{H^\bullet(D_J)} + \frac{1}{12} \int_{D_J} c_{n-|J|}\big(TD_J,g^{TD_J}_\omega\big) c_1\big(TD_J,g^{TD_J}_\omega\big) \;.
\end{equation}
By \eqref{introeq-def-wJ-DJ}, \eqref{eq-def-tau-gamma-omega}, \eqref{eq-ap-pf-thm-ind-omega}, \eqref{eq-bp-pf-thm-ind-omega} and \eqref{eq1-pf-thm-curvature},
we have
\begin{equation}
\label{eq2-pf-thm-curvature}
\frac{\overline{\partial}\partial}{2\pi i} \tau_d(X,\gamma,\omega)\Big|_U
= \sum_{J\subseteq\{1,\cdots,l\}} w_d^J \omega_{H^\bullet(D_J)} \;.
\end{equation}
From Definition \ref{def-bcov-inv} and \eqref{eq2-pf-thm-curvature},
we obtain \eqref{introeq-thm-curvature}.
This completes the proof.
\end{proof}

\subsection{BCOV invariant of projective bundle}

Let $Y$ be a compact K{\"a}hler manifold.
Let $N$ be a holomorphic vector bundle of rank $r$ over $Y$.
Let $\mathbb{1}$ be the trivial line bundle over $Y$.
Set
\begin{equation}
X = \mathbb{P}(N\oplus\mathbb{1}) \;.
\end{equation}
Let $\pi: X \rightarrow Y$ be the canonical projection.

Let $s\in\{0,\cdots,r\}$.
Let $\big(L_k\big)_{k=1,\cdots,s}$ be holomorphic line bundles over $Y$.
We assume that there is a surjection between holomorphic vector bundles
\begin{equation}
\label{eq-Nsurj}
N \rightarrow L_1 \oplus \cdots \oplus L_s \;.
\end{equation}
Let $N^*$ be the dual of $N$.
Taking the dual of \eqref{eq-Nsurj},
we get
\begin{equation}
\label{eq-Ninj}
L_1^{-1} \oplus \cdots \oplus L_s^{-1} \hookrightarrow N^* \;.
\end{equation}
Let $d$ be a non-zero integer.
Let $m_1,\cdots,m_s$ be positive integers.
Let
\begin{equation}
\label{eq-def-gammaY}
\gamma_Y \in \mathscr{M}\big(Y,\big(K_Y\otimes \det N^*\big)^d \otimes L_1^{-m_1} \otimes \cdots \otimes L_s^{-m_s}\big)
\end{equation}
be an invertible element.
We assume that
\begin{itemize}
\item[-] $\mathrm{div}(\gamma_Y)$ is of simple normal crossing support;
\item[-] $\mathrm{div}(\gamma_Y)$ does not possess component of multiplicity $-d$.
\end{itemize}

Denote $m = m_1+\cdots+m_s$.
Let $S^mN^*$ be the $m$-th symmetric tensor power of $N^*$.
By \eqref{eq-Ninj} and \eqref{eq-def-gammaY},
we have
\begin{equation}
\label{eq-gammaY}
\gamma_Y \in \mathscr{M}\big(Y,\big(K_Y\otimes \det N^*\big)^d \otimes S^m N^*\big) \;.
\end{equation}
Let $\mathcal{N}$ be the total space of $N$.
We have
\begin{equation}
\label{eq-KX-pi}
X = \mathcal{N} \cup \mathbb{P}(N) \;,\hspace{4mm}
K_X\big|_\mathcal{N} = \pi^* \big( K_Y\otimes \det N^* \big) \;.
\end{equation}
We may view a section of $S^mN^*$ as a function on $\mathcal{N}$.
By \eqref{eq-gammaY} and \eqref{eq-KX-pi},
$\gamma_Y$ may be viewed as an element of $\mathscr{M}(\mathcal{N},K_X^d)$.
Let
\begin{equation}
\gamma_X\in\mathscr{M}(X,K_X^d)
\end{equation}
be such that $\gamma_X\big|_\mathcal{N} = \gamma_Y$.

For $j = 1,\cdots,s$,
let $N\rightarrow L_j$ be the composition of the map \eqref{eq-Nsurj} and the canonical projection $L_1 \oplus \cdots \oplus L_s \rightarrow L_j$.
Set
\begin{equation}
N_j = \mathrm{Ker}\big(N\rightarrow L_j\big) \;,\hspace{4mm}
X_j = \mathbb{P}(N_j\oplus\mathbb{1}) \subseteq X \;,\hspace{4mm}
X_\infty = \mathbb{P}(N) \subseteq X \;.
\end{equation}
We denote
\begin{equation}
\mathrm{div}(\gamma_Y) = \sum_{j=s+1}^l m_j Y_j \;,
\end{equation}
where $Y_j\subseteq Y$ are mutually distinct and irreducible.
For $j=s+1,\cdots,l$,
set
\begin{equation}
X_j = \pi^{-1}(Y_j) \subseteq X \;.
\end{equation}
Denote
\begin{equation}
\label{eq-def-minf}
m_\infty = - m_1 - \cdots - m_s - rd - d \;.
\end{equation}
We have
\begin{equation}
\label{eq-div-gammaX}
\mathrm{div}(\gamma_X)
= \pi^* \mathrm{div}(\gamma_Y)  + m_\infty X_\infty + \sum_{j=1}^s m_j X_j
= m_\infty X_\infty + \sum_{j=1}^l m_j X_j \;,
\end{equation}
which is of simple normal crossing support.
Hence $(X,\gamma_X)$ is a $d$-Calabi-Yau pair.

For $y\in Y$,
we denote $Z_y = \pi^{-1}(y)$.
Let $K_{Y,y}$ be the fiber of $K_Y$ at $y\in Y$.
We have
\begin{equation}
\label{eq-KXYZ}
K_X\big|_{Z_y} = K_{Z_y} \otimes \pi^* K_{Y,y} \;.
\end{equation}
For $y\in Y \backslash \bigcup_{j=s+1}^l Y_j$,
there exist $\gamma_{Z_y}\in\mathscr{M}(Z_y,K_{Z_y}^d)$ and $\eta_y\in K_{Y,y}^d$ such that
\begin{equation}
\gamma_X\big|_{Z_y} = \gamma_{Z_y} \otimes \pi^* \eta_y \;.
\end{equation}
Then $(Z_y,\gamma_{Z_y})$ is a $d$-Calabi-Yau pair,
which is independent of $y$ up to isomorphism.
We may omit the index $y$ as long as there is no confusion.
We remark that $(Z,\gamma_Z)$ is isomorphic to $(\CP^r,\gamma_{r,m_1,\cdots,m_s})$ constructed in the paragraph containing \eqref{introeq-gammaZ}.

Recall that $\chi_d(\cdot,\cdot)$ was defined in Definition \ref{def-chi-pair}.

\begin{lemme}
\label{lem-vanish-chi}
The following identity holds,
\begin{equation}
\label{eq-lem-vanish-chi}
\chi_d(Z,\gamma_Z) = 0 \;.
\end{equation}
\end{lemme}
\begin{proof}
Set
\begin{equation}
\label{eq1-pf-lem-vanish-chi}
f(t) = t^{r-s} \prod_{j \in \{1,\cdots,s,\infty\}} \Big( t - \frac{m_j}{m_j+d} \Big) \;.
\end{equation}
For $J \subseteq \{1,\cdots,s,\infty\}$,
let $w_d^J$ be as in \eqref{introeq-def-wJ-DJ}.
By \eqref{introeq-def-chi-pair}, \eqref{introeq2-def-chi-pair} and the fact that $\chi(\CP^k) = k+1$,
we have
\begin{equation}
\label{eq2-pf-lem-vanish-chi}
\chi_d(Z,\gamma_Z) = \sum_{J \subseteq \{1,\cdots,s,\infty\}} w_d^J (r+1-|J|) = f'(1) \;.
\end{equation}
On the other hand,
we have
\begin{equation}
\label{eq3-pf-lem-vanish-chi}
\frac{f'(1)}{f(1)} = r - s + \sum_{j \in \{1,\cdots,s,\infty\}} \Big( 1 - \frac{m_j}{m_j+d} \Big)^{-1}
= \frac{m_1 + \cdots + m_s + m_\infty}{d} + r + 1 \;.
\end{equation}
From \eqref{eq-def-minf}, \eqref{eq2-pf-lem-vanish-chi} and \eqref{eq3-pf-lem-vanish-chi},
we obtain \eqref{eq-lem-vanish-chi}.
This completes the proof.
\end{proof}

\begin{thm}
\label{thm-tau-bl}
The following identity holds,
\begin{equation}
\label{eq-thm-tau-bl}
\tau_d(X,\gamma_X) = \chi_d(Y,\gamma_Y) \tau_d(Z,\gamma_Z) \;.
\end{equation}
\end{thm}
\begin{proof}
The proof consists of several steps.

\noindent\textbf{Step 0.} We introduce several notations.

We denote $A = \{s+1,\cdots,l\}$ and $B = \{1,\cdots,s,\infty\}$.
For $I\subseteq A$ and $J\subseteq B$,
set
\begin{equation}
\label{eq0X-pf-thm-tau-bl}
Y_I = Y \cap \bigcap_{j\in I} Y_j \;,\hspace{4mm}
X_{I,J} = X \cap \bigcap_{j\in I \cup J} X_j \;,\hspace{4mm}
X_I = X_{I,\emptyset} \;,\hspace{4mm}
X_J = X_{\emptyset,J} \;.
\end{equation}
For $y\in Y$ and $J\subseteq B$,
set
\begin{equation}
\label{eq0Z-pf-thm-tau-bl}
Z_{J,y} = Z_y \cap X_J \;.
\end{equation}
Note that $Z_{J,y}$ is independent of $y$ up to isomorphism,
we may omit the index $y$ as long as there is no confusion.
We remark that $\pi\big|_{X_{I,J}}: X_{I,J} \rightarrow Y_I$ is a fibration with fiber $Z_J$.

Let $\omega_X$ be a K{\"a}hler form on $X$ such that Lemma \ref{lem2-kahler-bundle} holds.
Let $\omega_Y$ be a K{\"a}hler form on $Y$.
For $\e>0$,
set
\begin{equation}
\label{eqomega-pf-thm-tau-bl}
\omega_\e = \omega_X + \frac{1}{\e}\pi^*\omega_Y \;.
\end{equation}

For $I\subseteq A$, $J\subseteq B$ and $j \in (A \cup B) \backslash (I \cup J)$,
let $a_{I,J}(\gamma_X,\omega_\e)$ and $b_{I,J,j}(\omega_\e)$
be as in \eqref{eq-def-aJ} and \eqref{eq-def-bIk}
with $(X,\gamma,\omega)$ replaced by $(X,\gamma_X,\omega_\e)$ and $J$ replaced by $I \cup J$.
Let $w_d^I$ be as in \eqref{introeq-def-wJ-DJ} with $J$ replaced by $I$.
By Definition \ref{def-bcov-inv}, \eqref{introeq-def-wJ-DJ} and \eqref{eq-def-tau-gamma-omega},
we have
\begin{align}
\label{eq0tau-pf-thm-tau-bl}
\begin{split}
\tau_d(X,\gamma_X)
& = \sum_{I\subseteq A} \sum_{J\subseteq B}  w_d^I w_d^J \tau_\mathrm{BCOV}(X_{I,J},\omega_\e)
- \sum_{I\subseteq A} \sum_{J\subseteq B}  w_d^I w_d^J a_{I,J}(\gamma_X,\omega_\e) \\
& \hspace{50mm} - \sum_{I\subseteq A} \sum_{J\subseteq B} \sum_{j\in I \cup J} w_d^I w_d^J \frac{m_j+d}{d} b_{I,J,j}(\omega_\e) \;.
\end{split}
\end{align}

\noindent\textbf{Step 1.} We estimate $\tau_\mathrm{BCOV}(X_{I,J},\omega_\e)$.

For $y\in Y$,
we denote $\omega_{Z_y} = \omega_X\big|_{Z_y}$.
Since $\omega_X$ satisfies Lemma \ref{lem2-kahler-bundle},
for any $J\subseteq B$,
$\big(Z_{J,y},\omega_{Z_y}\big|_{Z_{J,y}}\big)_{y\in Y}$ are mutually isometric.
We may omit the index $y$ as long as there is no confusion.
For ease of notation,
we denote
\begin{equation}
\tau_\mathrm{BCOV}(Y_I,\omega_Y) = \tau_\mathrm{BCOV}\big(Y_I,\omega_Y\big|_{Y_I}\big) \;,\hspace{4mm}
\tau_\mathrm{BCOV}(Z_J,\omega_Z) = \tau_\mathrm{BCOV}\big(Z_J,\omega_Z\big|_{Z_J}\big) \;.
\end{equation}
For $I\subseteq A$ and $J\subseteq B$,
by Theorem \ref{thm-bcov-adiabatic},
as $\e\rightarrow 0$,
\begin{align}
\label{eq11-pf-thm-tau-bl}
\begin{split}
& \tau_\mathrm{BCOV}(X_{I,J},\omega_\e)
- \frac{\chi(Z_J)}{12} \Big( \dim(Y_I)\chi(Y_I) + c_1c_{\dim(Y_I)-1}(Y_I) \Big) \log \e  \\
& \hspace{45mm} \rightarrow \chi(Z_J) \tau_\mathrm{BCOV}(Y_I,\omega_Y) + \chi(Y_I) \tau_\mathrm{BCOV}(Z_J,\omega_Z) \;.
\end{split}
\end{align}
On the other hand,
by Lemma \ref{lem-vanish-chi}, \eqref{introeq-def-chi-pair} and \eqref{introeq2-def-chi-pair},
we have
\begin{equation}
\label{eq12-pf-thm-tau-bl}
\sum_{I\subseteq A} w_d^I \chi(Y_I) = \chi_d(Y,\gamma_Y) \;,\hspace{4mm}
\sum_{J\subseteq B} w_d^J \chi(Z_J) = 0 \;.
\end{equation}
By \eqref{eq11-pf-thm-tau-bl} and \eqref{eq12-pf-thm-tau-bl},
as $\e\rightarrow 0$,
\begin{equation}
\label{eq1-pf-thm-tau-bl}
\sum_{I\subseteq A} \sum_{J\subseteq B} w_d^I w_d^J \tau_\mathrm{BCOV}(X_{I,J},\omega_\e)
\rightarrow \chi_d(Y,\gamma_Y) \sum_{J\subseteq B} w_d^J \tau_\mathrm{BCOV}(Z_J,\omega_Z) \;.
\end{equation}

\noindent\textbf{Step 2.} We estimate $a_{I,J}(\gamma_X,\omega_\e)$.

For $I\subseteq A$ and $J\subseteq B$,
let $K_{I,J}$ be as in \eqref{eq-def-KJ}
with $(X,\gamma)$ replaced by $(X,\gamma_X)$ and $J$ replaced by $I\cup J$.
Then $K_{I,J}$ is a holomorphic line bundle over $X_{I,J}$.
Let
\begin{equation}
\label{eq2gammaIJ-pf-thm-tau-bl}
\gamma_{I,J} \in \mathscr{M}(X_{I,J},K_{I,J})
\end{equation}
be as in \eqref{eq-def-gammaJ} with $(X,\gamma)$ replaced by $(X,\gamma_X)$ and $J$ replaced by $I\cup J$.

Let $U\subseteq Y$ be a small open subset.
Set $\mathcal{U} = \pi^{-1}(U)$.
Recall that $\gamma_Z \in \mathscr{M}(Z,K_Z^d)$ was constructed in the paragraph containing \eqref{eq-KXYZ}.
We fix an identification $\mathcal{U} =  U \times Z$
such that there exists $\eta \in \mathscr{M}(U,K_Y^d)$ satisfying
\begin{equation}
\gamma_X\big|_\mathcal{U} = \mathrm{pr}_1^*\eta \otimes \mathrm{pr}_2^*\gamma_Z \;,
\end{equation}
where $\mathrm{pr}_1: U \times Z \rightarrow U$
and $\mathrm{pr}_2: U \times Z \rightarrow Z$ are canonical projections.

For $I\subseteq A$,
let $K_I$ be as in \eqref{eq-def-KJ}
with $(X,\gamma)$ replaced by $(U,\eta)$.
Then $K_I$ is a holomorphic line bundle over $U \cap Y_I$.
Let
\begin{equation}
\eta_I \in \mathscr{M}(U \cap Y_I,K_I)
\end{equation}
be as in \eqref{eq-def-gammaJ} with $(X,\gamma)$ replaced by $(U,\eta)$.
For $J\subseteq B$,
let $K_J$ be as in \eqref{eq-def-KJ}
with $(X,\gamma)$ replaced by $(Z,\gamma_Z)$.
Then $K_J$ is a holomorphic line bundle over $Z_J$.
Let
\begin{equation}
\gamma_J \in \mathscr{M}(Z_J,K_J)
\end{equation}
be as in \eqref{eq-def-gammaJ} with $(X,\gamma)$ replaced by $(Z,\gamma_Z)$.
By the constructions of $K_{I,J}$ and $\gamma_{I,J}$ in the paragraph containing \eqref{eq2gammaIJ-pf-thm-tau-bl},
we have
\begin{equation}
\label{eq2IJ-pf-thm-tau-bl}
K_{I,J}\big|_{\mathcal{U}\cap X_{I,J}} = \mathrm{pr}_1^* K_I \otimes \mathrm{pr}_2^* K_J \;,\hspace{4mm}
\gamma_{I,J}\big|_{\mathcal{U}\cap X_{I,J}} = \mathrm{pr}_1^*\eta_I \otimes \mathrm{pr}_2^*\gamma_J \;.
\end{equation}

For $I\subseteq A$ and $J\subseteq B$,
let $g^{TX_{I,J}}_\e$ (resp. $g^{TY_I}$, $g^{TZ_J}$)
be the metric on $TX_{I,J}$ (resp. $TY_I$, $TZ_J$) induced by $\omega_\e$ (resp. $\omega_Y$, $\omega_Z$),
let $\big|\cdot\big|_{K_{I,J},\e}$ (resp. $\big|\cdot\big|_{K_I}$, $\big|\cdot\big|_{K_J}$)
be the norm on $K_{I,J}$ (resp. $K_I$, $K_J$) induced by $\omega_\e$ (resp. $\omega_Y$, $\omega_Z$)
in the same way as in the paragraph above \eqref{eq-def-aJ}.
We denote
\begin{align}
\label{eq20-pf-thm-tau-bl}
\begin{split}
a_{I,J}(\mathcal{U},\gamma_X,\omega_\e) & = \frac{1}{12}
\int_{\mathcal{U}\cap X_{I,J}} c\big(TX,g^{TX}_{\e}\big) \log \big|\gamma_{I,J}\big|_{K_{I,J},\e}^{2/d} \;.
\end{split}
\end{align}
Recall that $\omega_\e$ was defined in \eqref{eqomega-pf-thm-tau-bl}.
Since $g^{TX_{I,J}}_\e$ is induced by $\omega_\e$,
by Proposition \ref{prop-adiabatic-bc},
as $\e\rightarrow 0$.
\begin{equation}
\label{eq21-pf-thm-tau-bl}
c\big(TX_{I,J},g^{TX_{I,J}}_\e\big) \rightarrow c\big(TZ_J,g^{TZ_J}\big)\pi^*c\big(TY_I,g^{TY_I}\big) \;.
\end{equation}
Recall that $\eta_I$, $\gamma_J$ and $\gamma_{I,J}$ are linked by \eqref{eq2IJ-pf-thm-tau-bl}.
Since $\big|\cdot\big|_{K_{I,J},\e}$ is induced by $\omega_\e$,
as $\e\rightarrow 0$,
\begin{equation}
\label{eq22-pf-thm-tau-bl}
\log \big|\gamma_{I,J}\big|_{K_{I,J},\e}^2 - \Big(\dim(Y)d + \sum_{j\in I} m_j \Big) \log\e
\rightarrow \log \big|\gamma_J\big|_{K_J}^2 + \log \big|\eta_I\big|_{K_I}^2 \;.
\end{equation}
Let $a_J(\gamma_Z,\omega_Z)$ be as in \eqref{eq-def-aJ}
with $(X,\gamma,\omega)$ replaced by $(Z,\gamma_Z,\omega_Z)$.
More precisely,
\begin{equation}
\label{eq23-pf-thm-tau-bl}
a_J(\gamma_Z,\omega_Z) = \frac{1}{12}
\int_{Z_J} c\big(TZ_J,g^{TZ_J}\big) \log\big|\gamma_Z\big|_{K_J}^{2/d} \;.
\end{equation}
By \eqref{eq20-pf-thm-tau-bl}-\eqref{eq23-pf-thm-tau-bl},
as $\e\rightarrow 0$,
\begin{align}
\label{eq24-pf-thm-tau-bl}
\begin{split}
& a_{I,J}(\mathcal{U},\gamma_X,\omega_\e) - \frac{\chi(Z_J)}{12} \Big(\dim(Y) + \frac{1}{d}\sum_{j\in I} m_j \Big) \log\e \int_{U\cap Y_I} c\big(TY_I,g^{TY_I}\big) \\
& \rightarrow \frac{\chi(Z_J)}{12} \int_{U\cap Y_I} c\big(TY_I,g^{TY_I}\big) \log\big|\eta_I\big|_{K_I}^{2/d}
+ a_J(\gamma_Z,\omega_Z) \int_{U\cap Y_I} c\big(TY_I,g^{TY_I}\big) \;.
\end{split}
\end{align}
By \eqref{eq12-pf-thm-tau-bl} and \eqref{eq24-pf-thm-tau-bl},
as $\e\rightarrow 0$,
\begin{equation}
\label{eq25-pf-thm-tau-bl}
\sum_{I\subseteq A} \sum_{J\subseteq B} w_d^I w_d^J a_{I,J}(\mathcal{U},\gamma_X,\omega_\e)
\rightarrow \sum_{J\subseteq B} w_d^J a_J(\gamma_Z,\omega_Z) \sum_{I\subseteq A} w_d^I \int_{U\cap Y_I}  c\big(TY_I,g^{TY_I}\big) \;.
\end{equation}
The left hand side of \eqref{eq25-pf-thm-tau-bl} yields a measure on $X$,
\begin{equation}
\label{eq26-pf-thm-tau-bl}
\mu_\e : \; \mathcal{U} \mapsto \sum_{I\subseteq A} \sum_{J\subseteq B} w_d^I w_d^J a_{I,J}(\mathcal{U},\gamma_X,\omega_\e) \;,
\end{equation}
The right hand of \eqref{eq25-pf-thm-tau-bl} yields a measure on $Y$,
\begin{equation}
\label{eq27-pf-thm-tau-bl}
\nu : \; U \mapsto \sum_{J\subseteq B} w_d^J a_J(\gamma_Z,\omega_Z) \sum_{I\subseteq A} w_d^I \int_{U\cap Y_I}  c\big(TY_I,g^{TY_I}\big) \;.
\end{equation}
The convergence in \eqref{eq25-pf-thm-tau-bl} is equivalent to the follows:  as $\e \rightarrow 0$,
\begin{equation}
\label{eq28-pf-thm-tau-bl}
\pi_*\mu_\e \rightarrow \nu \;.
\end{equation}
By \eqref{eq12-pf-thm-tau-bl} and \eqref{eq26-pf-thm-tau-bl}-\eqref{eq28-pf-thm-tau-bl},
as $\e\rightarrow 0$,
\begin{equation}
\label{eq2-pf-thm-tau-bl}
\sum_{I\subseteq A} \sum_{J\subseteq B} w_d^I w_d^J a_{I,J}(\gamma_X,\omega_\e) = \mu_\e(X)
\rightarrow \nu(Y) = \chi_d(Y,\gamma_Y) \sum_{J\subseteq B} w_d^J a_J(\gamma_Z,\omega_Z) \;.
\end{equation}

\noindent\textbf{Step 3.} We estimate $b_{I,J,j}(\omega_\e)$.

First we consider the case $j\in I$.
We denote $I'=I\backslash\{j\}$.
By \eqref{eq-def-bIk},
we have
\begin{equation}
\label{eq3Y0-pf-thm-tau-bl}
b_{I,J,j}(\omega_\e) = \frac{1}{12} \int_{X_{I,J}}
\widetilde{c}\Big(TX_{I,J},TX_{I',J}\big|_{X_{I,J}},g^{TX_{I',J}}_\e\big|_{X_{I,J}}\Big) \;.
\end{equation}
By Proposition \ref{prop2-adiabatic-bc},
as $\e\rightarrow 0$,
\begin{align}
\label{eq3Y1-pf-thm-tau-bl}
\begin{split}
& \widetilde{c}\Big(TX_{I,J},TX_{I',J}\big|_{X_{I,J}},g^{TX_{I',J}}_\e\big|_{X_{I,J}}\Big) \\
& \hspace{30mm} \rightarrow c\big(TZ_J,g^{TZ_J}\big) \pi^*\widetilde{c}\Big(TY_I,TY_{I'}\big|_{Y_I},g^{TY_{I'}}\big|_{Y_I}\Big) \;.
\end{split}
\end{align}
By \eqref{eq3Y0-pf-thm-tau-bl} and \eqref{eq3Y1-pf-thm-tau-bl},
as $\e\rightarrow 0$,
\begin{equation}
\label{eq3Y2-pf-thm-tau-bl}
b_{I,J,j}(\omega_\e) \rightarrow \frac{\chi(Z_J) }{12} \int_{Y_I} \widetilde{c}\Big(TY_I,TY_{I'}\big|_{Y_I},g^{TY_{I'}}\big|_{Y_I}\Big)  \;.
\end{equation}
By \eqref{eq12-pf-thm-tau-bl} and \eqref{eq3Y2-pf-thm-tau-bl},
as $\e\rightarrow 0$,
\begin{equation}
\label{eq3Y-pf-thm-tau-bl}
\sum_{I\subseteq A} \sum_{J\subseteq B} \sum_{j\in I} w_d^I w_d^J \frac{m_j+d}{d} b_{I,J,j}(\omega_\e) \rightarrow 0 \;.
\end{equation}

Now we consider the case $j\in J$.
We denote $J'=J\backslash\{j\}$.
By \eqref{eq-def-bIk},
we have
\begin{equation}
\label{eq3Z0-pf-thm-tau-bl}
b_{I,J,j}(\omega_\e) = \frac{1}{12} \int_{X_{I,J}}
\widetilde{c}\Big(TX_{I,J},TX_{I,J'}\big|_{X_{I,J}},g^{TX_{I,J'}}_\e\big|_{X_{I,J}}\Big) \;.
\end{equation}
By Proposition \ref{prop2-adiabatic-bc},
as $\e\rightarrow 0$,
\begin{align}
\label{eq3Z1-pf-thm-tau-bl}
\begin{split}
& \widetilde{c}\Big(TX_{I,J},TX_{I,J'}\big|_{X_{I,J}},g^{TX_{I,J'}}_\e\big|_{X_{I,J}}\Big) \\
& \hspace{30mm} \rightarrow \widetilde{c}\Big(TZ_J,TZ_{J'}\big|_{Z_J},g^{TZ_{J'}}\big|_{Z_J}\Big) \pi^*c\big(TY_I,g^{TY_I}\big) \;.
\end{split}
\end{align}
Let $b_{J,j}(\omega_Z)$ be as in \eqref{eq-def-bIk} with $(X,\gamma,\omega)$ replaced by $(Z,\gamma_Z,\omega_Z)$.
More precisely,
\begin{equation}
\label{eq3Z2a-pf-thm-tau-bl}
b_{J,j}(\omega_Z) = \frac{1}{12} \int_{Z_J} \widetilde{c}\Big(TZ_J,TZ_{J'}\big|_{Z_J},g^{TZ_{J'}}\big|_{Z_J}\Big) \;.
\end{equation}
By \eqref{eq3Z0-pf-thm-tau-bl}-\eqref{eq3Z2a-pf-thm-tau-bl},
as $\e\rightarrow 0$,
\begin{equation}
\label{eq3Z2-pf-thm-tau-bl}
b_{I,J,j}(\omega_\e) \rightarrow \chi(Y_I) b_{J,j}(\omega_Z)  \;.
\end{equation}
By \eqref{eq12-pf-thm-tau-bl} and \eqref{eq3Z2-pf-thm-tau-bl},
as $\e\rightarrow 0$,
\begin{equation}
\label{eq3Z-pf-thm-tau-bl}
\sum_{I\subseteq A} \sum_{J\subseteq B} \sum_{j\in J} w_d^I w_d^J \frac{m_j+d}{d} b_{I,J,j}(\omega_\e)
\rightarrow \chi_d(Y,\gamma_Y) \sum_{J\subseteq B} \sum_{j\in J} w_d^J \frac{m_j+d}{d} b_{J,j}(\omega_Z) \;.
\end{equation}

\noindent\textbf{Step 4.} We conclude.

Taking $\e\rightarrow 0$ on the right hand side of \eqref{eq0tau-pf-thm-tau-bl} and
applying \eqref{eq1-pf-thm-tau-bl}, \eqref{eq2-pf-thm-tau-bl}, \eqref{eq3Y-pf-thm-tau-bl} and \eqref{eq3Z-pf-thm-tau-bl},
we get
\begin{align}
\label{eq41-pf-thm-tau-bl}
\begin{split}
& \tau_d(X,\gamma_X) \\
& = \chi_d(Y,\gamma_Y) \sum_{J\subseteq B} w_d^J
\bigg( \tau_\mathrm{BCOV}(Z_J,\omega_Z) - a_J(\gamma_Z,\omega_Z) - \sum_{j\in J} \frac{m_j+d}{d} b_{J,j}(\omega_Z) \bigg) \;.
\end{split}
\end{align}
On the other hand,
by Definition \ref{def-bcov-inv} and \eqref{eq-def-tau-gamma-omega},
we have
\begin{equation}
\label{eq42-pf-thm-tau-bl}
\tau(Z,\gamma_Z) = \sum_{J\subseteq B} w_d^J
\bigg( \tau_\mathrm{BCOV}(Z_J,\omega_Z) - a_J(\gamma_Z,\omega_Z) - \sum_{j\in J} \frac{m_j+d}{d} b_{J,j}(\omega_Z) \bigg) \;.
\end{equation}
From \eqref{eq41-pf-thm-tau-bl} and \eqref{eq42-pf-thm-tau-bl},
we obtain \eqref{eq-thm-tau-bl}.
This completes the proof.
\end{proof}

\subsection{Proof of Theorem \ref{introthm-tau-bl}}

Now we are ready to prove Theorem \ref{introthm-tau-bl}.

\begin{proof}[Proof of Theorem \ref{introthm-tau-bl}]
The proof consists of several steps.

\noindent\textbf{Step 1.}
Following \cite[\textsection 1.5]{bfm},
we introduce a deformation to the normal cone.

Let $\mathscr{X} \rightarrow X\times\C$ be the blow-up along $Y \times \{0\}$.
Let $\Pi: \mathscr{X} \rightarrow \C$ be the composition of
the canonical projections $\mathscr{X} \rightarrow X\times\C$ and $X\times\C \rightarrow \C$.
For $z\in\C^*$,
we denote
\begin{equation}
X_z = \Pi^{-1}(z) \;.
\end{equation}
Let $\mathbb{1}$ be the trivial line bundle over $Y$.
Recall that $N_Y$ is the normal bundle of $Y \hookrightarrow X$.
Recall that $X'$ is the blow-up of $X$ along $Y$.
The variety $\Pi^{-1}(0)$ consists of two irreducible components:
$\Pi^{-1}(0) = \Sigma_1 \cup \Sigma_2$ with $\Sigma_1 \simeq \mathbb{P}(N_Y\oplus\mathbb{1})$ and $\Sigma_2 \simeq X'$.
We denote
\begin{equation}
X_0 = \Sigma_1 \;.
\end{equation}
For $j=1,\cdots, l$,
let $\mathscr{D}_j\subseteq\mathscr{X}$
be the closure of $D_j\times\C^*\subseteq\mathscr{X}$.
For $z\in\C$,
we denote
\begin{equation}
\label{eq1e1-pf-thm-bl-loc}
D_{j,z} = \mathscr{D}_j \cap X_z \;.
\end{equation}
Let $\mathscr{Y}\subseteq\mathscr{X}$
be the closure of $Y\times\C^*\subseteq\mathscr{X}$.
For $z\in\C$,
we denote
\begin{equation}
Y_z = \mathscr{Y} \cap X_z \;.
\end{equation}

Let $g^{TX}$ be a Hermitian metric on $TX$.
Let $d(\cdot,\cdot): X \times X \rightarrow \R$
be the geodesic distance associated with $g^{TX}$.
For $x\in X$,
we denote
\begin{equation}
d_Y(x) = \inf_{y\in Y} d(x,y) \;.
\end{equation}
For $z\in\C^*$,
set
\begin{equation}
U_z =
\Big\{ x \in X \;:\; d_Y(x)<|z| \Big\} \times \{z\}
\subseteq X_z \;.
\end{equation}
We identify the fiber of $\mathbb{1}$ with $\C$.
For $v\in N_Y$ and $s\in\C$ such that $(v,s)\neq(0,0)$,
we denote by $[v:s]$ the image of $(v,s)$ in $\mathbb{P}(N_Y\oplus\mathbb{1})$.
Let $\big|\cdot\big|$ be the norm on $N_Y$ induced by $g^{TX}$.
Set
\begin{equation}
U_0 =
\Big\{ [v:s] \in \mathbb{P}(N_Y\oplus\mathbb{1}) \;:\; \big|v\big|<|s| \Big\}
\subseteq X_0 \;.
\end{equation}
For $\e>0$ small enough,
we have smooth families
\begin{equation}
\big(U_z\big)_{|z|<\e} \;,\hspace{4mm}
\big(Y_z\big)_{|z|<\e} \;,\hspace{4mm}
\big(U_z\cap D_{j,z}\big)_{|z|<\e} \hspace{4mm}\text{with } j=1,\cdots,l \;.
\end{equation}
We remark that $Y_z \subseteq U_z$ for $z\in\C$.

\begin{figure}[h]
\setlength{\unitlength}{6.5mm}
\centering
\begin{picture}(18,8)
% befor blow-up
\qbezier(0,1.5)(-1,4)(0,6.5) % left X
\qbezier(5,1.5)(4,4)(5,6.5)  % middle X
\qbezier(8,1.5)(7,4)(8,6.5)  % right X
\put(0,6.5){\line(1,0){8}} % top line
\put(-0.4,5){\line(1,0){8}} % middle line
\put(0,1.5){\line(1,0){8}} % bottom line
\multiput(4.5,5)(0.2,0.03){16}{\line(1,0){0.1}} % dash I
\multiput(4.5,5)(0.2,-0.03){15}{\line(1,0){0.1}} % dash II
\multiput(4.5,5)(-0.2,-0.03){26}{\line(1,0){0.1}} % dash III
\multiput(4.5,5)(-0.2,0.03){24}{\line(1,0){0.1}} % dash IV
% notation
\put(0.8,6.8){\vector(0,-1){1.8}} \put(0,7){$Y\times\C$} % Y \times C
\put(5,0.5){\vector(0,1){1}} \put(4.3,0){$z=0$} % z=0
\put(7.85,5.15){\vector(-1,0){1.5}} \put(7.9,4.9){$U_z$} % U_z
\put(-1,0.5){\vector(1,2){1}} \put(-2,0){$X \times \C$} % X \times C
%neighborhood
\put(5.2,5){\qbezier(0,-0.1)(-0.05,0)(0,0.1)}
\put(5.8,5){\qbezier(0,-0.2)(-0.1,0)(0,0.2)}
\put(6.4,5){\qbezier(0,-0.3)(-0.15,0)(0,0.3)}
\put(3.8,5){\qbezier(0,-0.1)(-0.05,0)(0,0.1)}
\put(3.2,5){\qbezier(0,-0.2)(-0.1,0)(0,0.2)}
\put(2.6,5){\qbezier(0,-0.3)(-0.15,0)(0,0.3)}
\put(2,5){\qbezier(0,-0.4)(-0.2,0)(0,0.4)}
\put(1.3,5){\qbezier(0,-0.5)(-0.25,0)(0,0.5)}
\put(0.6,5){\qbezier(0,-0.6)(-0.3,0)(0,0.6)}
% after blow-up
\qbezier(10,1.5)(9,4)(10,6.5) % left X
\qbezier(18,1.5)(17,4)(18,6.5)  % right X
\qbezier(15,1.5)(14,3)(15,4.25)  % X'
\qbezier(15,3.75)(14,5)(15,6.5)  % E
\put(10,6.5){\line(1,0){8}} % top line
\put(9.6,5){\line(1,0){8}} % middle line
\put(10,1.5){\line(1,0){8}} % bottom line
\multiput(9.7,5.4)(0.2,0){40}{\line(1,0){0.1}} % dash top
\multiput(9.6,4.6)(0.2,0){40}{\line(1,0){0.1}} % dash bottom
% notation
\put(10.8,6.8){\vector(0,-1){1.8}} \put(9.5,7){$\mathscr{Y}\simeq Y\times\C$} % Y \times C
\put(14.6,6.8){\vector(0,-1){1}} \put(13.5,7){$\Sigma_1 \simeq \mathbb{P}(N_Y\oplus\mathbb{1})$} % PN
\put(13.5,3){\vector(1,0){1}} \put(11.2,2.9){$\Sigma_2 \simeq X'$} % X'
\put(15,0.5){\vector(0,1){1}} \put(14.3,0){$z=0$} % z=0
\put(17.95,5.15){\vector(-1,0){1.5}} \put(18,4.9){$U_z$} % U_z
\put(9,0.5){\vector(1,2){1}} \put(8.1,0.1){$\mathscr{X}$} % Bl X \times C
% neighborhood
\multiput(15.1,5)(0.7,0){3}{\qbezier(0,-0.4)(-0.17,0)(0,0.4)}
\multiput(14,5)(-0.7,0){6}{\qbezier(0,-0.4)(-0.17,0)(0,0.4)}
\end{picture}
\caption{deformation to the normal cone}
\label{fig-d2nc}
\end{figure}
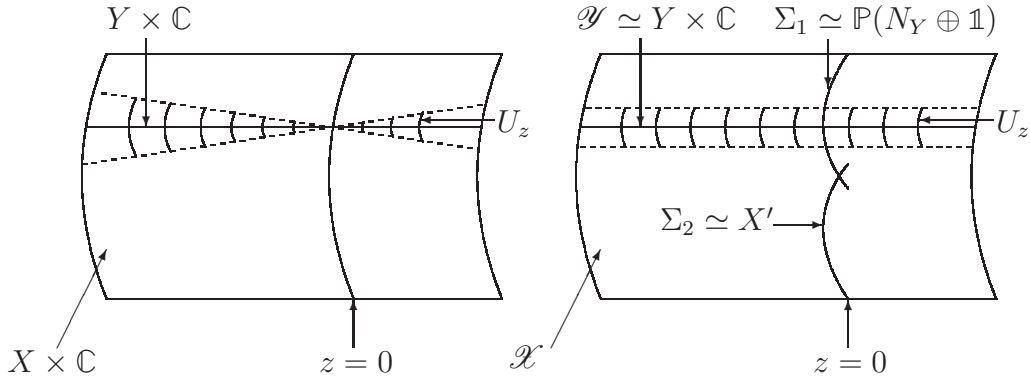

Let $\mathscr{F}: \mathscr{X}' \rightarrow \mathscr{X}$ be the blow-up along $\mathscr{Y}$.
For $z\in\C$,
we denote
\begin{equation}
X_z' = \mathscr{F}^{-1}(X_z) \;.
\end{equation}
Set
\begin{equation}
\label{eq1f1-pf-thm-bl-loc}
f_z = \mathscr{F}\big|_{X_z'} : X_z' \rightarrow X_z \;,
\end{equation}
which is the blow-up along $Y_z$.
For $z\in\C$,
set
\begin{equation}
\label{eq1prim1-pf-thm-bl-loc}
D_{0,z}' = f_z^{-1}(Y_z) \subseteq X_z' \;.
\end{equation}
For $z\in\C$ and $j=1,\cdots,l$,
let $D_{j,z}' \subseteq X_z'$ be the strict transformation of $D_{j,z}\subseteq X_z$.

For $z\in\C$,
set
\begin{equation}
U_z' = f_z^{-1}(U_z) \;.
\end{equation}
For $\e>0$ small enough,
we have smooth families
\begin{equation}
\big(U_z'\big)_{|z|<\e} \;,\hspace{4mm}
\big(U_z'\cap D_{j,z}'\big)_{|z|<\e} \hspace{4mm}\text{with } j=0,\cdots,l \;.
\end{equation}
We remark that $D_{0,z}' \subseteq U_z'$ for $z\in\C$.

\noindent\textbf{Step 2.}
We introduce a family of meromorphic pluricanonical sections.

Denote
\begin{equation}
m = m_1 + \cdots + m_s \;,
\end{equation}
which is the vanishing order of $\gamma$ on $Y$.
Recall that $r$ is the codimension of $Y\hookrightarrow X$.
Recall that $\gamma\in\mathscr{M}(X,K_X^d)$.
For $z\neq 0$,
we identity $X_z$ with $X$ in the obvious way.
For $z\neq 0$,
set
\begin{equation}
\label{eq23-pf-thm-bl-loc}
\gamma_z = z^{-m-rd} \gamma \in \mathscr{M}(X_z,K_{X_z}^d) \;.
\end{equation}
There is a unique $\gamma_0 \in \mathscr{M}(X_0,K_{X_0}^d)$ such that for $\e>0$ small enough,
\begin{equation}
\label{eq24-pf-thm-bl-loc}
\big(\gamma_z\big|_{U_z}\big)_{|z|<\e}
\end{equation}
is a smooth family.
Moreover,
$(X_0,\gamma_0)$ is a $d$-Calabi-Yau pair.

\noindent\textbf{Step 3.}
We introduce a family of K{\"a}hler forms.

Let $\mathscr{U} \subseteq \mathscr{X}$ be such that $\mathscr{U} \cap X_z = U_z$ for any $z\in\C$.
Then $\mathscr{U}$ is an open subset of $\mathscr{X}$.
Set $\mathscr{U}' = \mathscr{F}^{-1}(\mathscr{U}) \subseteq \mathscr{X}'$.
We have $\mathscr{U}' \cap X_z' = U_z'$ for any $z\in\C$.

Let $\omega$ be a K{\"a}hler form on $\mathscr{X}$.
Let $\omega'$ be a K{\"a}hler form on $\mathscr{X}'$ such that
\begin{equation}
\label{eq31-pf-thm-bl-loc}
\omega'\big|_{\mathscr{X}'\backslash\mathscr{U}'} =
\mathscr{F}^* \big( \omega\big|_{\mathscr{X}\backslash\mathscr{U}} \big) \;.
\end{equation}
For $z\in\C$,
set
\begin{equation}
\label{eq32-pf-thm-bl-loc}
\omega_z = \omega\big|_{X_z} \;,\hspace{4mm}
\omega_z' = \omega'\big|_{X_z'} \;.
\end{equation}
By \eqref{eq1f1-pf-thm-bl-loc}, \eqref{eq31-pf-thm-bl-loc} and \eqref{eq32-pf-thm-bl-loc},
we have
\begin{equation}
\label{eq3-pf-thm-bl-loc}
\omega_z'\big|_{X_z'\backslash U_z'} =  f_z^*\big(\omega_z\big|_{X_z\backslash U_z}\big) \hspace{4mm}\text{for } z\in\C \;.
\end{equation}
For $\e>0$ small enough,
we have smooth families
\begin{equation}
\label{eq3a-pf-thm-bl-loc}
\big(\omega_z\big|_{U_z}\big)_{|z|<\e} \;,\hspace{4mm}
\big(\omega_z'\big|_{U_z'}\big)_{|z|<\e} \;.
\end{equation}

\noindent\textbf{Step 4.}
We show that the function $z \mapsto \tau_d(X_z',f_z^*\gamma_z) - \tau_d(X_z,\gamma_z)$ is continuous at $z=0$.

Recall that $f_z: X_z' \rightarrow X_z$ was defined in \eqref{eq1f1-pf-thm-bl-loc}.
Recall that $\gamma_z\in\mathscr{M}(X_z,K_{X_z}^d)$ was defined in the paragraph containing \eqref{eq24-pf-thm-bl-loc}.
Recall that $\big(D_{j,z}\big)_{1\leqslant j\leqslant l}$ were defined in \eqref{eq1e1-pf-thm-bl-loc}.
Recall that $\big(D_{j,z}'\big)_{0\leqslant j\leqslant l}$ were defined in the paragraph containing \eqref{eq1prim1-pf-thm-bl-loc}.
Denote
\begin{equation}
m_0 = m_1 + \cdots + m_s + (r-1)d \;.
\end{equation}
For $z\in\C$,
we have
\begin{equation}
\mathrm{div}(\gamma_z) = \sum_{j=1}^l m_j D_{j,z} \;,\hspace{4mm}
\mathrm{div}(f_z^*\gamma_z) = \sum_{j=0}^l m_j D_{j,z}' \;.
\end{equation}
Here $D_{j,0}$ and $D_{j,0}'$ may be empty for certain $j$.
Let $\big(D_{J,z}\big)_{J\subseteq\{1,\cdots,l\}}$
be as in \eqref{introeq-def-wJ-DJ} with $X$ replaced by $X_z$ and $D_j$ replaced by $D_{j,z}$.
Let $\big(D_{J,z}'\big)_{J\subseteq\{0,\cdots,l\}}$
be as in \eqref{introeq-def-wJ-DJ} with $X$ replaced by $X_z'$ and $D_j$ replaced by $D_{j,z}'$.
By Definition \ref{thm-ind-omega} and \eqref{eq-def-tau-gamma-omega},
we have
\begin{align}
\label{eq4c-pf-thm-bl-loc}
\begin{split}
& \tau_d(X_z',f_z^*\gamma_z) - \tau_d(X_z,\gamma_z) \\
& = \sum_{0\in J\subseteq\{0,\cdots,l\}} w_d^J
\bigg( \tau_\mathrm{BCOV}\big(D_{J,z}',\omega_z'\big)
- a_J(f_z^*\gamma_z,\omega_z') - \sum_{j\in J}\frac{m_j+d}{d} b_J(\omega_z') \bigg) \\
& \hspace{4mm} - \sum_{J\subseteq\{1,\cdots,l\}} w_d^J
\Big( a_J(f_z^*\gamma_z,\omega_z') - a_J(\gamma_z,\omega_z) \Big)  \\
& \hspace{4mm} - \sum_{J\subseteq\{1,\cdots,l\}} \sum_{j\in J} w_d^J
\frac{m_j+d}{d} \Big( b_J(\omega_z') - b_J(\omega_z) \Big) \\
& \hspace{4mm} + \sum_{J\subseteq\{1,\cdots,l\}} w_d^J
\Big( \tau_\mathrm{BCOV}\big(D_{J,z}',\omega_z'\big) - \tau_\mathrm{BCOV}\big(D_{J,z},\omega_z\big) \Big) \;.
\end{split}
\end{align}

For $0\in J\subseteq\{0,\cdots,l\}$,
we have $D_{J,z}' \subseteq U_z'$.
Thus
\begin{equation}
\big(D_{J,z}'\big)_{z\in\C}
\end{equation}
is a smooth family.
Hence the first summation in \eqref{eq4c-pf-thm-bl-loc} is continuous at $z=0$.

For $J\subseteq\{1,\cdots,l\}$,
we denote
\begin{equation}
\label{eq4inex-pf-thm-bl-loc}
D_{J,z} = D_{J,z}^\mathrm{in} \sqcup D_{J,z}^\mathrm{ex}
\end{equation}
such that each irreducible component of $D_{J,z}^\mathrm{in}$ (resp. $D_{J,z}^\mathrm{ex}$) lies in (resp. does not lie in) $Y_z$.
Since $D_{J,z}^\mathrm{in} \subseteq Y_z \subseteq U_z$,
the family
\begin{equation}
\label{eq4Ein-pf-thm-bl-loc}
\big(D_{J,z}^\mathrm{in}\big)_{z\in\C}
\end{equation}
is smooth.
On the other hand,
we have
\begin{equation}
\label{eq4Eex-pf-thm-bl-loc}
D_{J,z}^\mathrm{ex} = f_z\big(D_{J,z}'\big)\;.
\end{equation}
Moreover,
the map $f_z\big|_{D_{J,z}'}: D_{J,z}' \rightarrow D_{J,z}^\mathrm{ex}$ is the blow-up along $D_{J,z}^\mathrm{ex} \cap Y_z$.

Recall that
\begin{equation}
\label{eq4K-pf-thm-bl-loc}
K_J \;,\hspace{4mm}
\gamma_J \;,\hspace{4mm}
g^{TD_J}_\omega \;,\hspace{4mm}
\big|\cdot\big|_{K_J,\omega}
\end{equation}
were constructed in \textsection \ref{subsect-ms} and \textsection \ref{subsect-bcov}
for a $d$-Calabi-Yau pair $(X,\gamma)$ together with a K{\"a}hler form $\omega$ on $X$.
Let
\begin{equation}
K_{J,z} \;,\hspace{4mm}
\gamma_{J,z} \;,\hspace{4mm}
g^{TD_{J,z}}_{\omega_z} \;,\hspace{4mm}
\big|\cdot\big|_{K_{J,z},\omega_z}
\end{equation}
be as in \eqref{eq4K-pf-thm-bl-loc} with $(X,\gamma)$ replaced by $(X_z,\gamma_z)$ and $\omega$ replaced by $\omega_z$.
Let
\begin{equation}
K_{J,z}' \;,\hspace{4mm}
\gamma_{J,z}' \;,\hspace{4mm}
g^{TD_{J,z}'}_{\omega_z'} \;,\hspace{4mm}
\big|\cdot\big|_{K_{J,z}',\omega_z'}
\end{equation}
be as in \eqref{eq4K-pf-thm-bl-loc} with $(X,\gamma)$ replaced by $(X_z',f_z^*\gamma_z)$ and $\omega$ replaced by $\omega_z'$.
By \eqref{eq-def-aJ}, \eqref{eq3-pf-thm-bl-loc}, \eqref{eq4inex-pf-thm-bl-loc} and \eqref{eq4Eex-pf-thm-bl-loc},
for $J\subseteq\{1,\cdots,l\}$,
we have
\begin{align}
\label{eq4s3a-pf-thm-bl-loc}
\begin{split}
& a_J(f_z^*\gamma_z,\omega_z') - a_J(\gamma_z,\omega_z) \\
& = \frac{1}{12} \int_{D_{J,z}' \cap U_z'} c_{n-|J|}\big(TD_{J,z}',g^{TD_{J,z}'}_{\omega_z'}\big)
\log \big|\gamma_{J,z}'\big|^{2/d}_{K_{J,z}',\omega_z'}  \\
& \hspace{4mm} - \frac{1}{12} \int_{D_{J,z}^\mathrm{ex} \cap U_z } c_{n-|J|}\big(TD_{J,z},g^{TD_{J,z}}_{\omega_z}\big)
\log \big|\gamma_{J,z}\big|^{2/d}_{K_{J,z},\omega_z} \\
& \hspace{4mm} - \frac{1}{12} \int_{D_{J,z}^\mathrm{in}} c_{n-|J|}\big(TD_{J,z},g^{TD_{J,z}}_{\omega_z}\big)
\log \big|\gamma_{J,z}\big|^{2/d}_{K_{J,z},\omega_z} \;.
\end{split}
\end{align}
Since each integration in \eqref{eq4s3a-pf-thm-bl-loc} depends continuously on $z$,
the second summation in \eqref{eq4c-pf-thm-bl-loc} is continuous at $z=0$.
The same argument shows that
the third summation in \eqref{eq4c-pf-thm-bl-loc} is continuous at $z=0$.

By \eqref{eq4inex-pf-thm-bl-loc},
we have the obvious identity
\begin{align}
\label{eq4s-pf-thm-bl-loc}
\begin{split}
& \tau_\mathrm{BCOV}\big(D_{J,z}',\omega_z'\big) - \tau_\mathrm{BCOV}\big(D_{J,z},\omega_z\big) \\
& = \tau_\mathrm{BCOV}\big(D_{J,z}',\omega_z'\big) - \tau_\mathrm{BCOV}\big(D_{J,z}^\mathrm{ex},\omega_z\big) - \tau_\mathrm{BCOV}\big(D_{J,z}^\mathrm{in},\omega_z\big) \;.
\end{split}
\end{align}
Since the families in \eqref{eq3a-pf-thm-bl-loc} are smooth,
by Theorem \ref{thm2-bl-bcovt} and \eqref{eq3-pf-thm-bl-loc},
the function $z\mapsto \tau_\mathrm{BCOV}\big(D_{J,z}',\omega_z'\big) - \tau_\mathrm{BCOV}\big(D_{J,z}^\mathrm{ex},\omega_z\big)$ is continuous at $z=0$.
Since the families in \eqref{eq3a-pf-thm-bl-loc} and \eqref{eq4Ein-pf-thm-bl-loc} are smooth,
the function $z\mapsto \tau_\mathrm{BCOV}\big(D_{J,z}^\mathrm{in},\omega_z\big)$ is continuous at $z=0$.
Hence the fourth summation in \eqref{eq4c-pf-thm-bl-loc} is continuous at $z=0$.

\noindent\textbf{Step 5.}
We conclude.

By Step 4,
we have
\begin{equation}
\label{eq51-pf-thm-bl-loc}
\lim_{z\rightarrow 0} \Big( \tau(X_z',f_z^*\gamma_z) - \tau(X_z,\gamma_z) \Big)
= \tau(X_0',f_0^*\gamma_0)-\tau(X_0,\gamma_0) \;.
\end{equation}
On the other hand,
by Proposition \ref{intro-prop-tau-z} and \eqref{eq23-pf-thm-bl-loc},
for $z\neq 0$,
we have
\begin{align}
\label{eq52-pf-thm-bl-loc}
\begin{split}
\tau_d(X_z,\gamma_z) & = \tau_d(X,\gamma) - \frac{\chi_d(X,\gamma)}{12} \log|z|^{-2(m+rd)/d} \;,\\
\tau_d(X_z',f_z^*\gamma_z) & = \tau(X',f^*\gamma) - \frac{\chi_d(X',f^*\gamma)}{12} \log|z|^{-2(m+rd)/d} \;.
\end{split}
\end{align}
Note that $(m+rd)/d>0$,
by \eqref{eq51-pf-thm-bl-loc} and \eqref{eq52-pf-thm-bl-loc},
we have
\begin{align}
\label{eq5a-pf-thm-bl-loc}
\begin{split}
\chi_d(X',f^*\gamma) - \chi_d(X,\gamma) & = 0 \;,\\
\tau_d(X',f^*\gamma) - \tau_d(X,\gamma) & = \tau_d(X_0',f_0^*\gamma_0) - \tau_d(X_0,\gamma_0) \;.
\end{split}
\end{align}

Note that $X_0$ is a $\CP^r$-bundle over $Y_0 \simeq Y$,
by Theorem \ref{thm-tau-bl},
we have
\begin{equation}
\label{eq5b-pf-thm-bl-loc}
\tau_d(X_0,\gamma_0)
= \chi_d(Y,D_Y) \tau_d\big(\CP^r,\gamma_{r,m_1,\cdots,m_s}\big) \;.
\end{equation}
Recall that $E = f^{-1}(Y)$.
Note that $X_0'$ is a $\CP^1$-bundle over $D_{0,0}' \simeq E$,
by Theorem \ref{thm-tau-bl},
we have
\begin{align}
\label{eq5c-pf-thm-bl-loc}
\begin{split}
\tau_d(X_0',f_0^*\gamma_0)
& = \chi_d(E,D_E) \tau_d\big(\CP^1,\gamma_{1,m_0}\big) \;.
\end{split}
\end{align}
From \eqref{eq5a-pf-thm-bl-loc}-\eqref{eq5c-pf-thm-bl-loc},
we obtain \eqref{introeq-thm-tau-bl}.
This completes the proof.
\end{proof}

\bibliographystyle{amsplain}
\bibliography{bcovpluricano}
\end{document}